\documentclass[preprint]{elsarticle}
\input xy
\xyoption{all}
\usepackage{amsmath,amssymb}
\usepackage{graphics}
\usepackage{color}

\newcommand{\letters}{\renewcommand{\theenumi}{\alph{enumi}}}
\newcommand{\leaveout}[1]{#1}

\DeclareMathOperator{\vertex}{\mathrm{vert}}
\DeclareMathOperator{\Hom}{\mathrm{Hom}}

\DeclareMathOperator{\Cone}{\mathrm{Cone}}

\DeclareMathOperator{\id}{\mathrm{id}}
\DeclareMathOperator{\vertices}{\mathrm{vert}}
\DeclareMathOperator{\im}{\mathrm{im}}
\DeclareMathOperator{\face}{\mathrm{face}}
\newcommand{\Cn}{C}

\DeclareMathOperator{\sgn}{\mathrm{sgn}}

\newcommand{\Z}{{\mathbb Z}}

\newcommand{\N}{{\mathbb N}}
\newcommand{\R}{{\mathbb R}}
\newcommand{\Q}{{\mathbb Q}}
\DeclareMathOperator{\rk}{\mathrm{rk}}
\DeclareMathOperator{\QSym}{\mathit{QSym}}
\DeclareMathOperator{\NSym}{\mathit{NSym}}
\DeclareMathOperator{\linhull}{\mathrm{lhull}}

\newtheorem{theorem}{Theorem}[section]

\newtheorem{proposition}[theorem]{Proposition}
\newtheorem{lemma}[theorem]{Lemma}
\newtheorem{corollary}[theorem]{Corollary}
\newdefinition{definition}[theorem]{Definition}
\newdefinition{remark}[theorem]{Remark}
\newdefinition{example}[theorem]{Example}
\newproof{proof}{Proof}
\newproof{proof14d}{Proof of Theorem~\ref{theo1.4}(d)}
\newproof{proof14c}{Proof of Theorem~\ref{theo1.4}(c)}
\newproof{proof11}{Proof of Theorem~\ref{theo:1.1}}
\newproof{proof14b}{Proof of Theorem~\ref{theo1.4}(b)}
\newproof{proof14a}{Proof of Theorem~\ref{theo1.4}(a)}
\newproof{proof15d}{Proof of Theorem~\ref{theo1.5}(d)}
\newproof{proof15c}{Proof of Theorem~\ref{theo1.5}(c)}
\newproof{proof15ab}{Proof of Theorem~\ref{theo1.5}(a),(b)}
\newproof{prooffreeab}{Proof of Theorem~\ref{theo:free}(a),(b)}
\newproof{prooffreecd}{Proof of Theorem~\ref{theo:free}(c),(d)}
\newproof{proofws}{Proof of Theorem~\ref{theoweakstrong}}

\begin{document}

\title{Valuative invariants for polymatroids}
\author[um]{Harm Derksen\fnref{fn1}} 
\ead{hderksen@umich.edu}
\author[ucb]{Alex Fink\corref{cr}} 
\ead{finka@math.berkeley.edu}
\fntext[fn1]{Supported by NSF grant DMS 0349019 and DMS 0901298.}
\address[um]{Department of Mathematics, Univeristy of Michigan} 
\address[ucb]{Department of Mathematics, University of California, Berkeley}
\cortext[cr]{Corresponding author} 

\begin{abstract}
Many important invariants for matroids and polymatroids,
such as the Tutte polynomial, the Billera-Jia-Reiner quasi-symmetric function, and
the invariant $\mathcal G$ introduced by the first author, are valuative.
In this paper we construct the $\Z$-modules of all $\Z$-valued 
valuative functions for labeled matroids and polymatroids on a fixed ground set,
and their unlabeled counterparts, the $\Z$-modules of valuative invariants.
We give explicit bases for these modules
and for their dual modules generated by indicator functions of polytopes,
and explicit formulas for their ranks.
Our results confirm a conjecture of the first author that $\mathcal G$
is universal for valuative invariants.
\end{abstract}


\maketitle

\section{Introduction}
Matroids were introduced by Whitney in 1935 (see~\cite{Whitney}) as a combinatorial abstraction of linear dependence of vectors in a vector space. 
Some standard references are \cite{Welch} and~\cite{Oxley}. 
Polymatroids are  multiset analogs of matroids and appeared in the late 1960s 
(see~\cite{Edmonds,HH}). 
There are many distinct but equivalent definitions
of matroids and polymatroids, for example in terms of bases, independent sets, flats, 
polytopes or rank functions. 
For polymatroids, the equivalence
between the various definitions is given in \cite{HH}.  
We will stick to the definition in terms of rank functions: 
\begin{definition}
Suppose that $X$ is a finite set (the {\em ground set}) 
and $\rk:2^X\to \N=\{0,1,2,\dots\}$, where $2^X$ is the set of subsets of $X$.
Then $(X,\rk)$ is called a {\em polymatroid} if:
\begin{enumerate}
\item $\rk(\emptyset)=0$;
\item $\rk$ is {\em weakly increasing}: if $A\subseteq B$ then $\rk(A)\leq \rk(B)$;
\item $\rk$ is {\em submodular}: $\rk(A\cup B)+\rk(A\cap B)\leq \rk(A)+\rk(B)$ for all $A,B\subseteq X$.
\end{enumerate}
If moreover, $\rk(\{x\})\leq 1$ for every $x\in X$, then $(X,\rk)$ is called a {\em matroid}.
\end{definition}
An {\em isomorphism} $\varphi:(X,\rk_X)\to (Y,\rk_Y)$
is a bijection $\varphi:X\to Y$ such that $\rk_Y\circ \varphi=\rk_X$. Every polymatroid is isomorphic
to a polymatroid with ground set $\underline{d}=\{1,2,\dots,d\}$ for some nonnegative integer $d$.
The rank of a polymatroid $(X,\rk)$ is $\rk(X)$.

Let $S_{\rm PM}(d,r)$ be the set of all polymatroids with ground set $\underline{d}$ of rank $r$, 
and $S_{\rm M}(d,r)$ be the set of all matroids with ground set $\underline{d}$ of rank $r$.
We will write $S_{\rm (P)M}(d,r)$\label{not:SPM}
when we want to refer to $S_{\rm PM}(d,r)$ or $S_{\rm M}(d,r)$ in parallel. 
A function $f$ on $S_{\rm (P)M}(d,r)$ is a (poly)matroid {\em invariant} if
$f\big((\underline{d},\rk)\big)=f\big((\underline{d},\rk')\big)$ whenever
$(\underline{d},\rk)$ and $(\underline{d},\rk')$ are isomorphic. Let $S_{\rm (P)M}^{\rm sym}(d,r)$\label{not:SPMsym}
be the set of isomorphism classes in $S_{\rm (P)M}(d,r)$. Invariant functions on $S_{\rm (P)M}(d,r)$
correspond to functions on $S_{\rm (P)M}^{\rm sym}(d,r)$.
Let $Z_{\rm (P)M}(d,r)$\label{not:ZPM} and $Z_{\rm (P)M}^{\rm sym}(d,r)$\label{not:ZPMsym} be the $\Z$-modules
freely generated by $S_{\rm (P)M}(d,r)$ and $S_{\rm (P)M}^{\rm sym}(d,r)$ respectively.
For an abelian group $A$, every function $f:S_{\rm (P)M}^{\rm (sym)}(d,r)\to A$ extends
uniquely to a group homomorphism $Z^{\rm (sym)}_{\rm (P)M}(d,r)\to A$.

One of the most important matroid invariants is the Tutte polynomial. It was first defined for graphs in \cite{Tutte} and
generalized to matroids in \cite{Brilawski,Crapo2}.
This bivariate polynomial is defined by\footnote{Regarded as a polynomial in $x-1$
and $y-1$, $\mathcal T$ is known as the {\em rank generating function}.}\label{not:T}
$$
{\mathcal T}\big((X,\rk)\big)=\sum_{A\subseteq X} (x-1)^{\rk(X)-\rk(A)}(y-1)^{|A|-\rk(A)}.
$$
The Tutte polynomial is universal for all matroid invariants satisfying a deletion-contraction formula.
Speyer defined a matroid invariant in \cite{Speyer} using $K$-theory. 
Billera, Jia and Reiner introduced a quasi-symmetric function ${\mathcal F}$\label{not:F} for matroids in \cite{BJR},
which is a matroid invariant. This quasi-symmetric function is a powerful invariant in the sense
that it can distinguish many pairs of non-isomorphic matroids. However, it does not specialize
to the Tutte polynomial. The first author introduced in \cite{Derksen} another quasi-symmetric function ${\mathcal G}$.
For some choice of  basis $\{U_{\alpha}\}$\label{not:Ualpha} of the ring of quasi-symmetric functions, 
${\mathcal G}$\label{not:G} is defined by
$$
{\mathcal G}\big((X,\rk)\big)=\sum_{\underline{X}}U_{r(\underline{X})},
$$
where
$$
\underline{X}:\emptyset=X_0\subset X_1\subset \cdots \subset X_d=X
$$
runs over all $d!$ maximal chains of subsets in $X$,
and
$$
r(\underline{X})=(\rk(X_1)-\rk(X_0),\rk(X_2)-\rk(X_1),\dots,\rk(X_d)-\rk(X_{d-1})).
$$
It was already shown in \cite{Derksen} that ${\mathcal G}$ specializes to
${\mathcal T}$ and ${\mathcal F}$.

To a (poly)matroid $(\underline{d},\rk)$ one can associate its 
{\em base polytope} $Q(\rk)$ in $\R^d$ (see Definition~\ref{def:basepolytope}).
For $d\geq 1$, the dimension of this polytope is $\leq d-1$.
The indicator function of a polytope $\Pi\subseteq \R^d$ is denoted by $[\Pi]:\R^d\to \Z$\label{not:[]}.
Let $P_{\rm (P)M}(d,r)$\label{not:PPM} be the $\Z$-module generated by all $[Q(\rk)]$
with $(\underline{d},\rk)\in S_{\rm (P)M}(d,r)$.
\begin{definition} Suppose that $A$ is an abelian group. A function $f:S_{\rm (P)M}(d,r)\to A$
is  {\em strongly valuative} if there exists a group homomorphism $\widehat{f}:P_{\rm (P)M}(d,r)\to A$
such that 
$$f\big((\underline{d},\rk)\big)=\widehat{f}([Q(\rk)])
$$
for all $(\underline{d},\rk)\in S_{\rm (P)M}(d,r)$.
\end{definition}
In Section~\ref{sec:valuative} we also define a {\em weak valuative property} in terms of base polytope decompositions. 
Although seemingly weaker, we will show that the weak valuative property is equivalent to the strong valuative property.
\begin{definition}
Suppose that $d>0$.
A valuative function $f:S_{\rm (P)M}(d,r)\to A$ is said to be {\em additive}, if
$f\big((\underline{d},\rk)\big)=0$ whenever the dimension of $Q(\rk)$ is $<d-1$.
\end{definition}

Most of the known (poly)matroid invariants are {\em valuative}. For example, ${\mathcal T}$, ${\mathcal F}$ and ${\mathcal G}$ all have this property in common.
Speyer's invariant is not valuative, but does have a similar
property, which we will call the {\em covaluative}  property. 
Valuative invariants and additive invariants can be useful for deciding
whether a given matroid polytope has a decomposition into smaller matroid polytopes
(see the discussion in \cite[Section 7]{BJR}). 
Decompositions of polytopes and their valuations are fundamental objects 
of interest in discrete geometry in their own right (see for instance
the survey~\cite{MS}).  
Matroid polytope decompositions
appeared in the work of Lafforgue (\cite{Laf1,Laf2}) on compactifications
of a fine Schubert cell in the Grassmannian associated to a matroid.
The work of Lafforgue implies that if the base polytope of a matroid
does not have a proper decomposition, then the matroid is rigid, i.e.,
it has only finitely many nonisomorphic realizations over a given field.

\subsection*{Main results}
The following
theorem proves a conjecture of the first author in \cite{Derksen}:
\begin{theorem}\label{theo:1.1}
The ${\mathcal G}$-invariant is  universal for all valuative (poly)ma\-troid invariants, i.e.,
the coefficients of ${\mathcal G}$ span the vector space of all valuative (poly)matroid invariants
with values in $\Q$.
\end{theorem}
 From ${\mathcal G}$ one
can also construct a universal invariant for the covaluative property
which specializes to Speyer's invariant.

It follows from the definitions that the dual $P_{\rm (P)M}(d,r)^\vee=\Hom_\Z(P_{\rm (P)M}(d,r),\Z)$\label{not:Vvee}
is the space of all $\Z$-valued valuative functions on $S_{\rm (P)M}(d,r)$.
If $P_{\rm (P)M}^{\rm sym}(d,r)$\label{not:PPMsym} is the push-out of the diagram
\begin{equation}\label{eq:pushout}
\xymatrix{
Z_{\rm (P)M}(d,r)\ar[r]^{\pi_{\rm (P)M}}\ar[d]_{\Psi_{\rm (P)M}} & Z_{\rm (P)M}^{\rm sym}(d,r)\ar@{.>}^{\Psi_{\rm (P)M}^{\rm sym}}[d]\\
P_{\rm (P)M}(d,r)\ar@{.>}[r]_{\rho_{\rm (P)M}} & P_{\rm (P)M}^{\rm sym}(d,r)}
\end{equation}
then the dual space $P_{\rm (P)M}^{\rm sym}(d,r)^\vee$
is exactly the set of all $\Z$-valued valuative (poly)ma\-troid {\em invariants}.
Let $p_{\rm (P)M}^{\rm sym}(d,r)$\label{not:ppmsym} be the rank of $P_{\rm (P)M}^{\rm sym}(d,r)$,
and $p_{\rm (P)M}(d,r)$\label{not:ppm} be the rank of $P_{\rm (P)M}(d,r)$. Then $p_{\rm (P)M}^{\rm sym}(d,r)$
is the number of independent $\Z$-valued valuative (poly)matroid invariants,
and $p_{\rm (P)M}(d,r)$ is the number of independent $\Z$-valued valuative functions
on (poly)matroids.
We will prove the following formulas:
\begin{theorem}\label{theo1.4}\ 

\begin{enumerate}\letters
\item $\displaystyle p_{\rm M}^{\rm sym}(d,r)=\textstyle {d\choose r}$ and $\displaystyle \sum_{0\leq r\leq d}p_{\rm M}^{\rm sym}(d,r)x^{d-r}y^r=\frac{1}{1-x-y}$,
\item
$\displaystyle p_{\rm PM}^{\rm sym}(d,r)=\left\{
\begin{array}{ll}
\textstyle {r+d-1\choose r} & \mbox{if $d\geq 1$ or $r\geq 1$;}\\
1 & \mbox{if $d=r=0$}\end{array}\right.$
and\newline
 $\displaystyle \sum_{r=0}^\infty\sum_{d=0}^{\infty}p_{\rm PM}^{\rm sym}(d,r)x^dy^r=\frac{1-x}{1-x-y}$,
\item $\displaystyle\sum_{0\leq r\leq d} \frac{p_{\rm M}(d,r)}{d!}x^{d-r}y^r=
\frac{x-y}{xe^{-x}-ye^{-y}}$,
\item 
$\displaystyle p_{\rm PM}(d,r)=\left\{
\begin{array}{ll}
(r+1)^d-r^d & \mbox{if $d\geq 1$ or $r\geq 1$;}\\
1 & \mbox{if $d=r=0$},\end{array}\right.$ and\newline
$\displaystyle \sum_{d=0}^\infty\sum_{r=0}^\infty \frac{p_{\rm PM}(d,r)x^dy^r}{d!}=
\frac{e^x(1-y)}{1-ye^x}.$
\end{enumerate}
\end{theorem}
We also will give {\bf explicit bases} for each of the spaces $P_{\rm (P)M}(d,r)$ and $P_{\rm (P)M}^{\rm sym}(d,r)$
and their duals (see Theorems~\ref{free generation}, \ref{theo:Psymgens}, Corollaries~\ref{cor s univ}, \ref{cor maximal chains}, \ref{cor:PMsymdual}, \ref{cor:PPMsymdual}).

The bigraded module
$$Z_{\rm (P)M}=\bigoplus_{d,r}Z_{\rm (P)M}(d,r)$$
has the structure of a Hopf algebra. Similarly, each of the bigraded modules
$Z_{\rm (P)M}^{\rm sym}$, $P_{\rm (P)M}$ and $P_{\rm (P)M}^{\rm sym}$ has a Hopf algebra structure.
The module $Z_{\rm (P)M}^{\rm sym}$ is the usual Hopf algebra of (poly)matroids, where multiplication is
given by the direct sum of matroids. 

In Sections~\ref{sec:T} and~\ref{sec:Tsym} we construct bigraded modules $T_{\rm (P)M}$\label{not:TPM} and $T_{\rm (P)M}^{\rm sym}$\label{not:TPMsym} such
that $T_{\rm (P)M}(d,r)^\vee$ is the space of all additive functions on $S_{\rm (P)M}(d,r)$
and $T_{\rm (P)M}^{\rm sym}(d,r)^\vee$ is the space of all additive invariants.
Let $t_{\rm (P)M}(d,r)$\label{not:tPM} be the rank of $T_{(P)M}(d,r)$ and $t_{(P)M}^{\rm sym}(d,r)$\label{not:tPMsym}
be the rank of $T_{(P)M}^{\rm sym}(d,r)$. Then $t_{\rm (P)M}(d,r)$ is the number of independent additive functions
on (poly)matroids, and $t_{\rm (P)M}^{\rm sym}(d,r)$ is the number of independent additive invariants
for (poly)matroids. 
We will prove the following formulas:
\begin{theorem}\label{theo1.5}\ 

\begin{enumerate}
\letters
\item
$\displaystyle
\prod_{0\leq r\leq d} (1-x^{d-r}y^r)^{t_{\rm M}^{\rm sym}(d,r)}=1-x-y$,
\item
$\displaystyle
\prod_{r,d} (1-x^dy^r)^{t_{\rm PM}^{\rm sym}(d,r)}=\frac{1-x-y}{1-y}$,
\item
$\displaystyle
\sum_{r,d}\frac{t_{\rm M}(d,r)}{d!}x^{d-r}y^r=\log\left(\frac{x-y}{xe^{-x}-ye^{-y}}\right)$,
\item 
$\displaystyle
t_{\rm PM}(d,r)=\left\{
\begin{array}{ll}
r^{d-1} & \mbox{if $d\geq 1$}\\
0 & \mbox{if $d=0$,}
\end{array}\right.$
and \newline
$\displaystyle\sum_{r,d}\frac{t_{\rm PM}(d,r)}{d!}x^dy^r=\log\big(\frac{e^x(1-y)}{1-ye^x}\big)$.
\end{enumerate}
\end{theorem}
We will also give {\bf explicit bases} for the the spaces $T_{\rm M}(d,r)$ and $T_{\rm PM}(d,r)$ in Theorem~\ref{theo:TPMgens},
and of the dual spaces $T_{\rm M}^{\rm sym}(d,r)^\vee\otimes_\Z\Q$, $T_{\rm PM}^{\rm sym}(d,r)^\vee\otimes_\Z\Q$
 in Theorem~\ref{theo:Tdual}.

For $\Q$-valued functions we will prove the following isomorphisms in Section~\ref{sec:free algebras}.
\begin{theorem}\label{theo:free}
Let $u_0,u_1,u_2,\dots$ be indeterminates, where $u_i$ has bidgree $(1,i)$.
We have the following isomorphisms of bigraded associative algebras over $\Q$:
\begin{enumerate}
\letters
\item
The space $(P^{\rm sym}_{M})^\vee\otimes_\Z \Q$ of $\Q$-valued valuative invariants
on matroids is isomorphic to $\Q\langle\langle u_0,u_1\rangle\rangle$, the completion (in power series) 
of the free associative algebra generated by $u_0,u_1$.
\item The space $(P^{\rm sym}_{\rm PM})^\vee \otimes_\Z \Q$ of $\Q$-valued valuative invariants
on polymatroids is isomorphic to $\Q\langle\langle u_0,u_1,u_2,\dots\rangle\rangle$.
\item The space $(T^{\rm sym}_{\rm M})^\vee
\otimes_\Z\Q$ of $\Q$-valued additive invariants
on matroids is isomorphic to $\Q\{\{u_0,u_1\}\}$, the completion of the free
Lie algebra generated by $u_0,u_1$.
\item The space $(T^{\rm sym}_{\rm PM})^\vee \otimes_\Z\Q$ of $\Q$-valued additive invariants
on polymatroids is isomorphic to $\Q\{\{u_0,u_1,u_2,\dots\}\}$.
\end{enumerate}
\end{theorem}
Tables for $p_{\rm (P)M}$, $p_{\rm (P)M}^{\rm sym}$, $t_{\rm (P)M}$, $t_{\rm (P)M}^{\rm sym}$
are given in~\ref{apB}.

An index of notations used in this paper appears on page~\pageref{notindex}.
To aid the reader in keeping them in mind we present an abridged table here.
In a notation of the schematic form 
$\mbox{Letter}_{\scriptsize\mbox{sub}}^{\scriptsize\mbox{super}}(d,r)$ :

\vspace{1ex}\noindent\begin{tabular}{llll}
The letter & $S$ & refers to & the set of *-matroids \\
& $Z$ & & the $\Z$-module with basis all *-matroids \\
& $P$ & & the $\Z$-module of indicator functions of *-matroids \\
& $T$ & & the $\Z$-module of indicator functions of *-matroids, \\
& & & modulo changes on subspaces of dimension $<d-1$ \\
\end{tabular}\\
with ground set $\underline d$ of rank $r$.  If the letter is lowercase,
we refer not to the $\Z$-module but to its rank.

\vspace{1ex}\noindent\begin{tabular}{llll}
The subscript & M & means the *-matroids are & matroids \\
& PM & & polymatroids \\
& MM & & megamatroids (Def.~\ref{def:megamatroid}); \\
\end{tabular}\\
additionally, when we want to refer to multiple cases in parallel,

\noindent\begin{tabular}{llll}
the subscript & (P)M & covers & matroids and polymatroids \\
& *M & & matroids and poly- and mega-matroids. \\
\end{tabular}\\

The superscript sym means that we are only considering *-matroids up to isomorphism.

\section{Polymatroids and their polytopes}
For technical reasons it will be convenient to have an
``unbounded'' analogue of polymatroids, especially when we work with their polyhedra.
So we make the following definition.
\begin{definition}\label{def:megamatroid}
A function $2^X\to \Z\cup \{\infty\}$ is called a {\em megamatroid}\footnote{A more appropiate terminology
would be {\em apeiromatroid}, but apeiromatroid simply does not sound as good as megamatroid.} 
if it has the following properties:
\begin{enumerate}
\item $\rk(\emptyset)=0$;
\item $\rk(X)\in \Z$;
\item $\rk$ is {\em submodular}: if $\rk(A),\rk(B)\in \Z$, then $\rk(A\cup B),\rk(A\cap B)\in \Z$ and
$\rk(A\cup B)+\rk(A\cap B)\leq \rk(A)+\rk(B).$
\end{enumerate}
\end{definition} 
Obviously, every matroid is a polymatroid, and every polymatroid is a megamatroid.
The {\em rank} of a megamatroid $(X,\rk)$ is the integer $\rk(X)$. 

By a {\em polyhedron} we will mean a finite intersection of closed half-spaces. 
A {\em polytope} is a bounded polyhedron.
\label{not:Q}
\begin{definition}\label{def:basepolytope}
For a megamatroid $(\underline{d},\rk)$, we define its {\em base polyhedron} $Q(\rk)$ as the set of all $(y_1,\dots,y_d)\in \R^d$
such that $y_1+y_2+\cdots+y_d=\rk(X)$ and $\sum_{i\in A}y_i\leq \rk(A)$ for all $A\subseteq X$.
\end{definition}
If $\rk$ is a polymatroid then $Q(\rk)$ is a polytope, called the {\em base polytope} of $\rk$. 
In \cite{Edmonds}, Edmonds studies a similar polytope for a polymatroid $(\underline{d},\rk)$ which contains $Q(\rk)$ as a facet. 

\begin{lemma}\label{lemnonempty}
If $(\underline{d},\rk)$ is a megamatroid, then $Q(\rk)$ is nonempty.
\end{lemma}
\begin{proof}
First, assume that $\rk$ is a megamatroid such that
$r_i:=\rk(\underline i)$ is finite for $i=0,1,\dots,d$.
We claim that
$$
y=(r_1-r_0,r_2-r_1,\dots,r_d-r_{d-1})\in Q(\rk).
$$
Indeed, if $A=\{i_1,\dots,i_k\}$ with $1\leq i_1<\cdots<i_k\leq d$
 then, by the submodular property, we have
\begin{multline*}
\sum_{i\in A}y_i=\sum_{j=1}^k \rk(\underline{i_j})-\rk(\underline{i_{j-1}})\leq\\
\leq 
\sum_{j=1}^k\rk(\{i_1,\dots,i_j\})-\rk(\{i_1,\dots,i_{j-1}\})=\rk(\{i_1,\dots,i_k\})=\rk(A).
\end{multline*}
where the inequality holds even if the right hand side is infinite.

Now, assume that $\rk$ is any megamatroid.
Define $\rk^N$ by
\begin{equation}\label{eq rk^N}
\rk^N(A)=\min_{X\subseteq A}\rk(X)+N(|A|-|X|).
\end{equation}
 Let $N$ be large enough such that $\rk^N(\underline{d})=\rk(\underline{d})$.
 If $A,B\subseteq \underline{d}$, then we have
 $$
 \rk^N(A)=\rk(X)+N(|A|-|X|),\quad \rk^N(B)=\rk(Y)+N(|A|-|Y|)
 $$
 for some $X\subseteq A$ and some $Y\subseteq B$.
 It follows that
 \begin{multline*}
 \rk^N(A\cap B)+\rk^N(A\cup B) \\
 \leq\rk(X\cap Y)+N(|A\cap B|-|X\cap Y|)+\rk(X\cup Y)+N(|A\cup B|-|X\cup Y|)\\
 =\rk(X\cap Y)+\rk(X\cup Y)+N(|A|+|B|-|X|-|Y|) \\
 \leq \rk(X)+\rk(Y)+N(|A|+|B|-|X|-|Y|)=\rk^N(A)+\rk^N(B).
 \end{multline*}
 This shows that $\rk^N$ is a megamatroid.
 Since $\rk^N(A)\leq \rk(A)$ for all $A\subseteq \underline{d}$,
 we have $Q(\rk^N)\subseteq Q(\rk)$. 
 Since $\rk^N(A)<\infty$ for all $A\subseteq \underline{d}$, we have that $Q(\rk^N)\neq \emptyset$.
 We conclude that $Q(\rk)\neq \emptyset$.
 \end{proof}

A megamatroid $(\underline{d},\rk)$ of rank $r$ is a polymatroid 
if and only if its base polytope is contained in the simplex
$$
\Delta_{\rm PM}(d,r)=\{(y_1,\dots,y_d)\in \R^d\mid y_1,\dots,y_d\geq 0,\ y_1+y_2+\cdots+y_d=r\}
$$ \label{not:delta}
and it is a matroid if and only if its base polytope is contained in the hypersimplex
$$
\Delta_{\rm M}(d,r)=\{(y_1,\dots,y_d)\in \R^d\mid 0\leq y_1,\dots,y_d\leq 1,\ y_1+y_2+\cdots+y_d=r\}.
$$

If $(\underline{d},\rk)$ is a matroid, then a subset $A\subseteq \underline{d}$ is
a {\em basis} when $\rk(A)=|A|=\rk(\underline d)$. 
In this case, the base polytope of $(\underline{d},\rk)$ is
the convex hull of all $\sum_{i\in A}e_i$ where $A\subseteq \underline{d}$ is a basis (see~\cite{GGMS}).
The base polytope of a matroid was characterized in \cite{GGMS}:
\begin{theorem}
A polytope $\Pi$ contained in $\Delta_{\rm M}(d,r)$ is the base polytope of a matroid if and only if it
has the following properties:
\begin{enumerate}
\item The vertices of $\Pi$ have integral coordinates;
\item every edge of $\Pi$ is parallel to $e_i-e_j$ for some $i,j$ with $i\neq j$.
\end{enumerate}
\end{theorem}
We will generalize this characterization to megamatroids.
\begin{definition}
A convex polyhedron contained in $y_1+\cdots+y_d=r$ is called a {\em megamatroid polyhedron} if
for every face $F$ of $\Pi$,  the linear hull $\linhull(F)$ is of the form $z+W$\label{not:linhull}
where $z\in \Z^d$ and $W$ is spanned by vectors of the form $e_i-e_j$.
\end{definition}

The bounded megamatroid polyhedra are exactly the lattice polytopes
among the {\em generalized permutohedra} of~\cite{Postnikov}
or the {\em submodular rank tests} of~\cite{MPSSW}.
General megamatroid polyhedra are the natural unbounded generalizations.

Faces of megamatroid polyhedra are again megamatroid polyhedra.
If we intersect a megamatroid polyhedron $\Pi$ with
the hyperplane $y_d=s$, we get again a megamatroid polyhedron.
For a megamatroid polyhedron $\Pi$, define $\rk_\Pi:2^{\underline{d}}\to\Z\cup \{\infty\}$ by
$$
\rk_{\Pi}(A):=\sup\{\textstyle\sum_{i\in A}y_i\mid y\in \Pi\}.
$$ \label{not:rkPi}
\begin{lemma}\label{lemrankequal}
Suppose that $\Pi$ is a megamatroid polyhedron, $A\subseteq B$ and $\rk_{\Pi}(A)<\infty$.
Let $F$ be the face of $\Pi$ on which $\sum_{i\in A}y_i$ is maximal. Then
$$
\rk_{\Pi}(B)=\rk_{F}(B).
$$
\end{lemma}
\begin{proof}
If $\rk_F(B)=\infty$ then $\rk_{\Pi}(B)=\infty$ and we are done.
Otherwise, there exists a face $F'$ of $F$ on which $\sum_{i\in B}y_i$ is maximal.
Suppose that $\rk_{F}(B)<\rk_\Pi(B)$. Define $g(y):=\sum_{i\in B}y_i-\rk_F(B)$.
Then $g$ is constant $0$ on $F'$, and $g(y)>0$ for some $y\in \Pi$.
Therefore, there exists a face $F''$ of $\Pi$ containing $F'$, such that
$\dim F''=\dim F'+1$ and $g(z)>0$ for some $z\in F''$. Clearly, $z\not\in F$ and $F$
does not contain $F''$.
 We have $\linhull(F'')=\linhull(F')+\R(e_k-e_j)$ for some $k\neq j$.
By possibly exchanging $j$ and $k$, we may assume that
$F''$ is contained in $\linhull(F')+\R_+(e_k-e_j)$, where $\R_+$ denotes
the nonnegative real numbers. Since $z\in \linhull(F')+\R_+(e_k-e_j)$ and $g(z)>0$
we have $k\in B$ and $j\not \in B$.
In particular $j\not \in A$, which means that $\sum_{i\in A}y_i\geq \rk_\Pi(A)$
for all $y\in F''$, so $F''\subseteq F$. This is a contradiction.
We conclude that $\rk_F(B)=\rk_\Pi(B)$.
\end{proof}

\begin{lemma}\label{lemsupequal}
Suppose that $f(y)=\sum_{j=1}^d \alpha_j\sum_{i\in X_j}y_i$ where
$$
\underline{X}:\emptyset\subset X_1\subset X_2\subset \cdots \subset X_d=\underline{d}
$$
is a maximal chain, and $\alpha_1.\dots,\alpha_{d-1}\geq 0$. 
For a megamatroid polyhedron $\Pi$ we have
$$
\sup_{y\in \Pi}f(y)= \sum_{j=1}^d \alpha_j\rk_{\Pi}(X_j).
$$
\end{lemma}
\begin{proof}
First, assume that $\Pi$ is bounded. Define $F_0=\Pi$, and for $j=1,2,\dots,d$, let
$F_j$ be the face of $F_{j-1}$ for which $\sum_{i\in X_j}y_i$ is maximal.
By induction on $j$ and Lemma~\ref{lemrankequal}, we have that $\rk_{F_j}(X_i)=\rk_\Pi(X_i)$
for all $j<i$. Also, $F_j$ is contained in the hyperplane
defined by the equation $\sum_{i\in X_j}y_i=\rk_{F_{j-1}}(X_j)=\rk_{\Pi}(X_j)$.
We have $F_d=\{z\}$ where $z=(z_1,\dots,z_d)$ is defined by the equations
$$
\sum_{i\in X_j}z_i=\rk_{\Pi}(X_j),\quad j=1,2,\dots,d.
$$
It follows that
$$
f(z)=\sum_{j=1}^d \alpha_j\sum_{i\in X_j}z_j=\sum_{j=1}^d\alpha_j\rk_{\Pi}(X_j).
$$

Suppose that  $\Pi$ is unbounded. Let $\Pi_N$ be the intersection
of $\Pi$ with the set $\{y\in \R^d\mid y_i\leq N,\ i=1,2,\dots,d\}$. Now $\Pi_N$ is a bounded megamatroid polyhedron
for large positive integers $N$. (For small $N$, $\Pi_N$ might be empty.)
We have
$$
\sup_{y\in \Pi}f(y)=\sup_N \sup_{y\in \Pi_N}f(y)=\sup_N\sum_{j=1}^d \alpha_j\rk_{\Pi_N}(X_j)=
\sum_{j=1}^d\alpha_j\rk_{\Pi}(X_j).
$$
\end{proof}
\begin{corollary}\label{cormegamatroid}
If $\Pi$ is a megamatroid polyhedron, then $\rk_{\Pi}$ is a megamatroid.
\end{corollary}
\begin{proof}
For subsets $A,B\subseteq\underline d$, 
choose a maximal chain $\underline{X}$ such that $X_j=A\cap B$ and $X_k=A\cup B$ for some $j$ and $k$, and let
$$
f_A(y) = \sum_{i\in A}y_i,\quad f_B(y) = \sum_{i\in B} y_i,\quad
f(y)=\sum_{i\in A\cap B} y_i+\sum_{i\in A\cup B}y_i = f_A(y) + f_B(y).
$$
By Lemma~\ref{lemsupequal}, 
\begin{multline*}
\rk_{\Pi}(A)+\rk_{\Pi}(B)=\sup_{y\in\Pi}f_A(y)+\sup_{y\in\Pi}f_B(y)
\\\geq \sup_{y\in \Pi}f(y)= \rk_{\Pi}(A\cap B)+\rk_{\Pi}(A\cup B).
\end{multline*}
\end{proof}
\begin{proposition}
A convex polyhedron $\Pi$ in the hypersurface $y_1+y_2+\cdots+y_d=r$
is a megamatroid polyhedron if and only if $\Pi=Q(\rk)$ for some megamatroid $\rk$.
\end{proposition}
\begin{proof}
Suppose that $\Pi$ is a megamatroid polyhedron. Then $\rk_\Pi$ is a megamatroid by Corollary~\ref{cormegamatroid}.
Clearly we have $\Pi\subseteq Q(\rk_{\Pi})$. Suppose that $f(y)=\sum_{i=1}^d\alpha_i y_i$ is a linear
function on the hypersurface $y_1+\cdots+y_d=r$.
Let $\sigma$ be a permutation of $\underline{d}$ such that
$\alpha_{\sigma(i)}\geq \alpha_{\sigma(j)}$ for $i<j$.
Define $X_k=\{\sigma(1),\dots,\sigma(k)\}$ for $k=1,2,\dots,d$.
We can write
$$
f(y)=\sum_{j=1}^d \beta_j \sum_{i\in X_j} y_i,
$$
where $\beta_j:=\alpha_{\sigma(j)}-\alpha_{\sigma(j+1)}\geq 0$ for $j=1,2,\dots,d-1$
and $\beta_d=\alpha_{\sigma(d)}$.

By Lemma~\ref{lemsupequal} we have
$$
\sup_{y\in \Pi}f(y)=\sum_{j=1}^d\beta_j\rk_{\Pi}(X_j)\geq
\sup_{z\in Q(\rk_\Pi)}\sum_{j=1}^d\beta_j\sum_{i\in X_j}z_i=
\sup_{z\in Q(\rk_\Pi)}f(z).
$$
Since $\Pi$ is defined by inequalities of the form $f(y)\leq c$,  where $f$ is
a linear function and $c=\sup_{y\in \Pi}f(y)$, we see that $Q(\rk_{\Pi})\subseteq \Pi$.
We conclude that $Q(\rk_\Pi)=\Pi$.

Conversely, suppose that $\rk$ is a megamatroid, and that
$F$ is a face of $Q(\rk)$. Choose $y$ in the relative interior of $F$.
Let $S_F$ denote the set of all subsets $A$ of $\underline{d}$ for which $\sum_{i\in A} y_i=\rk(A)$.
Note that $\emptyset,\underline{d}\in S_F$. The linear hull of $F$ is given by the equations
$$
\sum_{i\in A}y_i=\rk(A),\quad A\in S_F.
$$

We claim that $S_F$ is closed under intersections and unions.
If $A,B\in S_F$, then we have
\begin{multline*}
\big(\sum_{i\in A\cap B}y_i-\rk(A\cap B)\big)+\big(\sum_{i\in A\cup B}y_i-\rk(A\cup B)\big)=\\
=\sum_{i\in A} y_i+\sum_{i\in B} y_i-\rk(A\cap B)-\rk(A\cup B)=\\
\rk(A)+\rk(B)-\rk(A\cap B)-\rk(A\cup B)\geq 0
\end{multline*}
by the submodular property. Since $\sum_{i\in A\cap B}y_i-\rk(A\cap B)$ 
and $\sum_{i\in A\cup B}y_i-\rk(A\cup B)$ are nonpositive, we conclude that
$A\cap B,A\cup B\in S_F$ and 
$$\rk(A)+\rk(B)=\rk(A\cap B)+\rk(A\cup B).$$

Let us call $A\in S_F$ {\em prime} if $A$ is nonempty and not the union of two proper subsets in $S_F$.
Let $P_F$ be the set of primes in $S_F$. If $C=A\cup B$, then
$$
\sum_{i\in C} y_i=\rk(C)
$$
follows from the equations
$$
\sum_{i\in A} y_i=\rk(A),\quad \sum_{i\in B}y_i=\rk(B),\quad \sum_{i\in A\cap B} y_i=\rk(A\cap B).
$$
Let $C_1,C_2,\dots,C_k$ be all prime sets in $S_F$.
It follows that the linear hull of $F$ is defined by all the equations
$$
\sum_{i\in C_j}y_i=\rk(C_j),\quad j=1,2,\dots,k.
$$
Every element of $S_F$ is a union of some of the $C_j$'s.
For every $j$, let $B_j$ be the largest proper subset of $C_j$ which lies in $S_F$.
Define $A_j=C_j\setminus B_j$ and $r_j=\rk(C_j)-\rk(B_j)$.
Then $A_1\cup \cdots \cup A_k=\underline{d}$ is a partition of $\underline{d}$, and every element of $S_F$
is a union of some of the $A_j$'s. The linear hull of $F$ is defined by the equations
$$
\sum_{i\in A_j}y_i=r_j,\quad j=1,2,\dots,k.
$$
Clearly, $\linhull(F)$ contains some integral vector $z\in \Z^d$
and $\linhull(F)$ is equal to $z+W$ where $W$ is the space spanned
by all $e_i-e_j$ where $i,j$ are such that $i,j\in A_k$ for some $k$.
\end{proof}

\section{The valuative property}\label{sec:valuative}
There are essentially two definitions of the valuative property in the literature,
which we will refer to as the {\em strong valuative} and the {\em weak valuative} properties.  The equivalence of these definitions is shown in 
\cite{Groemer} and~\cite{Volland} when valuations are defined on
sets of polyhedra closed under intersection.
In this section we will show the two definitions equivalent 
for valuations defined on megamatroid polytopes, which are not
closed under intersection.

\begin{definition}
A {\em megamatroid polyhedron decomposition} is a decomposition
$$
\Pi=\Pi_1\cup \Pi_2\cup \cdots\cup \Pi_k
$$
such that $\Pi,\Pi_1,\dots,\Pi_k$ are megamatroid polyhedra,
and $\Pi_i\cap \Pi_j$ is empty or contained in a proper face of $\Pi_i$ and of $\Pi_j$ for all $i\neq j$.
\end{definition}

\label{not:S}
Let $S_{\rm MM}(d,r)$ 
be the set of megamatroids 
on $\underline{d}$ of rank $r$.
\label{not:Z}
Let $Z_{\rm MM}(d,r)$ 
be the $\Z$-module whose basis is given by all $\langle \rk\rangle$
where $\rk\in S_{\rm MM}(d,r)$. 

For a
megamatroid polyhedron decomposition
$$
\Pi=\Pi_1\cup \Pi_2\cup \cdots\cup \Pi_k
$$
we define $\Pi_I=\bigcap_{i\in I}\Pi_i$ if $I\subseteq \{1,2,\dots,k\}$. We will use the convention
that $\Pi_{\emptyset}=\Pi$.
Define 
$$m_{\rm val}(\Pi;\Pi_1,\dots,\Pi_k)=\sum_{I\subseteq \{1,2,\dots,k\}}(-1)^{|I|}m_I\in Z_{\rm MM}(d,r),$$
where 
$m_I=\langle \rk^I\rangle$ if  $\rk^I$ is the megamatroid with $Q(\rk^I)=\Pi_I$,
and $m_I=0$ if $\Pi_I=\emptyset$.  We also define
$$
m_{\rm coval}(\Pi;\Pi_1,\dots,\Pi_k)=\langle \rk_\Pi\rangle-\sum_{F}\langle\rk_F\rangle.
$$
where $F$ runs over all interior faces of the decomposition.

\begin{definition}\label{def:weak}
A homomorphism of abelian groups $f:Z_{\rm MM}(d,r)\to
A$ is called {\em weakly valuative}, if for every
megamatroid polyhedron decomposition
$$
\Pi=\Pi_1\cup \Pi_2\cup \cdots\cup \Pi_k
$$
we have $f(m_{\rm val}(\Pi;\Pi_1,\dots,\Pi_k))=0$.
We say it is {\em weakly covaluative}, if for every
megamatroid polyhedron decomposition
$$
\Pi=\Pi_1\cup \Pi_2\cup \cdots\cup \Pi_k
$$
we have $f(m_{\rm coval}(\Pi;\Pi_1,\dots,\Pi_k))=0$.
\end{definition}

\label{not:E}
We define a group homomorphism
$$
E:Z_{\rm MM}(d,r)\to Z_{\rm MM}(d,r)
$$
by
$$
E(\langle \rk\rangle)=\sum_{F}\langle \rk_F\rangle
$$
where $F$ runs over all faces of $Q(\rk)$ and $\rk_F$ is the megamatroid
with $Q(\rk_F)=F$.
For a polytope $\Pi$, we denote the set of faces of $\Pi$ by $\face(\Pi)$.\label{not:face}
\begin{lemma}\label{lemweakstrong}
The homomorphism  $f:Z_{\rm MM}(d,r)\to A$ of abelian groups
is weakly valuative if and only if $f\circ E$ is weakly covaluative.
\end{lemma}

\begin{proof}
We have
\begin{multline}
E\big(m_{\rm val}(\Pi;\Pi_1,\dots,\Pi_k)\big)=
\sum_{I\subseteq \{1,2,\dots,k\}}(-1)^{|I|}E(m_I)=\\
=\sum_{I\subseteq \{1,2,\dots,k\}}(-1)^{|I|}\sum_{F\in\face(\Pi_I)}\langle\rk_F\rangle=
\sum_{F}\langle\rk_F\rangle\sum_{\scriptstyle I\subseteq\{1,2,\dots,k\};\atop \scriptstyle F\in\face(\Pi_I)}
(-1)^{|I|}.
\end{multline}
Let $J(F)$ be the set of all indices $i$ such that $F$ is a face of $\Pi_i$.
Suppose that $F$ is a face of $\Pi$. Then $J(F)=\emptyset$ if and only if $F=\Pi$.
We have
$$
\sum_{\scriptstyle I\subseteq\{1,2,\dots,k\};\atop \scriptstyle F\in\face(\Pi_I)}(-1)^{|I|}=
\sum_{I\subseteq J(F)}(-1)^{|I|}=\left\{\begin{array}{ll}
1 & \mbox{if $F=\Pi$;}\\
0 & \mbox{if $F\neq \Pi$.}\end{array}\right.
$$
If $F$ is an interior face, then $J(F)\neq \emptyset$ and
$$
\sum_{\scriptstyle I\subseteq\{1,2,\dots,k\};\atop \scriptstyle F\in\face(\Pi_I)}(-1)^{|I|}=
\sum_{I\subseteq J(F);I\neq\emptyset}(-1)^{|I|}=-1.
$$
We conclude that
$$
E(m_{\rm val}(\Pi;\Pi_1,\dots,\Pi_k))=\langle \rk_\Pi\rangle-\sum_{F}\langle \rk_F\rangle=m_{\rm coval}(\Pi;\Pi_1,\dots,\Pi_k)
$$
where the sum is over all interior faces $F$. The lemma follows.
\end{proof}

For a polyhedron $\Pi$ in $\R^d$, let $[\Pi]$ denote its indicator function.\label{not:[Pi]}
\label{not:P}
Define $P_{\rm MM}(d,r)$ 
as the $\Z$-module generated by all $[Q(\rk)]$, where $\rk$
lies in $S_{\rm MM}(d,r)$.

\label{not:Psi_MM}
There is a natural $\Z$-module homomorphism
$$
\Psi_{\rm MM}:Z_{\rm MM}(d,r)\to P_{\rm MM}(d,r)
$$
such that
$$
\Psi_{\rm MM}(\langle \rk\rangle)=[Q(\rk)]
$$
for all $\rk\in S_{\rm MM}(d,r)$.
\begin{definition}
A homomorphism of groups $f:Z_{\rm MM}(d,r)\to A$ is {\em strongly valuative} if
there exists a group homomorphism $\widehat{f}:P_{\rm MM}(d,r)\to A$
such that $f=\widehat{f}\circ \psi_{\rm MM}$.
\end{definition}
Suppose that $\Pi=\Pi_1\cup \cdots \cup \Pi_k$ is a megamatroid decomposition. Then by the
inclusion-exclusion principle, we have
\begin{multline*}
\Psi_{\rm MM}(m_{\rm val}(\Pi;\Pi_1,\dots,\Pi_k))=
\Psi_{\rm MM}\Big(\sum_{I\subseteq \{1,2,\dots,k\}}(-1)^{|I|}m_I\Big)=\\
=\sum_{I\subseteq \{1,2,\dots,k\}}(-1)^{|I|}\prod_{i\in I}[\Pi_i]
=
\sum_{I\subseteq \{1,2,\dots,k\}}(-1)^{|I|}[\Pi_I]=\prod_{i=1}^k ([\Pi]-[\Pi_i])=0.
\end{multline*}
This shows that every homomorphism $f:Z_{\rm MM}(d,r)\to A$ of abelian groups with the strong valuative property has the weak valuative property.  
In fact the two valuative properties are equivalent by the following theorem:

\begin{theorem}\label{theoweakstrong}
A homomorphism $f:Z_{\rm MM}(d,r)\to A$ of abelian groups is weakly valuative if and only if
it is strongly valuative.
\end{theorem}
The proof of Theorem~\ref{theoweakstrong} is in~\ref{apA}.

In view of this theorem, we will from now on just refer to the 
valuative property when we mean the weak or the strong valuative property.

\label{not:Psi_MM circ}
For a megamatroid polytope $\Pi$, let $\Pi^\circ$ be the relative interior of $\Pi$. \label{not:Picirc}
Define a homomorphism $\Psi^\circ_{\rm MM}:Z_{\rm MM}(d,r)\to P_{\rm MM}(d,r)$ by $\Psi^\circ_{\rm MM}(\langle\rk\rangle)=[Q^\circ(\rk)]$.
\begin{definition}
Suppose that $f:Z_{\rm MM}(d,r)\to A$ is a homomorphism of abelian groups.
We say that $f$ is {\em strongly covaluative} if $f$ factors through $\Psi^\circ_{\rm MM}$, i.e.,
there exists a group homomorphism $\widehat{f}$ such that $f=\widehat{f}\circ \Psi^\circ_{\rm MM}$.
\end{definition}

\begin{corollary}\label{corcoval}
A homomorphism $f:Z_{\rm MM}(d,r)\to A$ of abelian groups is weakly covaluative
if and only if it is strongly covaluative. 
\end{corollary}

\begin{proof}
If $\Pi=\Pi_1\cup \cdots \cup \Pi_k$ is a megamatroid polytope decomposition, then
$$
\Psi^\circ_{\rm MM}(m_{\rm coval}(\Pi;\Pi_1,\dots,\Pi_k))=
\Psi^\circ_{\rm MM}(\langle\rk_\Pi\rangle)-\sum_F \Psi^\circ_{\rm MM}(\langle\rk_F\rangle)=
[\Pi^\circ]-\sum_F [F^\circ]=0,
$$
where $F$ runs over all interior faces. This shows that if $f$ has the strong covaluative property, then
it has the weak covaluative property.

It is easy to verify that $\Psi^\circ_{\rm MM}\circ E=\Psi_{\rm MM}$.
Suppose that $f$ is weakly covaluative. By Lemma~\ref{lemweakstrong}, $f\circ E^{-1}$ is weakly valuative.
By Theorem~\ref{theoweakstrong}, $f\circ E^{-1}$ is strongly valuative, so $f\circ E^{-1}=\widehat{f}\circ \Psi_{\rm MM}$
for some group homomorphism $\widehat{f}$,
and $f=\widehat{f}\circ\Psi_{\rm MM}\circ E=\widehat{f}\circ\Psi^\circ_{\rm MM}$. This implies that $f$ is strongly covaluative.
\end{proof}
\begin{definition}
Suppose that $d\geq 1$.
A valuative group homomorphism $f:Z_{\rm MM}(d,r)\to A$ is {\em additive} if $f(\langle \rk\rangle)=0$
for all megamatroids $(\underline{d},\rk)$ for which $Q(\rk)$ has dimension $<d-1$.
\end{definition}

If $f:Z_{\rm MM}(d,r)\to
A$ is additive, then
megamatroid polyhedron decomposition
$$
\Pi=\Pi_1\cup \Pi_2\cup \cdots\cup \Pi_k
$$
we have
$$
f(\rk_\Pi)=\sum_{i=1}^kf(\langle \rk_{\Pi_i}\rangle).
$$

A megamatroid polyhedron decomposition $\Pi=\Pi_1\cup \cdots\cup \Pi_k$
is a {\em (poly)matroid polytope decomposition} 
if $\Pi,\Pi_1,\dots,\Pi_k$ are (poly)matroid polytopes. 
Let $S_{\rm (P)M}(d,r)$ 
be the set of (poly)matroids,
and let $Z_{\rm (P)M}(d,r)$ 
be the free abelian group generated by $S_{\rm (P)M}(d,r)$.
We say that $f:Z_{\rm (P)M}(d,r)\to A$
has the weak valuative
property if $f(m_{\rm val}(\Pi;\Pi_1,\dots,\Pi_k))=0$ for every 
(poly)matroid polytope decomposition.
We define the weak covaluative property for such homomorphisms $f$
in a similar manner. The group homomorphism $E:Z_{\rm MM}(d,r)\to Z_{\rm MM}(d,r)$
restricts to homomorphisms $Z_{\rm (P)M}(d,r)\to Z_{\rm (P)M}(d,r)$.
A group homomorphism $f:Z_{\rm (P)M}(d,r)\to A$
is weakly valuative
if and only if $f\circ E$ is weak covaluative. 
Let $P_{\rm (P)M}(d,r)=\Psi_{\rm MM}(Z_{\rm (P)M}(d,r))$
and define $\Psi_{\rm (P)M}:Z_{\rm (P)M}(d,r)\to P_{\rm (P)M}(d,r)$
as the restrictions of $\Psi_{\rm MM}$.
A homomorphism $f:Z_{\rm (P)M}(d,r)\to A$ 
is strongly valuative  if and only if it factors through $\Psi_{\rm (P)M}$.

\begin{corollary}
A homomorphism $f:Z_{\rm (P)M}(d,r)\to A$ 
is weakly valuative if and only
if it is strongly valuative.
\end{corollary}

\begin{proof}
We need to show that $\ker \Psi_{\rm (P)M}(d,r)=W_{\rm (P)M}(d,r)$.
It is clear that $W_{\rm (P)M}(d,r)\subseteq \ker \Psi_{\rm (P)M}(d,r)$. 
By Theorem~\ref{theoweakstrong}, 
we have that $\ker \Psi_{\rm MM}(d,r)=W_{\rm MM}(d,r)$, so
$\ker \Psi_{\rm (P)M}(d,r)=W_{\rm MM}(d,r)\cap Z_{\rm (P)M}(d,r)$.
Define $\pi_{\rm (P)M}:Z_{\rm MM}(d,r)\to Z_{\rm (P)M}(d,r)$ by
$\pi_{\rm (P)M}(\langle \rk\rangle)=\langle\rk'\rangle$ 
where $Q(\rk')=Q(\rk)\cap \Delta_{\rm (P)M}(d,r)$
if this intersection is nonempty and $\pi_{\rm (P)M}(\langle\rk\rangle)=0$ otherwise.
Note that $\pi_{\rm (P)M}$ is a projection of $Z_{\rm MM}(d,r)$ onto $Z_{\rm (P)M}(d,r)$.
We have
$$
\pi_{\rm (P)M}(m_{\rm val}(\Pi;\Pi_1,\dots,\Pi_k))=
m_{\rm val}(\Pi\cap \Delta;\Pi_1\cap \Delta,\dots,\Pi_k\cap \Delta)\in W_{\rm (P)M}(d,r),
$$
where $\Delta=\Delta_{\rm (P)M}(d,r)$.
This shows that $\pi_{\rm (P)M}(W_{\rm MM}(d,r))\subseteq W_{\rm (P)M}(d,r)$.
We conclude that 
$$\ker \Psi_{\rm (P)M}(d,r)=W_{\rm MM}(d,r)\cap Z_{\rm (P)M}(d,r)\subseteq \pi_{\rm (P)M}(W_{\rm MM}(d,r))\subseteq W_{\rm (P)M}(d,r).
$$
\end{proof}

We can also define the strong covaluative property for a group homomorphism 
$f:Z_{\rm (P)M}(d,r)\to A$.
The proof of Corollary~\ref{corcoval} generalizes to (poly)matroids and
 $f$ is weakly covaluative 
if and only
if $f$ is strongly covaluative.

\section{Decompositions into cones}
A chain of length $k$ in $\underline{d}$ is 
$$
\underline{X}:\emptyset \subset X_1\subset \cdots \subset X_{k-1}\subset X_{k}=\underline{d}
$$
(here $\subset$ denotes proper inclusion).
We will write $\ell(\underline{X})=k$ for the length of such a chain.\label{not:ell}
If $d>0$ then every chain has length $\geq 1$, but for $d=0$ there is exactly 1 chain,
namely
$$
\emptyset=\underline{0}
$$
and this chain has length $0$.
For a chain $\underline{X}$ of length $k$ 
and a $k$-tuple $\underline{r}=(r_1,r_2,\dots,r_k)\in(\Z\cup\{\infty\})^k$, 
we define a megamatroid polyhedron
\label{not:R}
$$
R_{\rm MM}(\underline{X},\underline{r})=\Big\{(y_1,\dots,y_d)\in \R^d \mathrel{\Big|}
\sum_{i=1}^d y_i=r_k,\ \forall j\ \sum_{i\in X_{j}}y_i\leq r_j\Big\}.
$$
We will always use the conventions $r_0=0$, $X_0=\emptyset$.
The megamatroid $\rk_{\underline{X},\underline{r}}$ is defined by $Q(\rk_{\underline{X},\underline{r}})=R_{\rm MM}(\underline{X},\underline{r})$.

For a megamatroid $\rk$ and a chain $\underline{X}$ of length $k$ we define 
$$
R_{\rm MM}(\underline{X},\rk)=R_{\rm MM}(\underline{X},(\rk(X_1),\rk(X_2),\dots,\rk(X_k))).
$$

Suppose that $\Pi$ is a polyhedron in $\R^d$ defined by $g_i(y_1,\dots,y_d)\leq c_i$
for $i=1,2,\dots,n$, 
where $g_i:\R^d\to \R$ is linear and $c_i\in \R$.
For every face $F$ of $\Pi$, the tangent cone $\Cone_F$ of $F$ is
defined by the inequalities
$$
g_i(y_1,\dots,y_d)\leq c_i
$$
for all $i$ for which the restriction of $g_i$ to $F$ is constant and equal to $c_i$.
\begin{theorem}[Brianchon-Gram Theorem~\cite{Brianchon,Gram}]\label{theoBG}
We have the following equality
$$
[\Pi]=\sum_{F}(-1)^{\dim F}[\Cone_F]
$$
where $F$ runs over all the bounded faces of $\Pi$.
\end{theorem}
For a proof, see~\cite{McMullen}.

\begin{theorem}\label{theo1}
For any  megamatroid  $\rk:2^{\underline{d}}\to \Z\cup \{\infty\}$ we have
$$
[Q(\rk)]= \sum_{\underline{X}}(-1)^{d-\ell(\underline{X})}[R_{\rm MM}(\underline{X},\rk)].
$$
\end{theorem}
\begin{proof}
Assume first that $\rk(X)$ is finite for all $X\subseteq\underline d$.  
We define a convex polyhedron $Q_{\varepsilon}(\rk)$ by the
inequalities
$$
\sum_{i\in A} y_i\leq \rk(A)+\varepsilon (d^2-|A|^2)
$$
for all $A\subseteq \underline{d}$
and the equality $y_1+\cdots+y_d=r$, where $r=\rk(\underline{d})$.

Faces of $Q_\varepsilon(\rk)$ are given by intersecting $Q_{\varepsilon}(\rk)$
with hyperplanes of the form 
$$
H_A=\Big\{(y_1,\dots,y_d)\in \R^d \mid \sum_{i\in A} y_i=\rk(A)+\varepsilon(d^2-|A|^2)\Big\}.
$$
If $A,B\subseteq\underline{d}$, and $A$ and $B$ are incomparable, and $(y_1,\dots,y_d)\in H_A\cap H_B\cap Q_{\varepsilon}(\rk)$, then
\begin{multline*}
\sum_{i\in A}y_i+\sum_{i\in B}y_i
=\rk(A)+\rk(B)+\varepsilon((d^2-|A|^2)+(d^2-|B|^2))>\\
>
\rk(A)+\rk(B)+\varepsilon((d^2-|A\cup B|^2)+(d^2-|A\cap B|^2))=\\
\geq\rk(A\cup B)+\rk(A\cap B)+\varepsilon((d^2-|A\cup B|^2)+(d^2-|A\cap B|^2))\geq\\
\geq
\sum_{i\in A\cup B} y_i+\sum_{i\in A\cap B}y_i=\sum_{i\in A}y_i+\sum_{i\in B}y_i
\end{multline*}
This contradiction shows that $H_A\cap H_B\cap Q_{\varepsilon}(\rk)=\emptyset$.
It follows that all faces are of the form 
$$
F_\varepsilon(\underline{X})=Q_\varepsilon(\rk)\cap H_{X_1}\cap \cdots\cap H_{X_{k-1}}
$$
where $k\geq 1$ and 
$$
X_0=\emptyset\subset X_1\subset \cdots \subset X_{k-1}\subset X_k=\underline{d}.
$$
Also, all these faces are distinct.

Let us view $Q_\varepsilon(\rk)$ as a bounded 
polytope in the hyperplane $y_1+y_2+\cdots+y_d=r$.
For a face $F_{\varepsilon}(\underline{X})$, its tangent cone $\Cone_{F_{\varepsilon}(\underline{X})}$
is defined by the inequalities
$$
\sum_{i\in X_j} y_i\leq \rk(X_j)+\varepsilon(d^2-|X_j|^2)
$$
(and the equality $\sum_{i=1}^d y_i=r$).
If $\underline{X}$ has length $k$, then the dimension of $F_{\varepsilon}(\underline{X})$ is $d-k$.
Theorem~\ref{theoBG} implies that
$$
[Q_\varepsilon(\rk)]= \sum_{\underline{X}}(-1)^{d-\ell(\underline{X})}[\Cone_{F_{\varepsilon}(\underline{X})}].
$$
When we take the limit $\varepsilon\downarrow 0$, then 
$[Q_\varepsilon(\rk)]$ converges pointwise to $[Q(\rk)]$,
and $[\Cone_{F_{\varepsilon}(\underline{X})}]$
converges pointwise to $[R_{\rm MM}(\underline{X},\rk)]$.

Finally, for a general polymatroid $\rk$, we have
$\rk = \lim_{N\to\infty} \rk^N$, where $\rk^N$ is as in the proof of Lemma~\ref{lemnonempty},
and $\rk^N$ has all ranks finite, and likewise
$$\lim_{N\to\infty} [R_{MM}(\underline X,\rk^N)] = [R_{MM}(\underline X,\rk)].$$
So the result follows by taking limits. 
\end{proof}


\begin{example}
\begin{figure}[t]
\leaveout{
\centerline{ 
\includegraphics[width=4in,height=4in]{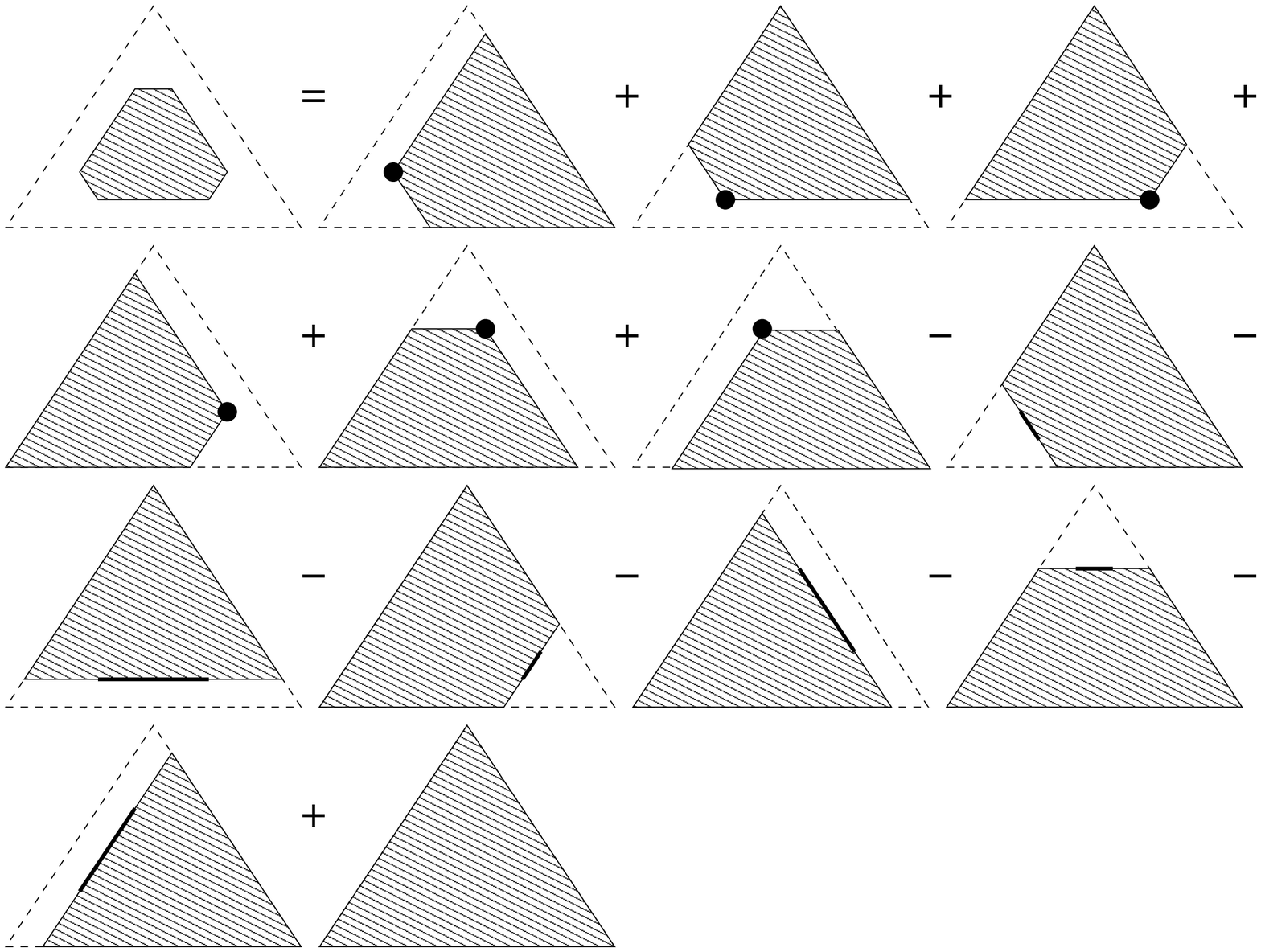}
}}
\caption{A decomposition of~$Q_\epsilon(\rk)$, as in Theorem~\ref{theo1}.}
\label{fig:BGtheorem}
\end{figure}

\begin{figure}[t]
\leaveout{
\centerline{ 
\includegraphics[width=4in,height=4in]{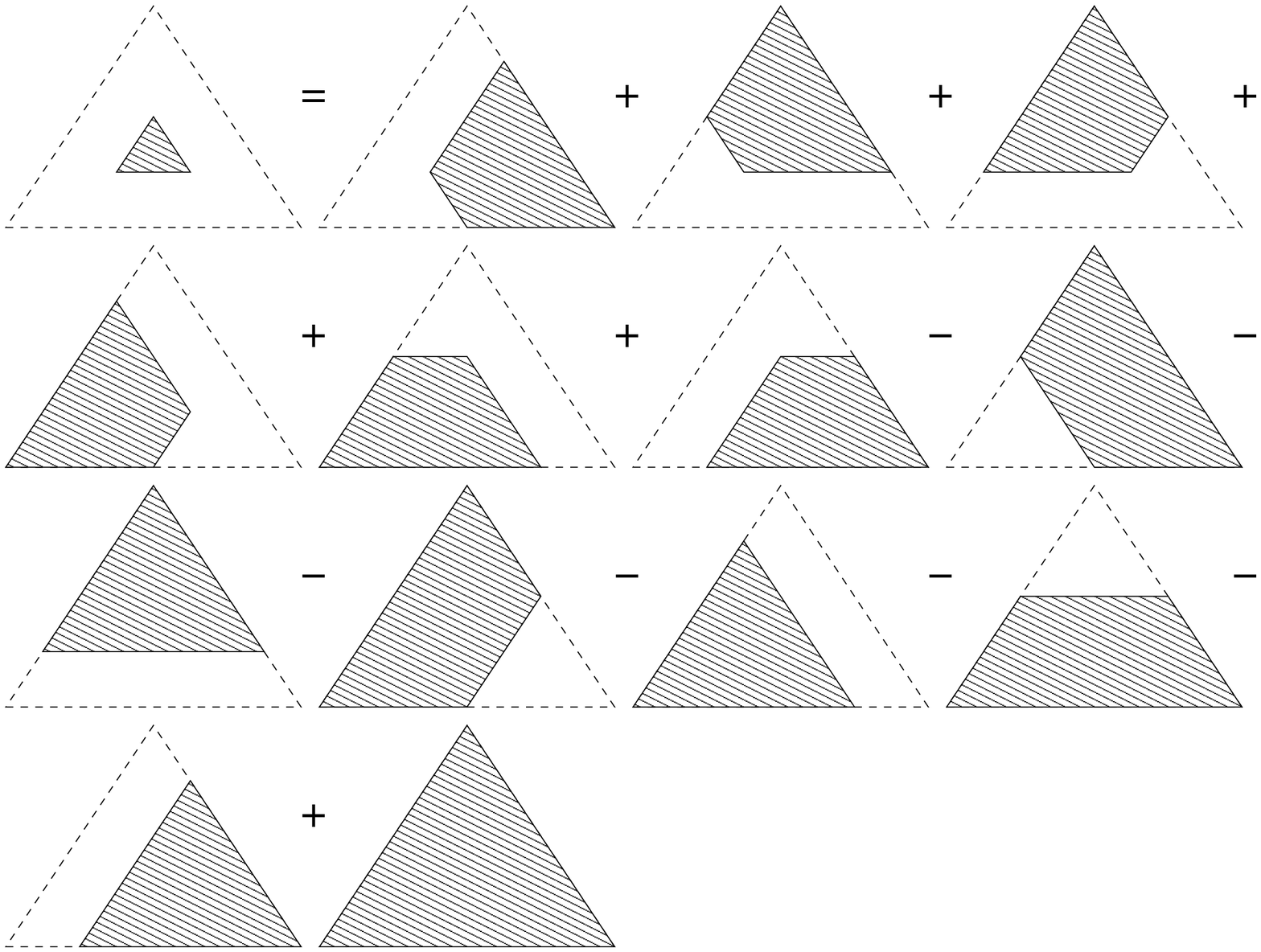}
}}
\caption{The limiting decomposition of~$Q(\rk)$ 
corresponding to Figure~\ref{fig:BGtheorem}.}
\label{fig:BGtheorem2}
\end{figure}

To illustrate the proof of Theorem~\ref{theo1}, consider the case where $d=3$ and $r=3$, and
$\rk$ is defined by $\rk(\{1\})=\rk(\{2\})=\rk(\{3\})=2$, 
$\rk(\{1,2\})=\rk(\{2,3\})=\rk\{(1,3\})=3$, $\rk(\{1,2,3\})=4$. 
The decomposition of $Q_{\varepsilon}(\rk)$ 
using the Brianchon-Gram theorem is depicted in Figure~\ref{fig:BGtheorem}.
Note how the summands in the decomposition correspond to the faces of $Q_{\varepsilon}(\rk)$.
The dashed triangle is the triangle defined by $y_1,y_2,y_3\geq 0$, $y_1+y_2+y_3=4$.
Instead of getting cones in the decomposition, we get polygons because we intersect
with this triangle. 

In the limit where $\varepsilon$ approaches $0$ we obtain Figure~\ref{fig:BGtheorem2}. 
This is exactly the decomposition in Theorem~\ref{theo1}. 
In this decomposition, the summands do not correspond to the faces of $Q(\rk)$.
\end{example}

\section{Valuative functions: the groups $P_{\rm M},P_{\rm PM},P_{\rm MM}$}
\begin{lemma}
The function~$\mathbf 1:Z_{MM}(d,r)\to \Z$ 
such that $\mathbf 1(\langle\rk\rangle)=1$ for every megamatroid~$\rk$ has the valuative property.
\end{lemma}

\begin{proof}
Let 
$$\Pi=\Pi_1\cup \Pi_2\cup \cdots\cup \Pi_k$$ 
be a megamatroid polyhedron decomposition.
By Rota's crosscut theorem~\cite{Rota}, 
$$\mathbf 1(m_{\rm val}(\Pi;\Pi_1,\dots,\Pi_k))=\sum_F\mu(\Pi,F)=0,$$
where $F$ runs over the faces of the decomposition,
and $\mu$ is the M\"obius function.  
\end{proof}

\begin{lemma}\label{lemj_H}
Let $H\subseteq\R^d$ be 
a closed halfspace.  
Define $j_H:Z_{MM}(d,r)\to\Z$ by 
$$j_H(\langle\rk\rangle)=\left\{
\begin{array}{ll}
1 & \mbox{if $Q(\rk)\subseteq H$,}\\
0 & \mbox{otherwise.}
\end{array}\right.$$
Then $j_H$ is valuative.
\end{lemma}

\begin{proof}
Let 
$$\Pi=\Pi_1\cup \Pi_2\cup \cdots\cup \Pi_k$$
be a megamatroid polyhedron decomposition.  
The intersections of the faces of this decomposition with~$\R^d\setminus H$ 
establish a regular cell complex structure on~$\Pi\setminus H$,
and a face $F$ of the decomposition meets $\R^d\setminus H$ if and only if
$(\mathbf 1-j_H)(\rk_F) = 1$.  It follows that $\mathbf 1-j_H$ 
is valuative, by the argument of the previous proof
applied to this complex.  

\end{proof}

Lemma~\ref{lemj_H} can also be deduced from the fact that 
the indicator function of the polar dual has the 
valuative property (see~\cite{Lawrence}).

\label{not:s}
Suppose that $\underline{X}$ is a chain of length $k$ and 
$\underline{r}=(r_1,\dots,r_k)$ is an integer vector with $r_k=r$.  
Define a homomorphism $s_{\underline X,\underline r}:Z_{\rm MM}(d,r)\to\Z$
by
$$s_{\underline{X},\underline{r}}(\rk)=\left\{
\begin{array}{ll}
1 & \mbox{if $\rk(X_j)=r_j$ for  $j=1,2,\dots,k$,}\\
0 & \mbox{otherwise.}
\end{array}\right.$$

\begin{proposition}\label{s is valuative}
The homomorphism $s_{\underline X,\underline r}$ is valuative. 
\end{proposition}
\begin{proof}
For $\varepsilon>0$, define the halfplane $H_1(\varepsilon)$ by the inequality
$$
\sum_{j=1}^k\varepsilon^{j-1}\sum_{i\in X_j}y_i\leq \sum_{j=1}^k \varepsilon^{j-1}r_j
$$
and define $H_2(\varepsilon)$ by
$$
\sum_{j=1}^k\varepsilon^{j-1}\sum_{i\in X_j}y_i\leq \sum_{j=1}^k \varepsilon^{j-1}r_j-\varepsilon^k.
$$
By Lemma~\ref{lemsupequal} and Lemma~\ref{lemj_H},  $(j_{H_{1}(\varepsilon)}-j_{H_2(\varepsilon)})(\rk)=1$ if and only if
\begin{equation}\label{eq:epsineq}
\sum_{j=1}^k \varepsilon^{j-1}r_j-\varepsilon^k < \sum_{j=1}^k\varepsilon^{j-1}\rk(X_j)=\max_{y\in Q(\rk)} \sum_{j=1}^k\varepsilon^{j-1}\sum_{i\in X_j}y_i\leq \sum_{j=1}^k \varepsilon^{j-1}r_j 
\end{equation}
If (\ref{eq:epsineq}) holds for arbitrary small $\varepsilon$, then it is easy to see
(by induction on $j$) that $\rk(X_j)=r_j$ for $j=1,2,\dots,k$.
From this follows that $\lim_{\varepsilon\to 0} j_{H_{1}(\varepsilon)}-j_{H_{2}(\varepsilon)}=s_{\underline{X},\underline{r}}$.
So $s_{\underline{X},\underline{r}}$ is valuative.
\end{proof}

Suppose that $d\geq 1$.
Let ${\mathfrak p}_{\rm MM}(d,r)$
 be the set of all pairs 
$(\underline{X},\underline{r})$ such that
$\underline{X}$ is a chain of length $k$ ($1\leq k\leq d$) and $\underline{r}=(r_1,r_2,\dots,r_k)$ is
an integer vector with $r_k=r$.
We define $R_{\rm (P)M}(\underline{X},\underline{r})=R_{\rm MM}(\underline{X},\underline{r})\cap \Delta_{\rm (P)M}(d,r)$.
If $R_{\rm (P)M}(\underline{X},\underline{r})$
is nonempty, then it is a (poly)matroid base polytope.
\label{not:mf p}
Define ${\mathfrak p}_{\rm PM}(d,r)\subseteq {\mathfrak p}_{\rm MM}(d,r)$
as the set of all pairs $(\underline{X},\underline{r})$ with
$0\leq r_1< \cdots < r_k=r$.
Let ${\mathfrak p}_{\rm M}(d,r)$ denote
the set of all pairs $(\underline{X},\underline{r})\in {\mathfrak p}_{\rm MM}(d,r)$
such that $\underline{r}=(r_1,\dots,r_k)$ for some $k$ ($1\leq k\leq d$),
$$
0\leq r_1<r_2<\cdots <r_k=r
$$
and
$$
0<|X_1|-r_1<|X_2|-r_2<\cdots <|X_{k-1}|-r_{k-1}\leq |X_k|-r_k=d-r.
$$
For $d=0$, we define ${\mathfrak p}_{\rm MM}(0,r)={\mathfrak p}_{\rm PM}(0,r)={\mathfrak p}_{\rm M}(0,r)=\emptyset$ for $r\neq 0$
and ${\mathfrak p}_{\rm MM}(0,0)={\mathfrak p}_{\rm PM}(0,0)={\mathfrak p}_{\rm M}(0,0)=\{(\emptyset \subseteq\underline{0},())\}$.

\begin{theorem}\label{free generation} 
The group $P_{\rm *M}(d,r)$ is freely generated by the basis
$$\big\{[R_{\rm *M}(\underline{X},\underline{r})] \mathrel{\big|}
(\underline{X},\underline{r})\in {\mathfrak p}_{\rm *M}(d,r)\big\}.$$
\end{theorem}

\begin{proof}
The case $d=0$ is easy, so assume that $d\geq 1$.

\noindent {\em For megamatroids.}
If $\rk$ is a megamatroid, then $[Q(\rk)]$ is an integral combination of functions 
$[R_{\rm MM}(\underline{X},\underline{r})]$,
$(\underline{X},\underline{r})\in {\mathfrak p}_{\rm MM}(d,r)$ by Theorem~\ref{theo1}.
This shows that $[R_{\rm MM}(\underline{X},\underline{r})]$,
$(\underline{X},\underline{r})\in {\mathfrak p}_{\rm MM}(d,r)$ generate $P_{\rm MM}(d,r)$.
If $s_{\underline{X},\underline{r}}
    (R_{\rm MM}(\underline{X}',\underline{r}'))\neq 0$
then $\rk_{\underline{X}',\underline{r}'}(X_j)=r_j$ for all $j$, and
$R_{\rm MM}(\underline{X}',\underline{r}')\subseteq R_{\rm MM}(\underline{X},\underline{r})$.
Suppose that
$$
\sum_{i=1}^k a_i[R_{\rm MM}(\underline{X}^{(i)},\underline{r}^{(i)})]=0
$$
with $k\geq 1$, $a_1,\dots,a_k$ nonzero integers, and $(\underline{X}^{(i)},\underline{r}^{(i)})$, $i=1,2,\dots,k$
distinct.
Without loss of generality we may assume that 
$R_{\rm MM}(\underline{X}^{(1)},\underline{r}^{(1)})$ does
not contain $R_{\rm MM}(\underline{X}^{(i)},\underline{r}^{(i)})$ for any $i>1$.
We have
$$
0=s_{\underline{X}^{(1)},\underline{r}^{(1)}}
\Big(\sum_{i=1}^k a_i R_{\rm MM}(\underline{X}^{(i)},\underline{r}^{(i)})\Big)=a_1.
$$
Contradiction. 

\noindent {\em For polymatroids.}
It is clear that  $P_{\rm PM}(d,r)$ is generated by all 
$[R_{\rm PM}(\underline{X},\underline{r})]$,
with $(\underline{X},\underline{r})\in {\mathfrak p}_{\rm MM}(d,r)$.
If $r_1<0$ then $R_{\rm PM}(\underline{X},\underline{r})$ is empty.
Suppose that $r_{i+1}\leq r_{i}$. It is obvious that
$$
R_{\rm PM}(\underline{X},\underline{r})=R_{\rm PM}(\underline{X}',\underline{r}')
$$
where 
$$
\underline{X}':\emptyset=X_0\subset X_1\subset \cdots \subset X_{i-1}\subset X_{i+1}\subset \cdots \subset X_k=\underline{d}
$$
and
$$
\underline{r}'=(r_1,r_2,\dots,r_{i-1},r_{i+1},\dots,r_k).
$$
Therefore, $P_{\rm PM}(d,r)$ is generated by all $[R_{\rm PM}(\underline{X},\underline{r})]$
where $(\underline{X},\underline{r})\in {\mathfrak p}_{\rm PM}(d,r)$.
If $\Pi=R_{\rm PM}(\underline{X},\underline{r})$ with $(\underline{X},\underline{r})\in {\mathfrak p}_{\rm PM}(d,r)$,
then $(\underline{X},\underline{r})$ is completely determined by the polytope $\Pi$.
For $1\leq i\leq d$, define $a_i=\max\{y_i\mid y\in \Pi\}$.
Then $\underline{r}$ is determined by $0\leq r_1<\cdots<r_k$
and 
$$
\{r_1,\dots,r_k\}=\{a_1,\dots,a_d\}.
$$
The sets $X_j$, $j=1,2,\dots,k$ are determined by $X_j=\{i\mid a_i\leq r_j\}$.
This shows that the polytopes $R_{\rm PM}(\underline{X},\underline{r})$, $(\underline{X},\underline{r})\in {\mathfrak p}_{\rm PM}(d,r)$,
are distinct.
A similar argument as in the megamatroid case 
shows that $[R_{\rm PM}(\underline{X},\underline{r})]$, $(\underline{X},\underline{r})\in {\mathfrak p}_{\rm PM}(d,r)$, are linearly independent.

\noindent {\em For matroids.}
From the polymatroid case it follows  that $P_{\rm M}(d,r)$ is generated
by all $[R_{\rm M}(\underline{X},\underline{r})]$, where $(\underline{X},\underline{r})\in {\mathfrak p}_{\rm PM}(d,r)$.
  Suppose that $|X_{i-1}|-r_{i-1}\geq |X_{i}|-r_{i}$
for some $i$ with $1\leq i\leq k$ (with the convention that $r_0=0$).
Then we have
$$
[R_{\rm M}(\underline{X},\underline{r})]=[R_{\rm M}(\underline{X}',\underline{r}')]
$$
where
$$
\underline{X}':\emptyset=X_0\subset X_1\subset \cdots \subset X_{i-1}\subset X_{i+1}\subset \cdots \subset X_k=\underline{d}.
$$
and
$$
\underline{r}=(r_1,\dots,r_{i-1},r_{i+1},\dots,r_k).
$$
This shows that $P_{\rm M}(d,r)$ is generated by all $[R_{\rm M(\underline{X},\underline{r})}]$
where $(\underline{X},\underline{r})\in {\mathfrak p}_{\rm M}(d,r)$.
If $\Pi=R_{\rm M}(\underline{X},\underline{r})$ with $(\underline{X},\underline{r})\in {\mathfrak p}_{\rm M}(d,r)$,
then $(\underline{X},\underline{r})$ is completely determined by the polytope $\Pi$.
Note that $\rk_{\Pi}(A)=\min_j\{\rk_{\Pi}(X_j)+|A|-|A\cap X_j|\}$.
If $\emptyset\subset A\subset \underline{d}$ then $A=X_j$ for some $j$ if and only if $\rk_{\Pi}(A)<\rk_{\Pi}(B)$ for all $B$ with $A\subset B\subseteq \underline{d}$ and $|A|-\rk_{\Pi}(A)>|B|-\rk_{\Pi}(B)$ for all $B$ with $\emptyset\subseteq B\subset A$.
So $X_1,\dots,X_{k}$ are determined by $\Pi$, and $r_i=\rk_{\Pi}(X_j)$, $j=1,2,\dots,k$ are
determined as well.
This shows that the polytopes $R_{\rm M}(\underline{X},\underline{r})$, $(\underline{X},\underline{r})\in {\mathfrak p}_{\rm M}(d,r)$,
are distinct.
A similar argument as in the megamatroid case 
shows that $[R_{\rm M}(\underline{X},\underline{r})]$, $(\underline{X},\underline{r})\in {\mathfrak p}_{\rm M}(d,r)$, are linearly independent.
\end{proof}

\newcommand{\sleq}{{\mbox{\tiny$\leq$}}}
Let $(\underline X,\underline r)\in\mathfrak p_{\rm MM}(d,r)$.
Consider the homomorphism $s^{\leq}_{\underline X,\underline r}:Z_{MM}(d,r)\to\Z$
defined by
$$s^\sleq_{\underline{X},\underline{r}}(\rk)=\left\{
\begin{array}{ll}
1 & \mbox{if $\rk(X_j)\leq r_j$ for  $j=1,2,\dots,k$,}\\
0 & \mbox{otherwise.}
\end{array}\right.$$
This homomorphism $s^\sleq_{\underline X,\underline r}$
is a (convergent infinite) sum of several homomorphisms
of the form $s_{\underline X',\underline r'}$, so by
Proposition~\ref{s is valuative} it is valuative. 

In view of Theorem~\ref{free generation},
if $f:Z_{\rm (P)M}(d,s)\to\Z$ is valuative,
$f$ is determined by its values on the (poly)matroids $R_{\rm (P)M}$,
since the spaces $P_{\rm (P)M}(d,r)$ are finite-dimensional.
For a (poly)matroid $\rk$, 
$s^\sleq_{\underline{X},\underline{r}}(\rk)=1$
if and only if $Q(\rk)$ is contained in $Q(R_{\rm (P)M}(\underline X,\underline R))$.
Therefore, the matrix specifying the pairing 
$P_{\rm (P)M}(r,d)\otimes P_{\rm (P)M}(r,d)^\vee\to\Z$
whose rows correspond to the polytopes 
$Q(R_{\rm (P)M}(\underline X,\underline R))$, in some linear extension
of the order of these polytopes by containment, and whose columns
correspond in the same order to~$s^\sleq_{\underline X,\underline r}$,
is triangular.  The next corollary follows.

\begin{corollary}\label{cor s univ}
The group $P_{\rm (P)M}(d,r)^\vee$ of valuations $Z_{\rm (P)M}(d,r)\to\Z$
has the two bases 
$$\big\{s_{\underline X,\underline r}
: (\underline X,\underline r)\in\mathfrak p_{\rm (P)M}(d,r)\big\}$$
and 
$$\big\{s^\sleq_{\underline X,\underline r}
: (\underline X,\underline r)\in\mathfrak p_{\rm (P)M}(d,r)\big\}.$$
\end{corollary}

If $\underline{X}$ is not a maximal chain, then $s_{\underline{X},\underline{r}}$ is
a linear combination of functions of the form $s_{\underline{X}',\underline{r}'}$ where $\underline{X}'$ is a
maximal chain. The following corollary follows from Corollary~\ref{cor s univ}.
\begin{corollary}\label{cor maximal chains}
The group $P_{\rm PM}(d,r)^\vee$ of valuations $Z_{\rm PM}(d,r)\to\Z$
is generated by the functions~$s_{\underline X,\underline r}$
where $\underline X$ is a chain of subsets of $[d]$ of length~$d$
and $r=(r_1,\ldots,r_d)$ is an integer vector with
$0\leq r_1\leq\cdots\leq r_d=r$.
\end{corollary}

The generating set of this corollary appeared as the coordinates of
the function $H$ defined in~\cite[\S6]{AFR}, which was introduced
there as a labeled analogue of the first author's $\mathcal G$.  



\begin{proof14d}
Let ${\mathfrak a}(d,r)$ be the set of all sequences $(a_1,\dots,a_d)$
with $0\leq a_i\leq r$ for all $i$ and $a_i=r$ for some $i$.
Clearly $|{\mathfrak a}(d,r)|=(r+1)^d-r^d$. We define a bijection $f:{\mathfrak p}_{\rm PM}(d,r)\to {\mathfrak a}(d,r)$
as follows. If $(\underline{X},\underline{r})\in {\mathfrak p}_{\rm PM}(d,r)$, then we
define
$$
f(\underline{X},\underline{r})=(a_1,a_2,\dots,a_d)
$$
where $a_i=r_j$ and $j$ is minimal such that $i\in X_{j}$.

Suppose that $(a_1,\dots,a_d)\in {\mathfrak a}(d,r)$.
Let $k$ be the cardinality of $\{a_1,\dots,a_d\}$. 
Now $r_1<r_2<\cdots<r_k$ are defined by
$$
\{r_1,r_2,\dots,r_k\}=\{a_1,\dots,a_d\}
$$
and 
for every $j$, we define
$$
X_{j}=\{i\in \underline{d}\mid a_i\leq r_j\}.
$$
Then we have
$$
f^{-1}(a_1,\dots,a_d)=(\underline{X},\underline{r}).
$$

A generating function for $p_{\rm PM}(d,r)$ is
\begin{multline*}
\sum_{d=0}^\infty\sum_{r=0}^\infty \frac{p_{\rm PM}(d,r)x^dy^r}{d!}=
1+\sum_{d=1}^\infty\sum_{r=0}^\infty\frac{(r+1)^d-r^d}{d!}x^dy^r=\\
=1+\sum_{r=0}^\infty\sum_{d=0}^\infty\frac{(r+1)^d-r^d}{d!}x^dy^r=1+\sum_{r=0}^\infty (e^{(r+1)x}-e^{rx})y^r=1+\frac{e^x-1}{1-ye^x}=\frac{e^x(1-y)}{1-ye^x}.
\end{multline*}
\end{proof14d}
\begin{proof14c}
Suppose that $(\underline{X},\underline{r})\in {\mathfrak p}_{\rm M}(d,r)$ has length $k$.
Define $u_1,u_2,\dots,u_k$ by 
$$u_1=r_1,\quad u_i=r_i-r_{i-1}-1\  (2\leq i\leq k).$$
Define $v_1,v_2,\dots,v_k$ by 
\begin{eqnarray*}v_i &=& (|X_i|-r_i)-(|X_{i-1}|-r_{i-1})-1\ (1\leq i\leq k-1),\\
 v_k &=& (|X_k|-r_k)-(|X_{k-1}|-r_{k-1})=d-r-|X_{k-1}|+r_{k-1}.
 \end{eqnarray*}
 
If $(\underline{X},\underline{r})\in {\mathfrak p}_{\rm M}(d,r)$, then we have that
$u_1,\dots,u_k,v_1,\dots,v_k$ are nonnegative, and
$$
u_1+\cdots+u_k=r-k+1,\quad v_1+\cdots+v_k=d-r-k+1.
$$
Let $Y_i=X_i\setminus X_{i-1}$ for $i=1,2,\dots,k$.
If $k\geq 2$, then we have
$u_1+v_1+1=|Y_1|$, $u_k+v_k+1=|Y_k|$ and $u_i+v_i+2=|Y_i|$ for $i=2,3,\dots,k-1$.
There are 
$$
\frac{d!}{(u_1+v_1+1)!(u_2+v_2+2)!(u_3+v_3+2)!\cdots (u_{k-1}+v_{k-1}+2)!(u_k+v_k+1)!}
$$
partitions of $\underline{d}$ into the subsets $Y_1,Y_2,\dots,Y_k$, such that
$(\underline{X},\underline{r})$ has the given $u$ and $v$ values.
If $k=1$, then $u_1+v_1=d$ and there is 
$$
1=\frac{d!}{(u_1+v_1)!}
$$
 pair $(\underline{X},\underline{r})$ with given $u$ and $v$ values.

This yields the generating function
\begin{multline}\label{eqgenfun}
\sum_{d=0}^\infty \sum_{r=0}^d \frac{p_{\rm M}(d,r)}{d!}x^{d-r}y^r=
\sum_{u_1,v_1\geq 0}\frac{t^{u_1}s^{v_1}}{(u_1+v_1)!}+\\
+
\sum_{\scriptstyle u_1,\dots,u_k\geq 0 \atop \scriptstyle v_1,\dots,v_k\geq 0}
\frac{x^{u_1+u_2+\cdots+u_k+k-1}y^{v_1+v_2+\cdots+v_k+k-1}}{
(u_1+v_1+1)! (u_2+v_2+2)!\cdots (u_{k-1}+v_{k-1}+2)!(u_k+v_k+1)!}
\end{multline}
We have that
\begin{equation}\label{equv1}
\sum_{u,v\geq 0}\frac{x^uy^v}{(u+v)!}=\sum_{d=0}^\infty \sum_{u+v=d}
\frac{t^us^v}{d!}=\sum_{d=0}^\infty \frac{x^{d+1}-y^{d+1}}{(x-y)d!}=\frac{xe^x-ye^y}{x-y},
\end{equation}
\begin{multline}\label{equv2}
\sum_{u,v\geq 0}\frac{t^us^v}{(u+v+1)!}=\sum_{d=0}^\infty \sum_{u+v=d}
\frac{x^uy^v}{(d+1)!}=\sum_{d=0}^\infty \frac{x^{d+1}-y^{d+1}}{(x-y)(d+1)!}=
\\
=
\sum_{d=1}^\infty \frac{x^d-y^d}{(x-y)d!}=\sum_{d=0}^\infty \frac{x^d-y^d}{(x-y)d!}=
\frac{e^x-e^y}{x-y},
\end{multline}
and
\begin{multline}\label{equv3}
\sum_{u,v\geq 0}\frac{x^uy^v}{(u+v+2)!}=\sum_{d=0}^\infty \sum_{u+v=d}
\frac{x^uy^v}{(d+2)!}=\sum_{d=0}^\infty \frac{x^{d+1}-y^{d+1}}{(x-y)(d+2)!}=\\
=
\sum_{d=1}^\infty \frac{x^{d}-y^{d}}{(x-y)(d+1)!}=
\frac{(e^x-1)/x-(e^y-1)/y}{x-y}=\frac{ye^x-y-xe^y+x}{(x-y)xy}.
\end{multline}
Using (\ref{equv1}), (\ref{equv2}) and (\ref{equv3}) with (\ref{eqgenfun}) yields
\begin{multline}
\sum_{d=0}^\infty \sum_{r=0}^d \frac{p_{\rm M}(d,r)}{d!}x^{d-r}y^r=\\
=
\frac{xe^x-ye^y}{x-y}+xy\left(\frac{e^x-e^y}{x-y}\right)^2\sum_{k=2}^\infty
\left(\frac{ye^x-y-xe^y+x}{x-y}\right)^{k-2}=\\
=\frac{xe^x-ye^y}{x-y}+\left(\frac{e^x-e^y}{x-y}\right)^2\frac{xy}{\displaystyle 1-
\frac{ye^x-y-xe^y+x}{x-y}}=\\
=
\frac{xe^x-ye^y}{x-y}+\frac{xy(e^x-e^y)^2}{(x-y)(xe^y-ye^x)}=\frac{x-y}{xe^{-x}-ye^{-y}}.
\end{multline}
\end{proof14c}

The values of $p_{\rm (P)M}(d,r)$ can be found in~\ref{apB}.


\section{Valuative invariants: the groups $P_{\rm M}^{\rm sym},P_{\rm PM}^{\rm sym},P_{\rm MM}^{\rm sym}$}

\label{not:Y}
Let $Y_{\rm MM}(d,r)$ be the group generated by all $\langle\rk\rangle-\langle\rk\circ\sigma\rangle$ where
$\rk:2^{\underline{d}}\to \Z\cup \{\infty\}$ is a megamatroid of rank $r$ and $\sigma$ is a permutation of $\underline{d}$.
\label{not:Z^sym}
We define $Z_{\rm MM}^{\rm sym}(d,r)=Z_{\rm MM}(d,r)/Y_{\rm MM}(d,r)$.
\label{not:pi}
Let $\pi_{\rm MM}:Z_{\rm MM}(d,r)\to Z_{\rm MM}^{\rm sym}(d,r)$ be the quotient homomorphism.
If $\rk_X:2^X\to \Z\cup \{\infty\}$ is any megamatroid, then we can choose a bijection $\varphi:\underline{d}\to X$, where
$d$ is the cardinality of $X$. 
Let $r=\rk_X(X)$. 
The image of $\langle \rk_X\circ \varphi\rangle$ in $Z_{\rm MM}^{\rm sym}(d,r)$ does not
depend on $\varphi$, and will be denoted by $[\rk_X]$.
The megamatroids $(X,\rk_X)$ and $(Y,\rk_Y)$ are isomorphic if and only if $[\rk_X]=[\rk_Y]$.
So we may think of $Z_{\rm MM}^{\rm sym}(d,r)$ as the free group generated by all isomorphism classes of rank $r$ megamatroids
on sets with $d$ elements.

\label{not:B}
Let $B_{\rm MM}(d,r)$ be the group generated by all $[Q(\rk)]-[Q(\rk\circ\sigma)]$ 
where $\rk:2^{\underline{d}}\to \Z\cup \{\infty\}$ is a megamatroid of rank $r$ 
and $\sigma$ is a permutation of $\underline{d}$.
\label{not:P^sym}
Define $P_{\rm MM}^{\rm sym}(d,r)=P_{\rm MM}(d,r)/B_{\rm MM}(d,r)$ 
\label{not:rho}
and let $\rho_{\rm MM}:P_{\rm MM}(d,r)\to P_{\rm MM}^{\rm sym}(d,r)$
be the quotient homomorphism.
From the definitions it is clear that $\Psi_{\rm MM}(Y_{\rm MM}(d,r))=B_{\rm MM}(d,r)$.
\label{not:mc G}
Therefore, there exists a unique group homomorphism 
$$\Psi^{\rm sym}_{\rm MM}:Z_{\rm MM}^{\rm sym}(d,r)\to P_{\rm MM}^{\rm sym}(d,r)$$
such that the following diagram commutes:
\begin{equation}\label{eqcomsquare}
\xymatrix{
Z_{\rm MM}(d,r)\ar[r]^{\Psi_{\rm MM}}\ar[d]_{\pi_{\rm MM}} & P_{\rm MM}(d,r)\ar[d]^{\rho_{\rm MM}} \\
Z_{\rm MM}^{\rm sym}(d,r)\ar[r]_{\Psi^{\rm sym}_{\rm MM}} & P_{\rm MM}^{\rm sym}(d,r).
}
\end{equation}
This diagram is a push-out. Define $Y_{\rm (P)M}(d,r)=Y_{\rm MM}(d,r)\cap Z_{\rm (P)M}(d,r)$.
The group $Y_{\rm (P)M}(d,r)$ is the
group generated by all $\langle \rk\rangle-\langle\rk\circ\sigma\rangle$
where $\rk:2^{\underline{d}}\to \N$ is a (poly)matroid
of rank $r$ and $\sigma$ is a permutation of $\underline{d}$. 
Define $Z_{\rm (P)M}^{\rm sym}(d,r)=Z_{\rm (P)M}(d,r)/Y_{\rm (P)M}(d,r)$.
The group $Z_{\rm (P)M}^{\rm sym}(d,r)$ is freely generated by all
$[\rk]$ where $\rk:X\to \N$ is a $*$matroid of rank $r$ and $d=|X|$.

Define $B_{\rm *M}(d,r)$ as the
group generated by all $[Q(\rk)]-[Q(\rk\circ \sigma)]$
where $\rk:2^{\underline{d}}\to \N$ is a $*$matroid
of rank $r$ and $\sigma$ is a permutation of $\underline{d}$.
Let $P_{\rm *M}^{\rm sym}(d,r)=P_{\rm *M}(d,r)/B_{\rm *M}(d,r)$.

\begin{lemma}
We have 
$$B_{\rm (P)M}(d,r)=B_{\rm MM}(d,r)\cap P_{\rm (P)M}(d,r).$$
\end{lemma}

\begin{proof}
Define $q_{\rm (P)M}:P_{\rm MM}(d,r)\to P_{\rm (P)M}(d,r)$ by
$q_{\rm (P)M}(f)=f\cdot [\Delta_{\rm (P)M}(d,r)]$. This is well defined because
for any megamatroid $\rk:2^{\underline{d}}\to \Z\cup \{\infty\}$ of rank $r$, we have
$q_{\rm (P)M}([Q(\rk)])=[Q(\rk)]\cdot [\Delta_{\rm (P)M}(d,r)]=[Q(\rk)\cap \Delta_{\rm (P)M}(d,r)]$
and $Q(\rk)\cap \Delta_{\rm (P)M}(d,r)$ is either empty or a polymatroid polyhedron.
Clearly, $q_{\rm (P)M}$ is a projection of $P_{\rm MM}(d,r)$ onto $P_{\rm (P)M}(d,r)$.
Since $q_{\rm (P)M}(B_{\rm MM}(d,r))\subseteq  B_{\rm (P)M}(d,r)$, it follows that
$$
B_{\rm MM}(d,r)\cap P_{\rm (P)M}(d,r)=q_{\rm PM}(B_{\rm MM}(d,r)\cap P_{\rm (P)M}(d,r))\subseteq B_{\rm (P)M}(d,r).
$$
It follows that $B_{\rm MM}(d,r)\cap P_{\rm (P)M}(d,r)=B_{\rm (P)M}(d,r)$. 
\end{proof}

By restriction, we get also the commutative push-out diagrams (\ref{eq:pushout}) from the introduction.
\label{not:mf a}
Define ${\mathfrak p}^{\rm sym}_{\rm *M}(d,r)$\label{not:mf psym} as the set of all pairs 
$(\underline{X},\underline{r})\in {\mathfrak p}_{\rm *M}(d,r)$ 
such that for every $j$, there exists an $i$ such that
$$
X_j=\underline{i}=\{1,2,\dots,i\}.
$$

\label{not:A}
We define $A_{\rm *M}(d,r)$ 
 as the $\Z$ module generated by all $[R_{\rm *M}(\underline{X},\underline{r})]$ 
with $(\underline{X},\underline{r})\in {\mathfrak p}^{\rm sym}_{\rm *M}(d,r)$.

\begin{lemma}\label{lemPAB}
We have
$$P_{\rm *M}(d,r)=A_{\rm *M}(d,r)\oplus B_{\rm *M}(d,r).$$
\end{lemma}

\begin{proof}
By the definitions of $A_{\rm *M}(d,r)$ and $B_{\rm *M}(d,r)$ it is clear that
$P_{\rm *M}(d,r)=A_{\rm *M}(d,r)+B_{\rm *M}(d,r)$.
Consider the homomorphism $\tau:P_{\rm *M}(d,r)\to P_{\rm *M}(d,r)$ defined by $\tau(f)=\sum_{\sigma} f\circ \sigma$
where $\sigma$ runs over all permutations of $\underline{d}$.
Clearly, $B_{\rm *M}(d,r)$ is contained in the kernel of $\tau$.
Recall that $[R_{\rm *M}(\underline{X},\underline{r})]$,
$(\underline{X},\underline{r})\in {\mathfrak p}_{\rm *M}(d,r)$ is a basis of $P_{\rm *M}(d,r)$.
From this it easily follows that the set\linebreak
$\tau([R_{\rm *M}(d,r)])$, $(\underline{X},\underline{r})\in {\mathfrak p}^{\rm sym}_{\rm *M}(d,r)$
is independent over $\Q$.
Therefore the restriction of $\tau$ to $A_{\rm *M}(d,r)$ is injective and $A_{\rm *M}(d,r)\cap B_{\rm *M}(d,r)=\{0\}$.
\end{proof}
\begin{theorem}\label{theo:Psymgens}
The $\Z$-module $P_{\rm \star M}^{\rm sym}(d,r)$ is freely generated by all $\rho_{\rm \star M}([R_{\star\rm M}(\underline{X},\underline r)])$
with $(\underline{X},\underline{r})\in {\mathfrak p}^{\rm sym}_{\star \rm M}(d,r)$.
\end{theorem}
\begin{proof}
It is clear that $\rho_{\rm *M}(A_{\rm *M}(d,r))=P_{\rm *M}^{\rm sym}(d,r)$. So the restriction is surjective.  It is also injective by Lemma~\ref{lemPAB}. So
the restriction of $\rho_{\rm *M}:P_{\rm *M}(d,r)\to P_{\rm *M}^{\rm sym}(d,r)$
to $A_{\rm *M}(d,r)$ is an isomorphism. From the definition of $A_{\star M}(d,r)$ it follows that the given set generates
$P_{\rm \star M}^{\rm sym}(d,r)$, and the set is independent because of Theorem~\ref{free generation}. 
\end{proof}

The matroid polytopes $R_{\rm M}(\underline X,\underline r)$
are the polytopes of {\em Schubert matroids} and their images
under relabeling the ground set.  
Schubert matroids were first described by Crapo~\cite{Crapo}, and have since 
arisen in several contexts.  
So Theorem~\ref{theo:Psymgens} says that the indicator functions 
of Schubert matroids form a basis for $P^{\rm sym}_{\rm M}(d,r)$.  

Recall that $Z^{\rm sym}_{\rm *M}$ can be viewed as the free $\Z$-module
generated by all isomorphism classes of $*$matroids
on a set with $d$ elements of rank $r$.
We say that a group homomorphism $f:Z^{\rm sym}_{\rm MM}(d,r)\to A$ 
is valuative if and only if
$f\circ \pi_{\rm MM}$ is valuative.
%
%
For any $(\underline X,\underline r)\in\mathfrak p_{\rm MM}(d,r)$ 
and $\sigma$ a permutation of~$\underline d$, we have 
$s_{\underline X,\underline r}(\rk\circ\sigma)=
 s_{\sigma\underline X,\underline r}(\rk)$, where
$\sigma$ acts on~$\underline X$ by permuting each set in the chain.
So the symmetric group $\Sigma_d$ acts naturally on $P_{\rm *M}(d,r)$.
It is easy to see that
$$
P_{\star M}^{\rm sym}(d,r)^\vee\cong \big(P_{\star M}(d,r)^\vee\big)^{\Sigma_d},
$$
where the right-hand side is the set of $\Sigma_d$-invariant elements of $P_{\star\rm M}(d,r)^\vee$.

For $(\underline X,\underline r)\in\mathfrak p^{\rm sym}_{\rm MM}(d,r)$,
define a homomorphism \label{not:s^sym}
$s^{\rm sym}_{\underline X,\underline r}:Z_{\rm MM}(d,r)\to\Z$ by
$$s^{\rm sym}_{\underline X,\underline r} = 
\sum_{\sigma\underline X} s_{\sigma\underline X,\underline r}$$
where the sum is over all chains $\sigma\underline X$
in the orbit of~$\underline X$ under the action of the symmetric group.
Then Corollary~\ref{cor s univ} implies the following.

\begin{corollary}\label{cor ssym}
The $\Q$-vector space $P^{\rm sym}_{\rm (P)M}(d,r)^\vee \otimes_{\Z}\Q$ of valuations $Z^{\rm sym}_{\rm (P)M}(d,r)\to\Q$
has a basis given by the functions 
$s^{\rm sym}_{\underline X,\underline r}$
for $(\underline X,\underline r)\in\mathfrak p^{\rm sym}_{\rm (P)M}(d,r)$.
\end{corollary}

For a sequence $\alpha=(\alpha_1,\dots,\alpha_d)$ of nonnegative integers with $|\alpha|=\sum_i \alpha_i=r$, we define
$$
u_{\alpha}=s_{\underline{X},\underline{r}} :Z^{\rm sym}_{\rm (P)M}(d,r)\to \Z,
$$\label{not:ualpha}
where $\underline{X}_i=\underline{i}$ for $i=1,2,\dots,r$ and $\underline{r}=(\alpha_1,\alpha_1+\alpha_2,\dots,\alpha_1+\cdots+\alpha_d)$. 
Parallel to Corollary~\ref{cor maximal chains} we also have the following.
\begin{corollary}\label{cor:PPMsymdual}
The $\Q$-vector space  $P^{\rm sym}_{\rm PM}(d,r)^\vee\otimes_{\Z}\Q$ of valuations $Z^{\rm sym}_{\rm PM}(d,r)\to\Q$ 
has a $\Q$-basis given by the functions~$u_{\alpha}$, where $\alpha$ runs over all sequences $(\alpha_1,\dots,\alpha_d)$ of
nonnegative integers with $|\alpha|=r$.
\end{corollary}
\begin{corollary}\label{cor:PMsymdual}
The $\Q$-vector space $P^{\rm sym}_{\rm M}(d,r)^\vee\otimes_\Z\Q$ of valuations $Z^{\rm sym}_{\rm M}(d,r)\to \Q$
has a $\Q$-basis given by all functions $u_{\alpha}$ where $\alpha$ runs over all sequences $(\alpha_1,\dots,\alpha_d)\in \{0,1\}^d$
with $|\alpha|=r$.
\end{corollary}
\begin{proof11}
From the definitions of the $U_\alpha$ 
and the $u_{\alpha}$, it follows
that $u_\alpha(\langle \rk\rangle)$ is the coefficient of $U_\alpha$ in
${\mathcal G}(\langle \rk\rangle)$. In other words, $\{u_{\alpha}\}$ is a dual
basis to $\{U_{\alpha}\}$. The universality follows from Corollary~\ref{cor:PMsymdual}.
\end{proof11}

The rank of $P_{\rm (P)M}^{\rm sym}(d,r)$ is equal to the cardinality 
of ${\mathfrak p}^{\rm sym}_{\rm (P)M}(d,r)$.
If $(\underline{X},\underline{r})$ and $\ell(\underline{X})=k$ lies in ${\mathfrak p}^{\rm sym}_{\rm (P)M}(d,r)$ 
then $\underline{X}$ is completely determined by the numbers $s_i:=|X_i|$, 
$1\leq i\leq k$.
\begin{proof14b}
Given $k$, there are ${r\choose k-1}$ ways of choosing $\underline{r}=(r_1,\dots,r_k)$ 
with $0<r_1<r_2<\cdots<r_k=r$ and ${d-1\choose k-1}$ ways of choosing $(s_1,\dots,s_k)$ with
$0<s_1<s_2<\cdots<s_k=d$. So the cardinality of ${\mathfrak p}_{\rm PM}^{\rm sym}(d,r)$
is
$$
\sum_{k\geq 1}{r\choose k-1}{d-1\choose k-1}=
\sum_{k\geq 0}{r\choose k}{d-1\choose k}={r+d-1\choose r}.
$$
$$
\sum_{r,d}p_{\rm PM}^{\rm sym}(d,r)x^dy^r=\sum_{r,d}{\textstyle {r+d-1\choose r}}x^dy^r=\sum_{d}(1-x)^{-d}y^d=\frac{1}{1-\frac{y}{1-x}}=\frac{1-x}{1-x-y}.
$$

\end{proof14b}
\begin{proof14a}
Let $t_i=s_i-r_i$. Then we have $0<t_1<t_2<\cdots<t_{k-1}\leq t_k=d-r$. Given $k$, there are ${r\choose k-1}$ ways
of choosing $\underline{r}$ such that $0\leq r_1<\cdots<r_k=r$ and ${d-r\choose k-1}$ ways of choosing $(t_1,\dots,t_k)$ with
$0<t_1<\cdots<t_{k-1}\leq t_k=d-r$. So the cardinality of ${\mathfrak p}^{\rm sym}_{\rm M}(d,r)$
is 
$$
\sum_{k\geq 1}{r\choose k-1}{d-r\choose k-1}=
\sum_{k\geq 0}{r\choose k}{d-r\choose k}={d\choose r}.
$$
So we have
$$
\sum_{r,d}p_{\rm M}^{\rm sym}(d,r)x^{d-r}y^r=\sum_{d}(x+y)^d=\frac{1}{1-x-y}.
$$
\end{proof14a}

\begin{example}\label{expolymatroid}
Consider polymatroids for $r=2$ and $d=3$. All polymatroid base polytopes
are contained in the triangle
$$
\{(y_1,y_2,y_3)\in \R^3\mid y_1+y_2+y_3=2,\ y_1,y_2,y_3\geq 0\}.
$$

\leaveout{
\centerline{ 
\includegraphics[width=1.5in]{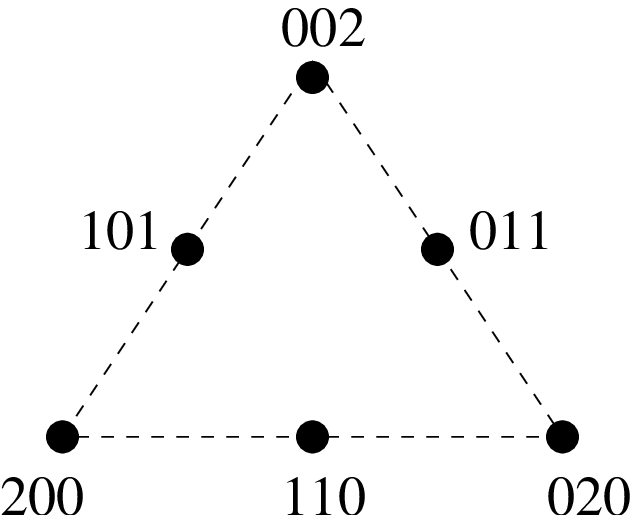}
}}

There are ${d-1+r\choose r}={4\choose 2}$ elements in ${\mathfrak p}_{\rm PM}^{\rm sym}(3,2)$ and the polytopes
$R(\underline{X},\underline{r})$, $(\underline{X},\underline{r})\in {\mathfrak p}^{\rm sym}_{\rm PM}(3,2)$ are given by:

\leaveout{
\centerline{ 
\includegraphics[width=3in]{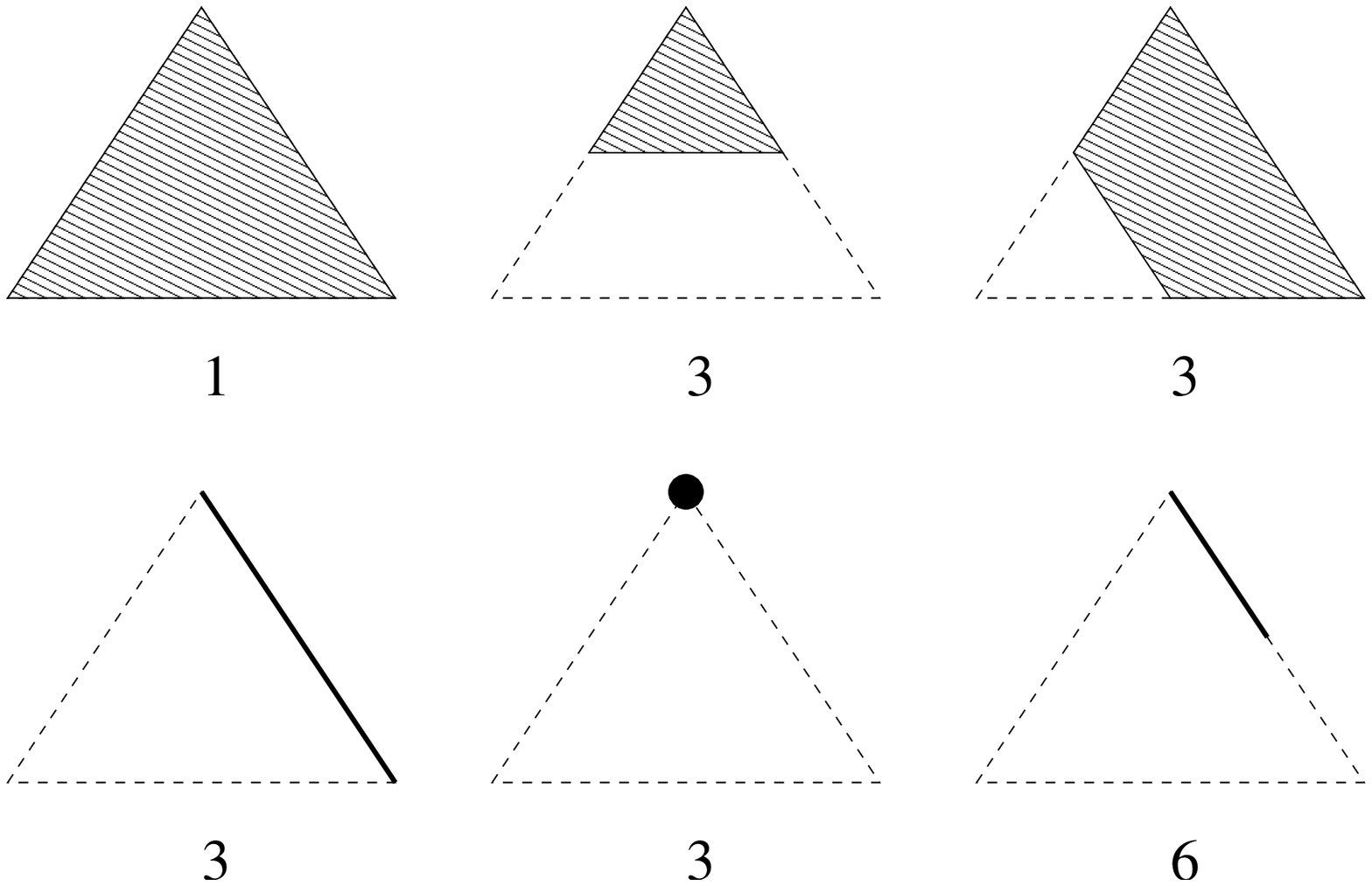}
}}

These $6$ polytopes correspond to the following pairs $(\underline{X},\underline{r})\in {\mathfrak p}_{\rm PM}(3,2)$.
$$
\small
\begin{array}{|rl|rl|rl|} \hline
\underline{X}:& \{1,2,3\} & \underline{X}: & \{1,2\}\subset \{1,2,3\}& \underline{X}:  & \{1\}\subset \{1,2,3\} \\\
\underline{r}=& (2) & \underline{r}= & (1,2) & \underline{r}= & (1,2) \\ \hline
\underline{X}:& \{1\} \subset \{1,2,3\} & \underline{X}: & \{1,2\} \subset \{1,2,3\}& \underline{X}: &  \{1\}\subset \{1,2\}\subset \{1,2,3\} \\
\underline{r}=& (0,2) & \underline{r}= & (0,2) & \underline{r}= & (0,1,2) \\ \hline
\end{array}
$$

The symmetric group $\Sigma_3$ acts on the triangle by permuting the coordinates $y_1,y_2,y_3$.

If $\Sigma_3$ acts on the generators $R(\underline{X},\underline{r})$
with $(\underline{X},\underline{r})\in {\mathfrak p}_{\rm PM}(3,2)$, then we get
all $R(\underline{X},\underline{r})$ with $(\underline{X},\underline{r})\in {\mathfrak p}_{\rm PM}(3,2)$.
In the figure, we wrote for each polytope the cardinality of the orbit under $\Sigma_3$.
The cardinality of ${\mathfrak p}_{\rm PM}(3,2)$ is $1+3+3+3+3+6=19$. This is consistent with Theorem~\ref{theo1.4}, because
the cardinality is $(r+1)^d-r^d=3^3-2^3=19$.

\end{example}
\begin{example}\label{exmatroid}
Consider matroids for $r=2$ and $d=4$. All matroid base polytopes are contained
in the set 
$$
\{(y_1,y_2,y_3,y_4)\in \R^4\mid y_1+y_2+y_3+y_4=2,\ \forall i\ 0\leq y_i\leq 1\}.
$$
This set is an octahedron:

\leaveout{
\centerline{ 
\includegraphics[width=1.5in]{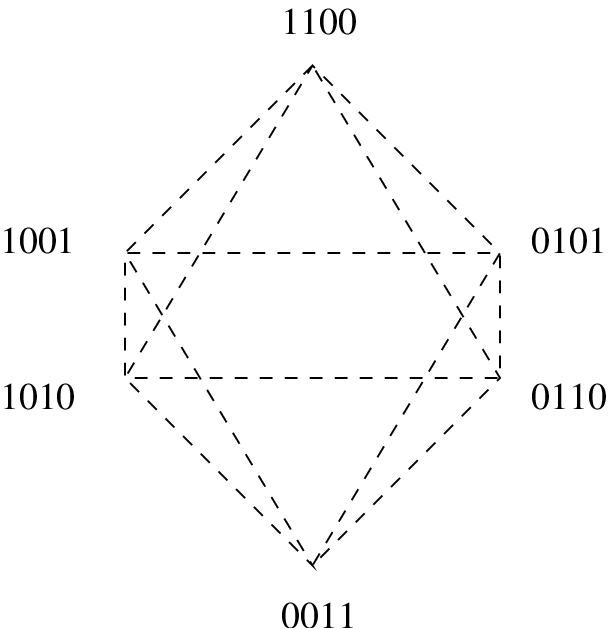}
}}

There are ${d\choose r}={4\choose 2}$ elements in ${\mathfrak p}_{\rm M}(4,2)$, and the polytopes
$R_{\rm M}(\underline{X},\underline{r})$, $(\underline{X},\underline{r})\in {\mathfrak p}_{\rm M}(4,2)$ are given by:

\leaveout{
\centerline{ 
\includegraphics[width=3in]{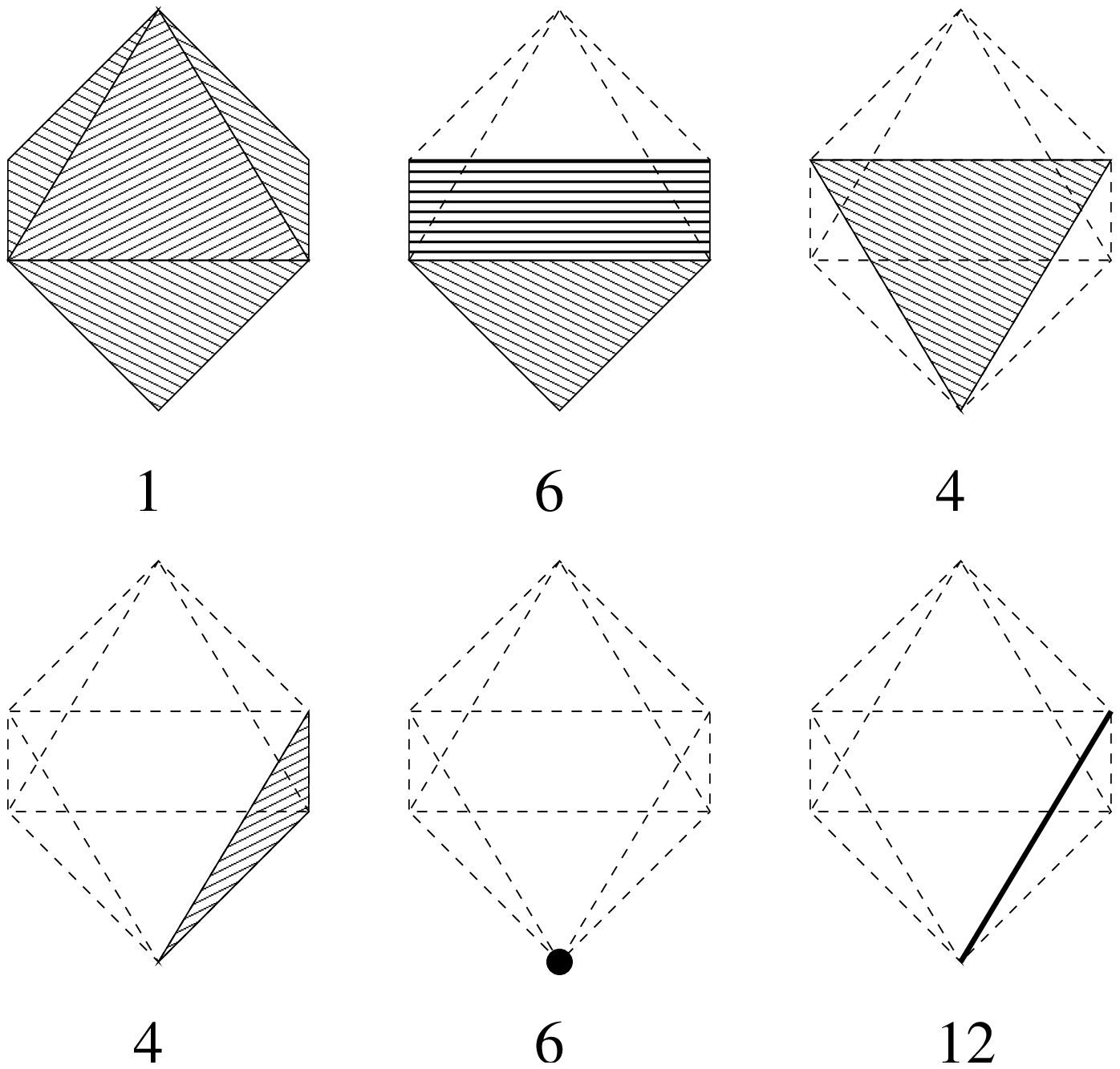}
}}
These $6$ polytopes correspond to the following pairs $(\underline{X},\underline{r})\in {\mathfrak p}_{\rm M}(4,2)$.
$$
\small\begin{array}{|rl|rl|rl|} \hline
\underline{X}:& \{1,2,3,4\} & \underline{X}: & \{1,2\}\subset \{1,2,3,4\}& \underline{X}:  & \{1,2,3\}\subset \{1,2,3,4\} \\\
\underline{r}=& (2) & \underline{r}= & (1,2) & \underline{r}= & (1,2) \\ \hline
\underline{X}:& \{1\} \subset \{1,2,3,4\} & \underline{X}: & \{1,2\} \subset \{1,2,3,4\}& \underline{X}: &  \{1\}\subset \{1,2,3\}\subset \{1,2,3,4\} \\
\underline{r}=& (0,2) & \underline{r}= & (0,2) & \underline{r}= & (0,1,2) \\ \hline
\end{array}
$$

The symmetric group $\Sigma_4$ acts by permuting the coordinates $y_1,y_2,y_3,y_4$.
This group acts on the octahedron, but it is not the full automorphism group of the
octahedron. Also note that not all elements of $\Sigma_4$ preserve the orientation.
If $\Sigma_4$ acts on the generators $R_{\rm M}(\underline{X},\underline{r})$
with $(\underline{X},\underline{r})\in {\mathfrak p}_{\rm M}(4,2)$, then we get
all $R(\underline{X},\underline{r})$ with $(\underline{X},\underline{r})\in {\mathfrak p}_{\rm M}(4,2)$.
In the figure, we write for each polytope the cardinality of the orbit under $\Sigma_4$.
The cardinality of ${\mathfrak p}_{\rm M}(4,2)$ is $1+6+4+4+6+12=33$, which is
compatible with Theorem~\ref{theo1.4} and the table in~\ref{apB}.
Besides the polytopes $R(\underline{X},\underline{r})$, $(\underline{X},\underline{r})\in {\mathfrak p}_{\rm M}(4,2)$, 
there are 3 more matroid base polytopes (belonging to
isomorphic matroids), but these decompose as follows.

\centerline{ 
\includegraphics[width=4in]{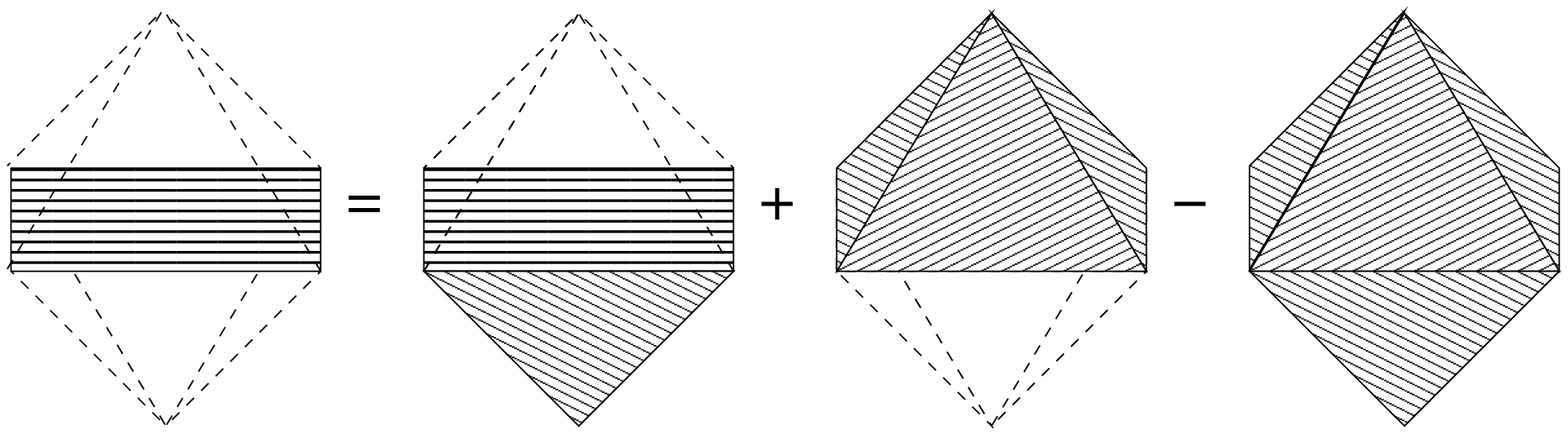}
}
\end{example}

\section{Hopf algebra structures}\label{sec:Hopf}
\label{not:no (d,r)}
Define $Z_{\rm *M}=\bigoplus_{d,r}Z_{\rm *M}(d,r)$, and in a similar way define 
$Z_{\rm *M}^{\rm sym}$, $P_{\rm *M}$, and $P_{\rm *M}^{\rm sym}$.
We can view $Z_{\rm *M}$
as the $\Z$-module freely generated by all isomorphism classes of $*$matroids.

If $\rk_1:2^{\underline{d}}\to \Z\cup \{\infty\}$
and $\rk_2:2^{\underline{e}}\to \Z\cup \{\infty\}$ then we define
$$
\rk_1\boxplus \rk_2:2^{\underline{d+e}}\to \Z\cup \{\infty\}
$$
by
$$
(\rk_1\boxplus \rk_2)(A)=\rk_1(A\cap \underline{d})+\rk_2(\{i\in \underline{e}\mid d+i\in A\})
$$
for any set $A\subseteq \underline{d+e}$. Note that $\boxplus$ is not commutative.
We have a homomorphism
$$
\nabla:Z_{\rm MM}(d,r)\otimes_{\Z} Z_{\rm MM}(e,s)\to Z_{\rm MM}(d+e,r+s).
$$
defined by
$$
\nabla(\langle \rk_1\rangle\otimes \langle \rk_2\rangle)=\langle \rk_1\boxplus \rk_2\rangle.
$$
The multiplication\label{not:mult} $\nabla:Z_{\rm MM}\otimes_{\Z} Z_{\rm MM}\to Z_{\rm MM}$ 
makes $Z_{\rm MM}(d,r)$ into an associative (noncommutative) ring with 1.
The unit $\eta:\Z\to Z_{\rm MM}(d,r)$ \label{not:unit} is given by $1\mapsto \langle \rk_0\rangle$ where $\rk_0:2^{\underline{0}}\to \Z\cup \{\infty\}$
is the unique megamatroid defined by $\rk(\emptyset)=0$.
With this multiplication, $Z_{\rm M}(d,r)$ and $Z_{\rm PM}(d,r)$ are subrings of $Z_{\rm MM}(d,r)$.
The multiplication also respects the bigrading of $Z_{\rm MM}(d,r)$.

Next, we define a comultiplication for $Z_{\rm MM}$. Suppose
that $X=\{i_1,i_2,\dots,i_d\}$ is a set of integers with $i_1<\cdots<i_d$ and
$\rk:2^X\to \Z\cup \{\infty\}$ is a megamatroid. We define a megamatroid $\widehat{\rk}:2^{\underline{d}}\to \Z\cup \{\infty\}$ 
by $\widehat{\rk}(A)=\rk(\{i_j\mid j\in A\})$. 
If $\rk:2^{X}\to \Z\cup \{\infty\}$ is a megamatroid and $B\subseteq A\subseteq X$ then we define
$\rk_{A/B}:2^{A\setminus B}\to \Z\cup \{\infty\}$ by $\rk_{A/B}(C)=\rk(B\cup C)-\rk(B)$ 
for all $C\subseteq A\setminus B$. We also define $\rk_A:=\rk_{A/\emptyset}$ and $\rk_{/B}=\rk_{X/B}$.

We now define\label{not:comult}
$$
\Delta:Z_{\rm MM}\to Z_{\rm MM}\otimes_\Z Z_{\rm MM} 
$$
by 
$$
\Delta(\langle\rk\rangle)=\sum_{A\subseteq \underline{d};\ \rk(A)<\infty} \langle \widehat{\rk_A}\rangle\otimes
\langle\widehat{\rk_{/A}}\rangle.
$$
where $A$ runs over all subsets of $\underline{d}$ for which $\rk(A)$ is finite.
This comultiplication is coassociative, but not cocommutative. If $\rk:2^{\underline{d}}\to \Z\cup \{\infty\}$
is a megamatroid, then  the counit is defined by
$$
\epsilon(\langle\rk\rangle)=\left\{\begin{array}{ll}
1 & \mbox{if $d=0$;}\\
0 & \mbox{otherwise.}
\end{array}\right.
$$
The reader may verify that the multiplicative and comultiplicative structure are compatible, making
$Z_{\rm MM}$ into an bialgebra. Note that $\Delta$ also restricts to comultiplications
for $Z_{\rm PM}$ and $Z_{\rm M}$, and $Z_{\rm PM}$ and $Z_{\rm M}$ are sub-bialgebras of $Z_{\rm MM}$.

We define a group homomorphism $S:Z_{\rm MM}\to Z_{\rm MM}$ by
$$ 
S(\langle\rk\rangle)=\sum_{r=1}^d(-1)^{r}\sum_{\scriptstyle\underline{X};\ \ell(\underline{X})=r,\atop \scriptstyle\rk(X_1)<\infty,\dots,\rk(X_r)<\infty}
\prod_{i=1}^r\;\langle\widehat{\rk_{X_i/X_{i-1}}}\rangle.
$$
Here we use the convention $X_0=\emptyset$. One can check that  $S$ makes $Z_{\rm MM}$ into a Hopf algebra.
Restriction of $S$ makes $Z_{\rm M}$ and $Z_{\rm PM}$ into sub-Hopf algebras of $Z_{\rm MM}$.
We conclude that $Z_{\rm *M}$ has the structure of bigraded Hopf algebras over $\Z$.

It is well-known that $Z_{\rm M}^{\rm sym}$ has the structure of a Hopf algebra over $\Z$. Similarly
we have that $Z_{\rm MM}^{\rm sym}$ and $Z_{\rm PM}^{\rm sym}$ have a Hopf algebra structure.
The multiplication
$$
\nabla:Z_{\rm MM}^{\rm sym}\otimes_{\Z} Z_{\rm MM}^{\rm sym}\to Z_{\rm MM}^{\rm sym}
$$
is defined by
$$
\nabla([\rk_1]\otimes [\rk_2])=[\rk_1\oplus \rk_2].
$$
The comultiplication is defined by
$$
\Delta([\rk])=\sum_{A\subseteq X;\rk(A)<\infty} [(A,\rk_A)]\otimes [(X\setminus A,\rk_{/A})]
$$
for any megamatroid $\rk:2^X\to \Z\cup \{\infty\}$.
The unit $\eta:\Z\to Z_{\rm MM}^{\rm sym}$ is given by $1\mapsto [(\emptyset,\rk_0)]$ and
the counit $\epsilon: Z_{\rm MM}^{\rm sym}\to\Z$ \label{not:counit} is defined by 
$$
\epsilon([(X,\rk)])=\left\{
\begin{array}{ll}
1 & \mbox{if $X=\emptyset$;}\\
0 & \mbox{otherwise.}
\end{array}
\right.
$$

Finally, we define the antipode $S:Z_{\rm MM}^{\rm sym}\to Z_{\rm MM}^{\rm sym}$ \label{not:antipode} by
$$ 
S([\rk])=\sum_{r=1}^d(-1)^{r}\sum_{\scriptstyle\underline{X};\ \ell(\underline{X})=r,\atop\scriptstyle\rk(X_1)<\infty,\dots,\rk(X_r)<\infty} 
\prod_{i=1}^r\;[(X_i\setminus X_{i-1},\rk{X_i/X_{i-1}})].
$$
From the definitions, it is clear that the $\pi_{\rm *M}$ are Hopf algebra morphisms. 

The space $P_{\rm MM}$ inherits a Hopf algebra structure from $Z_{\rm MM}$. We define the multiplication $\nabla:P_{\rm MM}\otimes P_{\rm MM}\to P_{\rm MM}$ by
\begin{equation}\label{eq:nabla}
\nabla([\Pi_1]\otimes [\Pi_2])=[\Pi_1\times \Pi_2].
\end{equation}
It is easy to verify that $\nabla\circ(\Psi_{\rm MM}\otimes \Psi_{\rm MM})=\Psi_{\rm MM}\circ \nabla$.

To define the comultiplication $\Delta:P_{\rm MM}\to P_{\rm MM}\otimes P_{\rm MM}$, we would like to have that
$(\psi_{\rm MM}\otimes \psi_{\rm MM})\otimes \Delta=\Delta\circ \psi_{\rm MM}$.
So for a megamatroid polytope $\rk:2^{\underline{d}}\to \Z\cup\{\infty\}$ we would like to have
\begin{multline*}
\Delta([Q(\rk)])=\Delta(\psi_{\rm MM}(\langle \rk\rangle))=\sum_{A\subseteq \underline{d};\rk(A)<\infty}
\psi_{\rm MM}(\widehat{\rk_{A}})\otimes \psi_{\rm MM}(\widehat{\rk_{/A}})=\\
=
\sum_{A\subseteq \underline{d};\rk(A)<\infty}
[Q(\rk_A)]\otimes [Q(\rk_{/A})].
\end{multline*}

A basis of $P_{\rm MM}$ is given by all $R_{\rm MM}(\underline{X},\underline{r})$, with $(\underline{X},\underline{r})\in {\mathfrak p}_{\rm MM}=\bigcup_{d,r}{\mathfrak p}_{\rm MM}(d,r)$. Recall that the rank function $\rk_{\underline{X},\underline{r}}$ is
defined such that $Q(\rk_{\underline{X},\underline{r}})=R_{\rm MM}(\underline{X},\underline{r})$.
We have that $\rk_{\underline{X},\underline{r}}(A)<\infty$ if and only if $A=X_i$ for some $i$.
In this case we have
$$
\Delta(\langle\rk_{\underline{X},\underline{r}}\rangle)=\sum_{i=0}^k \langle\widehat{\rk_{\underline{X}_{i},\underline{r}_i}}\rangle\otimes
\langle\widehat{\rk_{\underline{X}^i,\underline{r}^i}}\rangle,
$$
where 
$$
\underline{X}_i:\emptyset\subset X_1\subset \cdots\subset X_i,\quad \underline{r}_i=(r_1,r_2,\dots,r_i)
$$
$$
\underline{X}^i:\emptyset\subset X_{i+1}\setminus X_i\subset \cdots \subset X_k\setminus X_{i},\quad
\underline{r}^i=(r_{i+1}-r_i,\dots,r_k-r_i).
$$
We define $\Delta$ by
$$
\Delta([R_{\rm MM}(\underline{X},\underline{r})])=\sum_{i=0}^k [R_{\rm MM}\widehat{(\underline{X}_i,\underline{r}_i)}]\otimes
[R_{\rm MM}\widehat{(\underline{X}^i,\underline{r}^i)}].
$$
From this definition and Theorem~\ref{theo1} follows that
\begin{multline}\label{eq:Delta}
\Delta([Q(\rk)])=\sum_{\underline{X}}(-1)^{d-\ell(\underline{X})}\Delta[R_{\rm MM}(\underline{X},\rk)]=\\
=
\sum_{\underline{X}}\sum_{i=0}^{\ell(\underline{X})}(-1)^{|X_i|-i}[R_{\rm MM}\widehat{(\underline{X}_i,\rk_{X_{i}})}]\otimes (-1)^{d-|X_i|-\ell(\underline{X})+i}
[R_{\rm MM}\widehat{(\underline{X}^i,\rk_{/X_{i}})}]=\\
=
\sum_{A\subseteq \underline{d};\rk(A)<\infty}
[Q(\widehat{\rk}_A)]\otimes [Q(\widehat{\rk}_{/A})].
\end{multline}
In a similar fashion we can define the antipode $S:P_{\rm MM}\to P_{\rm MM}$.

The Hopf algebra structure on $P_{\rm MM}$ naturally induces a Hopf algebra structure on $P_{\rm MM}^{\rm sym}$
such that $\rho_{\rm MM}$ and $\Psi^{\rm sym}_{\rm MM}$ are Hopf algebra homomorphisms. 
Also $P_{\rm PM}$ is a Hopf subalgebra of $P_{\rm MM}$ and $P_{\rm M}$ is a Hopf subalgebra of $P_{\rm PM}$.
Similarly $P_{\rm PM}^{\rm sym}$ is a Hopf subalgebra of $P_{\rm MM}^{\rm sym}$,
and $P_{\rm M}^{\rm sym}$ is a Hopf subalgebra of $P_{\rm PM}^{\rm sym}$.

As a first observation to motivate the consideration of these Hopf algebra structures,
we consider multiplicative invariants.  
\begin{definition}
A multiplicative invariant for $\ast$matroids with values in a commutative ring $A$ (with 1) is
a ring homomorphism $f:Z^{\rm sym}_{\rm *M}\to A$.
\end{definition}
That is to say, $f$ is multiplicative if $f(\rk_1\oplus\rk_2)=f(\rk_1)f(\rk_2)$.
This is exactly the condition that $f$ be a group-like element of the graded dual algebra
$P^{\rm sym}_{\rm M}(d,r)^\#$.  
Many (poly)matroid invariants of note have this property, 
for instance the Tutte polynomial.  

\begin{proposition}\label{prop:Tutte}
The Tutte polynomial $\mathcal T\in P^{\rm sym}_{\rm M}(d,r)^\#$ is given by
\begin{equation}\label{eq:tutte}
\mathcal T = e^{(y-1)u_0+u_1}e^{u_0+(x-1)u_1}.
\end{equation}
\end{proposition}

\begin{proof} 
Recall the definition of~$u_\alpha$ in terms of rank conditions on a chain of sets.
In view of~\eqref{eq:Delta}, we have that 
the multiplication in $(P_{\rm *M}^{\rm sym})^\#$ is given by
$u_\alpha\cdot u_\beta=\binom{d+e}du_{\alpha\beta}$,
where $\alpha$ has length~$d$ and $\beta$ has length~$e$.
Denote the right side of~\eqref{eq:tutte} by~$f$.  We have
\begin{align*}
f &= \sum_{i=0}^\infty\sum_{j=0}^\infty \frac{(i+j)!}{i!j!}((y-1)u_0+u_1)^i(u_0+(x-1)u_1)^j
\\&= \sum_i\sum_j \frac{(i+j)!}{i!j!} \sum_{\alpha\in\{0,1\}^{i+j}}
(x-1)^{r_{i+j}-r_i}(y-1)^{i-r_i}\ (i+j)!u_\alpha
\end{align*}
where $r_i = \sum_{k=1}^i\alpha_k$, so that 
$i-r_i$ is the number of indices $1\leq k\leq i$ such that $\alpha_k=0$, and 
$r_{i+j}-r_i$ is the number of indices $i+1\leq k\leq j$ such that $\alpha_k=1$.

Let $d=i+j$.  For a matroid $\rk$ on~$\underline d$ of rank~$r$, the elements $\rk$ and 
$1/d!\,\sum_{\sigma\in\Sigma_d} \rk\circ\sigma$ of~$Z_{\rm M}(d,r)$
have equal image under~$\pi_{\rm M}$.  Therefore
\begin{align*}
f(\rk) &= \frac1{d!}\sum_{\sigma\in\Sigma_d} f(\rk\circ\sigma)
\\&= \frac1{d!}\sum_{\sigma\in\Sigma_d}\sum_{i+j=d}\frac{d!}{i!j!}\sum_{\alpha\in\{0,1\}^d}
(x-1)^{r_d-r_i}(y-1)^{i-r_i}\ u_\alpha(\rk\circ\sigma)
\\&= \sum_{\sigma\in\Sigma_d}\sum_{i+j=d}\frac1{i!j!}
(x-1)^{r-\rk(\sigma(\underline i))}(y-1)^{i-\rk(\sigma(\underline i))}.
\end{align*}
The set $\sigma(\underline i)$ takes each value $A\subseteq\underline d$
in $|A|!(d-|A|)!$ ways, so
$$f(\rk) = \sum_{A\subseteq\underline d}(x-1)^{r-\rk(A)}(y-1)^{|A|-\rk(A)}
= \mathcal T(\rk).$$
\end{proof}

\section{Additive functions: the groups $T_{\rm M},T_{\rm PM},T_{\rm MM}$}\label{sec:T}
For $0\leq e\leq d$ we define $P_{\rm *M}(d,r,e)\subseteq P_{\rm *M}(d,r)$\label{not:Pdre} as 
the span of all $[\Pi]$ where $\Pi\subseteq \R^d$ is
a $*$matroid polytope of dimension $\leq d-e$.
We have $P_{\rm *M}(0,r,0)=P_{\rm *M}(0,r)$ and $P_{\rm *M}(d,r,1)=P_{\rm *M}(d,r)$ for $d\geq 1$.
These subgroups form a filtration
$$
\cdots\subseteq P_{\rm *M}(d,r,2)\subseteq P_{\rm *M}(d,r,1)\subseteq P_{\rm *M}(d,r,0)=P_{\rm *M}(d,r).
$$
Define $\overline{P}_{\rm *M}(d,r,e):=P_{\rm *M}(d,r,e)/P_{\rm *M}(d,r,e+1)$\label{not:Pbar}.
If $\Pi_1$ and $\Pi_2$ are polytopes of codimension $e_1$ and $e_2$
respectively, then $\Pi_1\times \Pi_2$ has codimension $e_1+e_2$.
It follows from (\ref{eq:nabla}) that  the multiplication $\nabla$ respects the filtration.
Since $Q(\rk_A)\times Q(\rk_{/A})$ is contained in $Q(\rk)$, it follows from (\ref{eq:Delta})
that the comultiplication $\Delta$ also respects the filtration:
$$
\Delta(P_{\rm *M}(d,r,e))\subseteq \sum_{i,j,k} P_{\rm *M}(i,j,k)\otimes P_{\rm *M}(d-i,r-j,e-k)
$$
Similarly, the antipode $S$ respects the grading.
The associated graded algebra
$$
\overline{P}_{\rm *M}=\bigoplus_{d,r,e} \overline{P}_{\rm *M}(d,r,e)
$$
has an induced Hopf algebra structure.

We define $T_{\star \rm M}(d,r)=\overline{P}_{\star \rm M}(d,r,1)$\label{not:TstarM}.

For every partition $\underline X:\underline{d}=\coprod_{i=1}^e X_i$ 
into nonempty subsets there exists a natural map
$$
\Phi_{\underline{X}}:\prod_i \R^{X_i}\to \R^d
$$
Define
$$
P_{\rm *M}(\underline{X})=\bigoplus_{r_1,r_2,\dots,r_e\in \Z}
P_{\rm *M}(|X_1|,r_1)\otimes \cdots \otimes P_{\rm *M}(|X_e|,r_e)
$$
and
$$
\overline{P}_{\rm *M}(\underline{X})=\bigoplus_{r_1,r_2,\dots,r_e\in \Z}
T_{\rm *M}(|X_1|,r_1)\otimes \cdots \otimes T_{\rm *M}(|X_e|,r_e).
$$
The map $\Phi_{\underline{X}}$ induces a group homomorphism
$$
\phi_{\underline{X}}:P_{\rm *M}(\underline{X},e)
\to P_{\rm *M}(d,r,e)
$$
defined by
$$
\phi_{\underline{X}}([\Pi_1]\otimes [\Pi_2]\otimes \cdots \otimes [\Pi_e])=
[\Phi_{\underline{X}}(\Pi_1\times \Pi_2\times \cdots \times \Pi_e)].
$$
The map $\phi_{\underline{X}}$ induces a group homomorphism
$$
\overline{\phi}_{\underline{X}}:\overline{P}_{\rm *M}(\underline{X},e)
\to \overline{P}_{\rm *M}(d,r,e).
$$

 A vector $y=(y_1,\dots,y_d)\in \R^d$ is called
$\underline{X}$-integral if $\sum_{i\in X_j}y_i\in \Z$ for $j=1,2,\dots,e$.
An $\underline{X}$-integral vector $y$ is called $\underline{X}$-regular, if for every
$j$ and every $Y\subseteq X_j$ we have: if $\sum_{i\in Y}y_i\in \Z$, then $Y=\emptyset$ or
$Y=X_j$. In other words, an $\underline{X}$-integral vector $y$ is called $\underline{X}$-regular
if it is not integral for any refinement of $\underline{X}$.
We call $y$ $\underline{X}$-balanced if $\sum_{i\in S}y_i=0$ holds 
if and only if $S$ is a union of some of the $X_j$'s.

Choose an $\underline{X}$-balanced vector $y_{\underline{X}}$ for every $\underline{X}$.
For $f\in P_{\rm *M}(d,r)$ we define 
$$
\gamma_{\underline{X}}(f)(x):=\lim_{\varepsilon\downarrow 0}f(x+\varepsilon y_{\underline X}).
$$
If $\Pi$ is a $*$matroid base polytope, then $\gamma_{\underline{X}}([\Pi])(x)$ is constant on faces of $\Pi$.
This shows that $\gamma_{\underline{X}}([\Pi])\in P_{\rm *M}(d,r)$. So $\gamma_{\underline{X}}$ is a 
endomorphism of $P_{\rm *M}(d,r)$.  Now $\gamma_{\underline{X}}$ also induces an endomorphism
$\overline{\gamma}_{\underline{X}}$ of $\overline{P}_{\rm *M}(d,r)$.

\begin{lemma}
We have that $\gamma_{\underline{X}}\circ\gamma_{\underline{X}}=\gamma_{\underline{X}}$.
\end{lemma}
\begin{proof}
Suppose that $x\in \R^d$. Consider the set $S$ of all $x+\varepsilon y_{\underline{X}}$ with $\varepsilon\in \R$.
There exists a partition $\underline{Y}$ of $\underline{d}$ and a dense open subset $U$ of $S$ such
that all points in $U$ are $\underline{Y}$-regular. Then there exists a $\delta>0$ such that
$T=\{x+\varepsilon y\mid 0<\varepsilon<\delta\}$ has only $\underline{Y}$-regular points.
For every $*$matroid base polytope $\Pi$, we have that $T\cap \Pi=\emptyset$ or  $T\subseteq \Pi$. It follows that for every $f\in P_{\rm M*}(d,r)$ there exists a constant $c$ such that
$f$ is equal to $c$ on $T$. Therefore $\gamma_{\underline{X}}(f)(x)=c$ and $\gamma_{\underline{X}}(f)$ is constant and equal to $c$ on $T$.
We conclude that  $\gamma_{\underline{X}}(\gamma_{\underline{X}}(f))(x)=c=\gamma_{\underline{X}}(f)(x)$.
\end{proof}

\begin{lemma}\label{lem:gamma_phi}
Suppose that $\underline{X},\underline{Y}$ are partitions of $\underline{d}$
into $e$ nonempty subsets, and $\underline{X}\neq \underline{Y}$.
Then we have
$$
\gamma_{\underline{X}}\circ \phi_{\underline{Y}}=0.
$$
\end{lemma}
\begin{proof}
For some $k$, $Y_k$ is not the
union of $X_j$'s. The image
$$
\Phi_{\underline{Y}}(\Pi_1\times \cdots \times\Pi_e)
$$
consists of $\underline{Y}$-integral points. For any $x\in \R^{\underline{d}}$,
 $x+\varepsilon y_{\underline{X}}$ is not $\underline{Y}$-integral for
small $\varepsilon>0$. In follows that
$$
\gamma_{\underline{X}}(\phi_{\underline{Y}}([\Pi_1\times\cdots\times \Pi_e]))(x)=
\gamma_{\underline{X}}([\Phi_{\underline{Y}}(\Pi_1\times \cdots \times \Pi_e)])(x)=0 
$$
for all $x$.
\end{proof}

\begin{theorem}\label{theo:iso}
We have the following isomorphism
\begin{equation}\label{eq:iso}
\overline{\phi} : \bigoplus_{\scriptstyle \underline{X}=(X_1,X_2,\dots,X_e)
\atop \scriptstyle \underline{d}=X_1\sqcup X_2\sqcup \cdots \sqcup X_e;X_1,\dots,X_e\neq\emptyset}
\overline{P}_{\rm *M}(\underline{X})\to\bigoplus_{r\in \Z}\overline{P}_{\rm *M}(d,r,e)\end{equation}
where $\overline{\phi}=\sum_{\underline{X}}\overline{\phi}_{\underline{X}}$.
\end{theorem}
\begin{proof}
We know that a $*$matroid base polytope of codimension $e$ is a product of
$e$ $*$matroid base polytopes of codimension $1$. This shows that $\overline{\phi}$ is surjective.
It remains to show that $\overline{\phi}$ is injective.

Suppose that $\overline{\phi}(u)=0$
where $u=\sum_{\underline{X}}u_{\underline{X}}$, and  $u_{\underline{X}}\in \overline{P}_{\rm *M}(\underline{X})$ for all $\underline{X}$.
We have $\gamma_{\underline{X}}\circ \phi_{\underline Y}=0$ if $\underline{X}\neq \underline{Y}$ by Lemma~\ref{lem:gamma_phi}.
It follows that $\overline{\gamma}_{\underline{X}}(\overline{\phi}_{\underline{X}}(u_{\underline{X}}))=\overline{\gamma}_{\underline{X}}(\overline{\phi}(u))=0$.
We can lift $u_{\underline{X}}$ to an element $\widetilde{u}_{\underline{X}}\in P_{\star\rm M}(\underline{X})$.
Then we have that
$$
\gamma_{\underline{X}}(\phi_{\underline{X}}(\widetilde{u}_{\underline{X}}))=\sum_i a_i [\Lambda_i]
$$
where the $\Lambda_i$ are $*$matroid polytopes of codimension $>e$. 
We have that $[\Lambda_i]\in\im\phi_{\underline Y'}$ for some partition $\underline Y'$
with more than $e$ parts.  Therefore $[\Lambda_i]\in\im\phi_{\underline Y}$ as well
for any coarsening $\underline Y$ of~$\underline Y'$ with $e$ parts, and 
we may choose $\underline Y$ so that $\underline Y\neq\underline X$,
so by Lemma~\ref{lem:gamma_phi}, $\gamma_{\underline{X}}([\Lambda_i])=0$ for all $i$.
Therefore, we have
$$
\gamma_{\underline{X}}(\phi_{\underline{X}}(\widetilde{u}_{\underline{X}}))=
\gamma_{\underline{X}}(\gamma_{\underline{X}}(\phi_{\underline{X}}(\widetilde{u}_{\underline{X}})))=
\sum_i a_i \gamma_{\underline{X}}([\Lambda_i])=0.
$$
Note that $\gamma_{\underline{X}}$ induces a map
$\gamma_{\underline{X}}':P_{\rm *M}(\underline{X})\to P_{\rm *M}(\underline{X})$
such that $\phi_{\underline{X}}\circ \gamma_{\underline{X}}'=\gamma_{\underline{X}}\circ \phi_{\underline{X}}$.
We have that
$$
\phi_{\underline{X}}(\widetilde{u}_{\underline{X}})=
(\id-\gamma_{\underline{X}})(\phi_{\underline{X}}(\widetilde{u}_{\underline{X}}))=
\phi_{\underline{X}}((\id-\gamma_{\underline{X}}')(\widetilde{u}_{\underline{X}})).
$$
Since $\phi_{\underline{X}}$ is injective, we have
$$
\widetilde{u}_{\underline{X}}=(\id-\gamma_{\underline{X}}')(\widetilde{u}_{\underline{X}}).
$$
So $\widetilde{u}_{\underline{X}}$ lies in the image of $\id-\gamma_{\underline{X}}$.

For $*$matroid polytopes $\Pi_1,\dots,\Pi_e$ of codimension 1 in $\R^{|X_1|},\dots,\R^{|X_e|}$ respectively,
we have
$$
\gamma_{\underline{X}}'([\Pi_1\times \cdots \times \Pi_e])(x)=1.
$$
for any relative interior point $x$ of $\Pi_1\times \cdots \times \Pi_e$.
It follows that
$$
(\id-\gamma_{\underline{X}}')([\Pi_1\times \cdots \times \Pi_e])=\sum_F a_F[F]
$$
where $F$ runs over the proper faces of $\Pi_1\times \cdots \times \Pi_r$ and $a_F\in \Z$ for all $F$.
Therefore,  the composition
$$
\xymatrix{
P_{\rm *M}(\underline{X})\ar[r]^{\id-\gamma_{\underline{X}}'} & P_{\rm *M}(\underline{X})\ar[r] & \overline{P}_{\rm *M}(\underline{X},e)}
$$
is equal to 0. Since $u_{\underline{X}}$ is the image of $\widetilde{u}_{\underline{X}}=(\id-\gamma_{\underline{X}}')(\widetilde{u}_{\underline{X}})$, we have that $u_{\underline{X}}=0$.
\end{proof}

\noindent Let $p_{\rm (P)M}(d,r,e)$ be the rank of $\overline{P}_{\rm (P)M}(d,r,e)$\label{not:pdre},
and $t_{\rm (P)M}(d,r):=p_{\rm (P)M}(d,r,1)$\label{not:tstarM} be the rank of $T_{\rm (P)M}(d,r)$.

\begin{proof15d}
From Theorem~\ref{theo:iso} follows that
\begin{multline*}
\exp\left(\sum_{d,r\geq 0} \frac{t_{\rm PM}(d,r)x^dy^r u}{d!}\right)=
\sum_{e\geq 0}\frac{1}{e!}\big(\sum_{d,r\geq 0} \frac{t_{\rm PM}(d,r)x^dy^r u}{d!}\big)^e=\\
=
\sum_{e,d,r\geq 0} \frac{p_{\rm PM}(d,r,e)}{d!}x^dy^r u^e
\end{multline*}
If we substitute $u=1$, we get
$$
\exp\left(\sum_{d,r\geq 0}\frac{t_{\rm PM}(d,r)x^dy^r}{d!}\right)=\sum_{e,d,r\geq 0} \frac{p_{\rm PM}(d,r,e)}{d!}x^dy^r=
\frac{e^x(1-y)}{1-ye^x}.
$$
It follows that
\begin{multline*}
\sum_{d,r\geq 0}\frac{t_{\rm PM}(d,r)x^dy^r}{d!}=\log\big(\frac{e^x(1-y)}{1-ye^x}\big)=\\
=x+\log(1-y)-\log(1-ye^x)=
x+\sum_{r\geq 1}\frac{(e^{rx}-1)y^r}{r}.
\end{multline*}
Comparing the coefficients of $x^dy^r$ gives
$$
t_{\rm PM}(d,r)=\left\{
\begin{array}{rl}
r^{d-1} &\mbox{ if $d\geq 1$};\\
0 & \mbox{otherwise.}
\end{array}
\right.
$$
(Recall that $0^0=1$.)  
\end{proof15d}
We also have
$$
\sum_{d,r\geq 0} \frac{p_{\rm PM}(d,r,e)t^ds^r u^e}{d!}=
\exp\left(\log\big(\frac{e^t(1-s)}{1-se^t}\big)u\right)=\big(\frac{e^t(1-s)}{1-se^t}\big)^u.
$$
\begin{proof15c}
The proof is similar to the proof of part (d). We have
\begin{equation}\label{eq:PMgenfunc}
\sum_{d,r\geq 0}\frac{t_{\rm M}(d,r)x^{d-r}y^r}{d!}=
\log\left(\sum_{d,r,e}\frac{p_{\rm M}(d,r,e)x^{d-r}y^r}{d!}\right)=
\log\left(\frac{x-y}{xe^{-x}-ye^{-y}}\right),
\end{equation}
and
$$
\sum_{d,r,z\geq 0}\frac{p_{\rm M}(d,r,e) x^{d-r}y^rz^e}{d!}=\big(\frac{x-y}{xe^{-x}-ye^{-y}}\big)^z.
$$
\end{proof15c}
A table for the values $t_{\rm (P)M}(d,r)$ can be found in~\ref{apB}.

If $d\geq 1$, let $\mathfrak t_{\rm PM}(d,r)$\label{not:mf t}
be the set of all pairs $(\underline X,\underline r)\in\mathfrak p_{\rm PM}(d,r)$ 
such that $r_1>0$, 
and $d\not\in X_{k-1}$, where $k$ is the length of $\underline X$.
Similarly, if $d\geq 2$, let $\mathfrak t_{\rm M}(d,r)$
be the set of all pairs $(\underline X,\underline r)\in\mathfrak t_{\rm M}(d,r)$ 
such that $r_1>0$,  $|X_{k-1}|-r_{k-1}<d-r$,
and $d\not\in X_{k-1}$.

\begin{lemma}\label{lem p1 card}
We have
$|\mathfrak t_{\rm (P)M}(d,r)| = t_{\rm (P)M}(d,r)$
whenever the former is defined.  
\end{lemma}

\begin{proof}
\noindent {\em For polymatroids.}
We revisit the bijection 
$f:\mathfrak p_{\rm PM}(d,r)\to\mathfrak a(d,r)$ defined in the proof of
Theorem~\ref{theo1.4}(d).  It is easy to see that
$\underline a\in f(\mathfrak t_{\rm PM}(d,r))$ 
if and only if $a_d=r_k=r$ and no $a_i$ equals~0.  Accordingly
such an $\underline a$ has the form $(a_1,\ldots,a_{d-1},r)$
with $a_i$ freely chosen from $\{1,\ldots,r\}$ for each $i=1,\ldots,d-1$, 
so $|f(\mathfrak t_{\rm PM}(d,r))| = r^{d-1}$.  

\noindent {\em For matroids.}
We proceed by means of generating functions.  
We begin by invoking the
exponential formula: the  coefficient of $x^{d-r}y^r$ of the generating function
$$\exp\left(\sum_{d=0}^\infty\sum_{r=0}^d
  \frac{|\mathfrak t_{\rm M}(d,r)|}{d!}x^{d-r}y^r\right)$$
enumerates the ways to choose a partition 
$\underline d=Z_1\cup\cdots\cup Z_l$ and a composition $r=s_1+\cdots+s_l$
and an element $(\underline X_i,\underline r_i)$ 
of $\mathfrak t_{\rm M}(|Z_i|,s_i)$ for each 
$i=1,\ldots,l$.  Let us denote by $\mathfrak q(d,r)$
the set of tuples 
$(\underline d,r,(\underline X^{(1)},\underline r^{(1)}),\ldots,(\underline X^{(l)},\underline r^{(l)}))$.

We describe a bijection between $\mathfrak q(d,r)$ and
$\mathfrak p^{\rm sym}_{\rm M}(d,r)$.  
Roughly, given $(\underline X,\underline r)\in\mathfrak p^{\rm sym}_{\rm M}(d,r)$,
we break it into pieces, breaking after $X_i$ whenever 
$X_i\setminus X_{i-1}$ contains the largest remaining element 
of~$\underline d\setminus X_{i-1}$.
More formally, given $(\underline X,\underline r)\in\mathfrak p^{\rm sym}_{\rm M}(d,r)$, 
for each $j\geq 1$ 
let $Z_j=X_{i_j}\setminus X_{i_j-1}$ (taking $i_0=0$)
where $i_j$ is minimal such that $X_{i_j}$ contains 
the maximum element of~$\underline d\setminus X_{i_{j-1}}$, 
and let $s_j=r_{i_j}-r_{i_{j-1}}$.
This definition eventually fails, in that we cannot find a
maximum element when $X_{i_{j-1}}=X_k=\underline d$,
so we stop there and let $l$ be such that $i_l=k$.  
For $j=1,\ldots,l$, 
let $f_j:Z_j\to\underline{|Z_j|}$ be the unique order-preserving map, 
and define the chain and list of integers
$(\underline X^{(j)},\underline r^{(j)})$ by
\begin{align*}
X^{(j)}_i &= f_j(X_{i_{j-1}+i}\setminus X_{i_{j-1}}),&(i=1,\ldots,i_j-i_{j-1}) \\
r^{(j)}_i &= r_{i_{j-1}+i}-r_{i_{j-1}}.&(i=1,\ldots,i_j-i_{i-1})
\end{align*}
We have that 
$(\underline X^{(j)},\underline r^{(j)})\in\mathfrak t_{\rm M}(|Z_j|,s_j)$:
the crucial property that $d\not\in X_{k-1}$ obtains by choice 
of~$i_j$ and monotonicity of~$f_j$.  
This finishes defining the bijection.  Its inverse is easily constructed.  

From this bijection and (\ref{eq:PMgenfunc}) it follows that
\begin{multline*}
\exp\left(\sum_{d=1}^\infty\sum_{r=0}^d
  \frac{|\mathfrak t_{\rm M}(d,r)|}{d!}x^{d-r}y^r\right) 
= 1+\sum_{d\geq1}\sum_{r}\frac{|\mathfrak q(d,r)|x^{d-r}y^{r}}{d!}=\\
=\sum_{d,r}\frac{p_{\rm M}(d,r)x^{d-r}y^r}{d!}=
\exp\left(\sum_{d=1}^\infty\sum_{r=0}^d
\frac{t_{\rm M}(d,r)}{d!}x^{d-r}y^r\right).
\end{multline*}
\end{proof}

\begin{lemma}\label{lem p1 indep}
The classes of~$[R_{\rm (P)M}(\underline X,\underline r)]$
for $(\underline X,\underline r)\in\mathfrak t_{\rm (P)M}(d,r)$  
are linearly independent in $T_{\rm (P)M}(d,r)$.  
\end{lemma}

\begin{proof}
Let $y=(-1,\ldots,-1,d-1)$.
Let $\Pi_{\rm (P)M}$ be the set of points $x\in\Delta_{\rm (P)M}(d,r)$
such that $x+\varepsilon y\in\Delta_{\rm (P)M}(d,r)$
for sufficiently small $\varepsilon>0$.
Choose some $(\underline X,\underline r)\in\mathfrak t_{\rm (P)M}(d,r)$.
If $x\in R_{\rm (P)M}(\underline X,\underline r)\cap \Pi_{\rm (P)M}$, and
$\varepsilon>0$ is sufficiently small, we have
$x+\varepsilon y\in R_{\rm (P)M}(\underline X,\underline r)$,
since the defining inequalities 
of~$R_{\rm (P)M}(\underline X,\underline r)$ involve only the variables
$x_1,\ldots,x_{d-1}$.  
It follows that for $x\in \Pi_{\rm (P)M}$ we have
$$[R_{\rm (P)M}(\underline X,\underline r)](x)
= \gamma_y([R_{\rm (P)M}(\underline X,\underline r)])(x).$$
We will write $\{\underline d\}$ for the partition $\underline d=\{1\}\cup \{2\} \cup \cdots\cup \{d\}$.
Observe that $y$ is $\{\underline d\}$-balanced, so that
for any point $x$, $x+\varepsilon y$ is $\{\underline d\}$-regular for
sufficiently small $\varepsilon>0$.  

Suppose the sum
$$S=\sum_{(\underline X,\underline r)\in\mathfrak t_{\rm (P)M}(d,r)}
a(\underline X,\underline r)[R_{\rm (P)M}(\underline X,\underline r)]$$
vanishes in $T_{\rm (P)M}(d,r)$, i.e.\ 
is contained in $P_{\rm (P)M}(d,r,2)$.  
Then the support of~$S$ contains no $\{\underline d\}$-regular points.
So for any $x\in \Pi_{\rm (P)M}$ we have 
$S(x) = \gamma_y(S)(x) = 0$.

We specialize now to the matroid case.  
If $r=d$, then $T_{\rm M}(d,r)=0$ and the result is trivial.  
Otherwise let $H$ be the hyperplane~$\{x_d=0\}$; we will examine the
situation on restriction to~$H$.  
Identifying $H$ with~$\R^{d-1}$ in the obvious fashion, we have 
$\Delta_{\rm M}(d,r)\cap H = \Delta_{\rm M}(d-1,r)$, 
$\Pi_{\rm M}\cap H = \{x\in\Delta_{\rm M}(d-1,r) : \mbox{$x_i\neq 0$ for all $i$}\}$.
For any $(\underline X,\underline r)\in\mathfrak t_{\rm M}(d,r)$, 
$R_{\rm M}(\underline X,\underline r)\cap H = R_{\rm M}(\underline X',\underline r')$
where, supposing $\underline X$ has length~$k$,
$$\underline X'' : \emptyset\subset X_1\subset\cdots\subset 
X_{k-1}\subset X_k\setminus\{d\} = \underline{d-1}$$
and $(\underline X',\underline r')$ is obtained from~$(\underline X'',\underline r)$
by dropping redundant entries as in the proof of Theorem~\ref{free generation}.

Suppose $T\in P_{\rm M}(d-1,r)$ is supported on~$\{x_i=0\}$.
By Theorem~\ref{free generation} 
we have a unique expression 
$$T = \sum_{(\underline X,\underline r)\in\mathfrak p_{\rm M}(d-1,r)}
b(\underline X,\underline r)[R_{\rm M}(\underline X,\underline r)].$$
But we also have
$$T = T|_{\{x_i=0\}} = \sum_{(\underline X,\underline r)\in\mathfrak p_{\rm M}(d-1,r)}
b(\underline X,\underline r)[R_{\rm M}(\underline X,\underline r)\cap\{x_i=0\}]$$
in which each $[R_{\rm M}(\underline X,\underline r)\cap\{x_i=0\}]$ is
either zero or another $[R_{\rm M}(\underline X',\underline r')]$,
so that by uniqueness $b(\underline X,\underline r)=0$ when
$R_{\rm M}(\underline X,\underline r)\not\subseteq\{x_i=0\}$.  

The restriction $S|_H$ is supported on
$$
\Delta_{\rm M}(d-1,r)\cap\left(\bigcup_{i=1}^{d-1}\{x_i=0\}\right),$$
so it is a linear combination of those $[R_{\rm M}(\underline X,\underline r)]$
supported on some~$\{x_i=0\}$, i.e.\ those for which 
$r_1=0$.  On the other hand, 
$$S|_H=\sum_{(\underline X,\underline r)\in\mathfrak t_{\rm M}(d,r)}
a(\underline X,\underline r)[R_{\rm M}(\underline X,\underline r)\cap H]$$
in which each $R_{\rm M}(\underline X,\underline r)\cap H$ is
another matroid polytope $R_{\rm M}(\underline X',\underline r)$
with $r_1>0$ (and $X'$ only differing from $X$ by dropping the $d$
in the $k$\/th place).  Note that $(\underline{X},\underline{r})\in {\mathfrak t}_{\rm M}(d,r)$
is completely determined by $R_{\rm M}(\underline{X},\underline{r})\cap H$.
Therefore, by Theorem~\ref{free generation},
$a(\underline X,\underline r)=0$ for all $(\underline{X},\underline{r})\in {\mathfrak t}_{\rm M}(d,r)$. 

The polymatroid case is similar, but in place of the hyperplane~$H$
we use all the hyperplanes $H_i = \{x_d=i\}$ for $i=0,\ldots,r-1$.

Note that $\Delta_{\rm PM}(d,r)\cap H_i=\Delta_{\rm PM}(d,r-i)$.  
For $(\underline X,\underline r)\in\mathfrak t_{\rm PM}(d,r)$, 
supposing $\underline X$ has length~$k$, 
$$R_{\rm PM}(\underline X,\underline r)\cap H_i = 
\left\{\begin{array}{ll}
R_{\rm M}(\underline X',\underline r') & r_{k-1}\geq r-i\\
\emptyset & \mbox{otherwise}
\end{array}\right.$$
where again
$$\underline X'' : \emptyset\subset X_1\subset\cdots\subset 
X_{k-1}\subset X_k\setminus\{d\} = \underline{d-1}$$
and
$$
r''=(r_1,r_2,\dots,r_{k-1},r_k-i).
$$
and $(\underline X',\underline r')$ is obtained from~$(\underline X'',\underline r'')$
by dropping redundant entries as in the proof of Theorem~\ref{free generation}.
Although $(\underline{X},\underline{r})\in {\mathfrak t}_{\rm PM}(d,r)$ is
not completely determined by $S\mid_{H_0}$, the arguments in the matroid case still
show that $S\mid_{H_0}=0$, and $a(\underline{X},\underline{r})=0$ for
all $(\underline{X},\underline{r})$ for which $X_{k-1}\neq \underline{d-1}$.
Restricting to $H_{r-1}$ shows that $a(\underline{X},\underline{r})=0$
for all $(\underline{X},\underline{r})$ for which $X_{k-1}=\underline{d-1}$
and $r_k=1$. Proceeding by induction on $i$, we restrict $S$ to $H_{r-i}$ and  see that $a(\underline{X},\underline{r})=0$
for all $(\underline{X},\underline{r})$ for which $X_{k-1}=\underline{d-1}$ and $r_{k-1}=i$.
\end{proof}

The following is an immediate consequence of 
Lemmas \ref{lem p1 card} and~\ref{lem p1 indep}.

\begin{theorem}\label{theo:TPMgens}
The group $T_{\rm (P)M}(d,r)$ is freely generated by all~$[R_{\rm (P)M}(\underline X,\underline r)]$ with
$(\underline X,\underline r)\in\mathfrak t_{\rm (P)M}(d,r)$.  
\end{theorem}
\begin{example}
Consider again Example~\ref{expolymatroid}.
The set $t_{\rm PM}(3,2)$ consists of the following elements:
$$
\begin{array}{|rl|rl|rl|} \hline
\underline{X}:& \{1,2,3\} & \underline{X}: & \{1,2\}\subset \{1,2,3\} \\
\underline{r}=& (2) & \underline{r}= & (1,2)  \\ \hline
\underline{X}:& \{1\} \subset \{1,2,3\} & \underline{X}: & \{2\} \subset \{1,2,3\} \\
\underline{r}=& (1,2) & \underline{r}= & (1,2)  \\ \hline
\end{array}
$$
The polytopes $R_{\rm PM}(\underline{X},\underline{r})$, $(\underline{X},\underline{r})\in {\mathfrak p}_{\rm PM}(3,2)$ are

\leaveout{
\centerline{ 
\includegraphics[width=2in,height=2in]{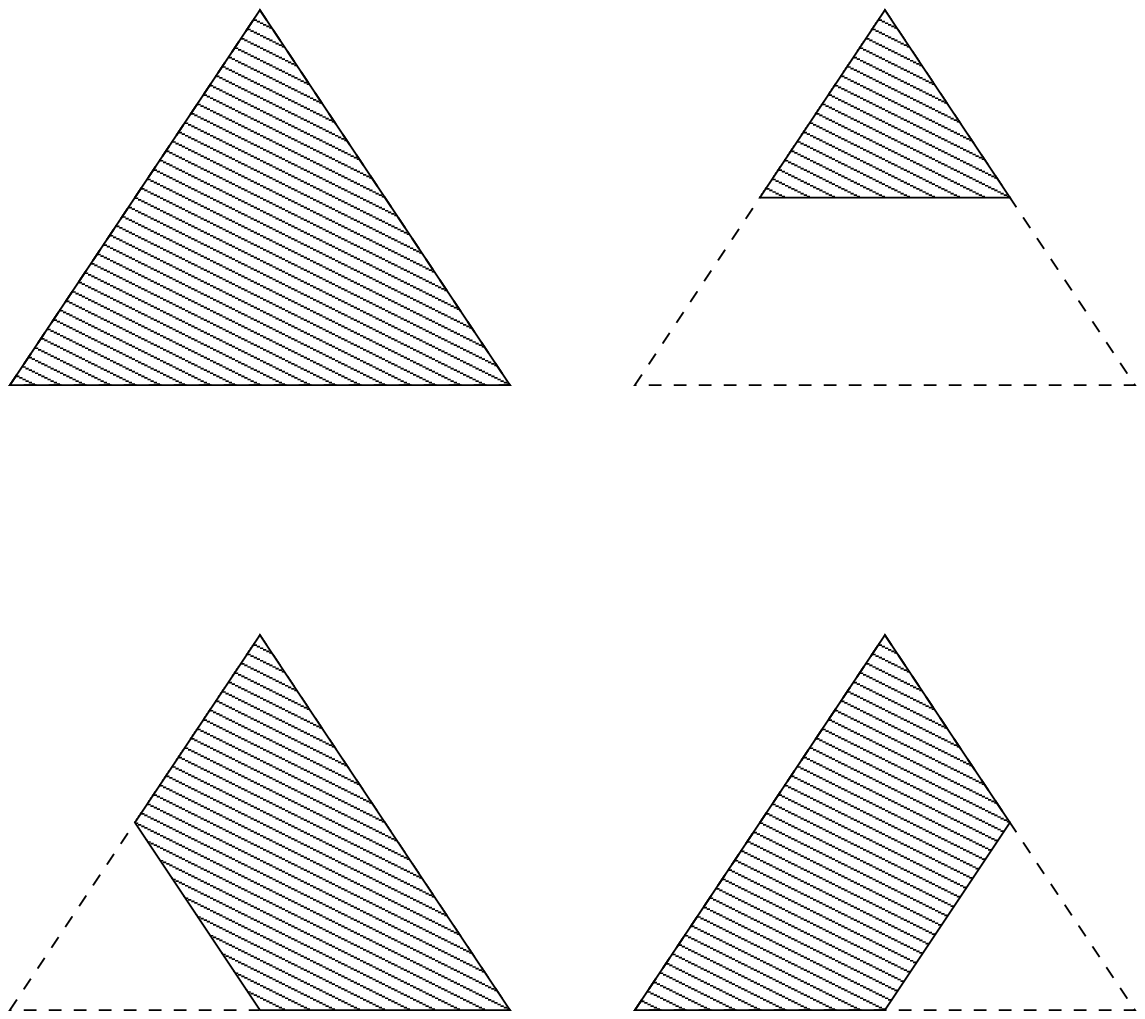}
}}

\end{example}
\begin{example}
Consider again Example~\ref{exmatroid}. The set $t_{\rm M}(4,2)$ consists of the following elements:
$$
\begin{array}{|rl|rl|rl|} \hline
\underline{X}:& \{1,2,3,4\} & \underline{X}: & \{1,2\}\subset \{1,2,3,4\} \\
\underline{r}=& (2) & \underline{r}= & (1,2)  \\ \hline
\underline{X}:& \{1,3\} \subset \{1,2,3,4\} & \underline{X}: & \{2,3\} \subset \{1,2,3,4\} \\
\underline{r}=& (1,2) & \underline{r}= & (1,2)  \\ \hline
\end{array}
$$
The polytopes $R_{\rm M}(\underline{X},\underline{r})$, $(\underline{X},\underline{r})\in {\mathfrak p}_{\rm M}(4,2)$ are

\leaveout{
\centerline{ 
\includegraphics[width=2in]{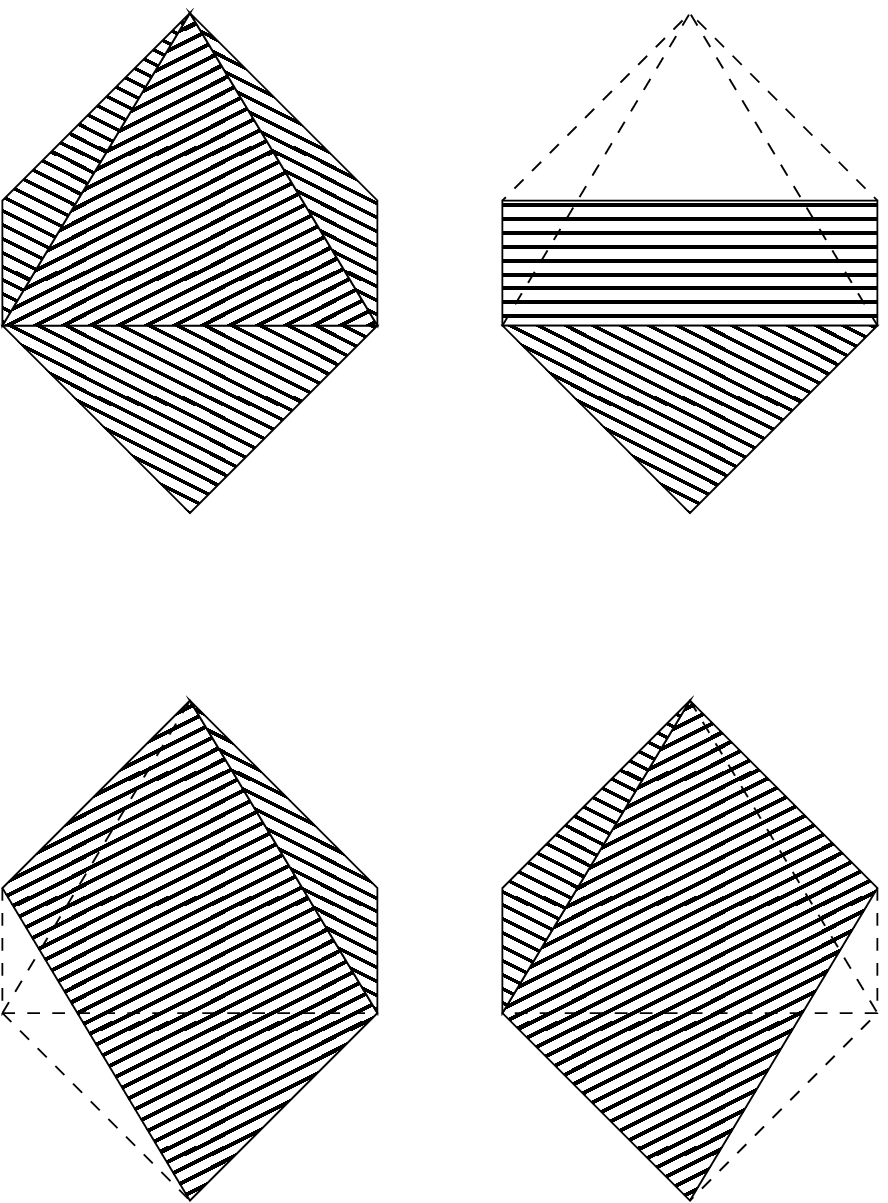}
}}

\end{example}
\section{Additive invariants: the groups $T_{\rm M}^{\rm sym}$, $T_{\rm PM}^{\rm sym}$, $T_{\rm MM}^{\rm sym}$}\label{sec:Tsym}

The algebra $P^{\rm sym}_{\star M}$ also has a natural filtration:
$$
\cdots \subseteq P^{{\rm sym}}_{\star \rm M}(d,r,2)\subseteq P^{{\rm sym}}_{\star \rm M}(d,r,1)\subseteq P^{{\rm sym}}_{\star M}(d,r,0)=P^{{\rm sym}}_{\star \rm M}(d,r).
$$
Here $P^{{\rm sym}}_{\rm\star M}(d,r,e)$\label{not:Pdresym} is spanned by the indicator functions of all $\ast$matroid base polytopes 
of rank $r$ and dimension $d-e$. Define $\overline{P}^{\rm sym}_{\rm \star M}(d,r,e)=P^{\rm sym}_{\rm *M}(d,r,e)/\linebreak P^{\rm sym}_{\rm *M}(d,r,e+1)$\label{not:Pbarsym}. 
Let $\overline{P}^{\rm sym}_{\star \rm M}=\bigoplus_{d,r,e} \overline{P}^{\rm sym}_{\star \rm M}(d,r,e)$
be the associated graded algebra.

Define $T_{\star \rm M}^{\rm sym}=\bigoplus_{d,r}\overline P^{{\rm sym}}_{\star \rm M}(d,r,1)$\label{not:TstarMsym}. 
The following Corollary follows from Theorem~\ref{theo:iso}.
\begin{theorem}\label{theo:symalg}
The algebra $\overline{P}^{\rm sym}_{\star\rm M}$ is the free symmetric algebra 
$S(T_{\star \rm M}^{\rm sym})$ on $T_{\star \rm M}^{\rm sym}$,
and there exists an isomorphism 
\begin{equation}\label{eq:Se}
S^e(T_{\star \rm M}^{\rm sym})\cong\bigoplus_{d,r}\overline{P}^{{\rm sym}}_{\star \rm M}(d,r,e).
\end{equation}
\end{theorem}
\begin{proof}
If we sum the isomorphism (\ref{eq:iso}) in Theorem~\ref{theo:iso} over all $d$, we get an isomorphism
$$
\bigoplus_{d,\underline{X}} \overline{P}(\underline{X})\to \bigoplus_{d,r}\overline{P}_{\star M}(d,r,e)
$$
where the sum on the left-hand side is over all $d$ and all partitions $\underline{X}$
of $\underline{d}$ into $e$ nonempty subsets. If we divide out the symmetries on both sides,
we get the isomorphism (\ref{eq:Se}).
\end{proof}
\begin{corollary}\label{cor:Ppoly}
The algebra $P_{\star \rm M}^{\rm sym}$ is a polynomial ring over $\Z$.
\end{corollary}
\begin{proof}
Consider the surjective map
$$
\bigoplus_{d,r}P^{\rm sym}_{\star \rm M}(d,r,1)\to \bigoplus_{d,r}\overline{P}^{\rm sym}_{\star \rm M}(d,r,1)=T^{\rm sym}_{\star\rm M}.
$$
Suppose that $G$ is a set of $\Z$-module generators of $T^{\rm sym}_{\rm \star M}$.
Each element of $G$ can be lifted 
 to $\bigoplus_{d,r}P^{{\rm sym}}_{\star M}(d,r,e)$. Let $\widetilde{G}$ be the set of all lifts.
  Since $G$ generates $\overline{P}^{\rm sym}_{\star M}$ by Theorem~\ref{theo:symalg}, 
  $\widetilde{G}$ generates $P^{\rm sym}_{\star M}$ over $\Z$.
  Since $G$ is an algebraically independent set, so is $\widetilde{G}$. 
  So $P_{\star \rm M}^{\rm sym}$ is a polynomial ring over $\Z$, generated by $\widetilde{G}$.
  \end{proof}

\begin{proof15ab}
We prove the stated formulas after taking the reciprocal of both sides.
Let $p_{\star\rm M}^{\rm sym}(d,r,e)$ be the rank of $\overline{P}^{\rm sym}_{\star\rm M}(d,r,e)$\label{not:pdresym}. 
Define $t_{\star \rm M}^{\rm sym}(d,r):=p_{\star \rm M}(d,r,1)$\label{not:tstarMsym} as the rank of $T_{\star M}^{\rm sym}(d,r)$.
From the matroid case of Theorem~\ref{theo:symalg} follows that
$$
\prod(1-x^ry^{d-r})^{-t_{\rm M}^{\rm sym}(d,r)}=\frac1{1-x-y}
$$
and
$$
\prod(1-ux^ry^{d-r})^{-t_{\rm M}^{\rm sym}(d,r)}=\sum_{d,r} p_{\rm M}^{\rm sym}(d,r,e)u^ex^{r}y^{d-r}
$$
From the polymatroid case follows that
$$
\prod(1-x^dy^{r})^{-t_{\rm PM}^{\rm sym}(d,r)}=\frac{1-y}{1-x-y},
$$
and
$$
\prod(1-zx^dy^{r})^{-t_{\rm PM}^{\rm sym}(d,r)}=\sum_{d,r} p_{\rm PM}^{\rm sym}(d,r,e)z^ex^{d}y^{r}.
$$
\end{proof15ab}

\section{Invariants as elements in free algebras}\label{sec:free algebras}
Let
$$
(P_{\rm *M}^{\rm sym})^\#:=\bigoplus_{d,r}P_{\rm *M}^{\rm sym}(d,r)^\vee
$$
be the {\em graded} dual of $P_{\rm *M}^{\rm sym}$\label{not:Vgr}.
\begin{prooffreeab}
A basis of $(P_{\rm PM}^{\rm sym})^\#\otimes_\Z \Q$ is given by all $u_{\alpha}$
where $\alpha$ runs over all sequences of nonnegative integers,
and a basis of $(P_{\rm M}^{\rm sym})^\#\otimes_\Z\Q$ is given by all $u_{\alpha}$
where $\alpha$ is a sequence of $0$'s and $1$'s (see Corollaries~\ref{cor:PPMsymdual}
and \ref{cor:PMsymdual}).  
The multiplication in $(P_{\rm *M}^{\rm sym})^\#$ is given by
$$
u_\alpha\cdot u_\beta=\binom{d+e}d u_{\alpha\beta},
$$
where $\alpha$ has length~$d$ and $\beta$ has length~$e$.
It follows that $(P_{\rm PM}^{\rm sym})^\#\otimes_\Z\Q$ is the free associative algebra $\Q\langle u_0,u_1,u_2,\dots\rangle$ generated by $u_0,u_1,u_2,\dots$
and $(P_{\rm M}^{\rm sym})^\#\otimes_\Z\Q$ is the free associative algebra $\Q\langle u_0,u_1\rangle$
(the binomial coefficients make no difference).
The ordinary dual, $(P_{\rm *M}^{\rm sym})^\vee$ is a completion of the graded dual $(P_{\rm *M}^{\rm sym})^\#$.
We get that $(P_{\rm PM}^{\rm sym})^\vee\otimes_\Z\Q$ is equal to $\Q\langle\langle u_0,u_1,u_2,\dots\rangle\rangle$
and $(P_{\rm M}^{\rm sym})^\vee\otimes_\Z\Q$ is equal to $\Q\langle\langle u_0,u_1\rangle\rangle$.
\end{prooffreeab}
Let ${\mathfrak m}_{\rm \star M}=\bigoplus_{d,r}P^{\rm sym}_{\rm *M}(d,r,1)$\label{not:mstarM}.
Then we have ${\mathfrak m}_{\rm \star M}^2=\bigoplus_{d,r}P^{\rm sym}_{\rm *M}(d,r,2)$
and $T^{\rm sym}_{\rm *M}={\mathfrak m}_{\rm *M}/{\mathfrak m}_{\rm *M}^2$.

The graded dual ${\mathfrak m}_{\rm \star M}^\#$ can be identified with 
$$(P_{\rm \star M}^{\rm sym})^\#/P_{\rm *M}^{\rm sym}(0,0)\cong \bigoplus_{d=1}^\infty\bigoplus_r P_{\rm \star M}^{\rm sym}(d,r)^\vee.$$
So ${\mathfrak m}_{\rm PM}^\#\otimes_\Z\Q$ will be identified with the ideal $(u_0,u_1,\dots)$ of $\Q\langle u_0,u_1,\dots\rangle$
and ${\mathfrak m}_{\rm M}^\#\otimes_\Z\Q$ will be identified with the ideal $(u_0,u_1)$ of $\Q\langle u_0,u_1\rangle$.
The graded dual $(T_{\rm PM}^{\rm sym})^\#\otimes_\Z\Q$ is a subalgebra (without 1) of the ideal
$(u_0,u_1,\dots)$, and $(T_{\rm PM}^{\rm sym})^\#\otimes_\Z\Q$ is a subalgebra of $(u_0,u_1)$.
\begin{lemma}\ \label{lem:Lie}

\begin{enumerate}
\letters
\item $u_0,u_1\in (T_{\rm M}^{\rm sym})^\vee\otimes_\Z\Q$, and $u_i\in (T_{\rm PM}^{\rm sym})^\vee\otimes_\Z\Q$ for all $i$;
\item If $f,g\in (T_{\rm (P)M}^{\rm sym})^\vee\otimes_\Z\Q$, then $[f,g]=fg-gf\in (T_{\rm (P)M}^{\rm sym})^\vee\otimes_\Z\Q$.
\end{enumerate}
\end{lemma}
\begin{proof}
Part (a) is clear. Suppose that $f,g\in (T_{\rm (P)M}^{\rm sym})^\vee$. Suppose that $a,b\in {\mathfrak m}_{\rm PM}$.
We can write $\Delta(a)=a\otimes 1+1\otimes a+a'$ 
and $\Delta(b)=b\otimes 1+1\otimes b+b'$
where
$a',b'\in {\mathfrak m}_{\rm PM}\otimes {\mathfrak m}_{\rm PM}$.
Note that $a'(b\otimes 1), a'(1\otimes b),a'b',b'(a\otimes 1),b'(1\otimes a)$
lie in ${\mathfrak m}_{\rm PM}^2\otimes {\mathfrak m}_{\rm PM}$
or ${\mathfrak m}_{\rm PM}\otimes {\mathfrak m}_{\rm PM}^2$.
It follows that
\begin{multline*}
fg(ab)=(f\otimes g)((a\otimes 1+1\otimes a)(b\otimes 1+1\otimes b))=
\\=
f\otimes g(ab\otimes 1+a\otimes b+b\otimes a+1\otimes ab)=
f(a)g(b)+f(b)g(a).
\end{multline*}
Similarly $gf(ab)=f(a)g(b)+f(b)g(a)$. We conclude that $[f,g](ab)=0$.
\end{proof}
\begin{prooffreecd}
From Lemma~\ref{lem:Lie} follows that $(T_{\rm PM}^{\rm sym})^\#\otimes_\Z\Q$
contains the free Lie algebra $Q\{u_0,u_1,u_2,\dots\}$ generated
by $u_0,u_1,\dots$, and $(T_{\rm M}^{\rm sym})^\#\otimes_\Z\Q$
contains $\Q\{u_0,u_1\}$.
By the Poincar\'e-Birkhoff-Witt theorem, the graded Hilbert series
of $(P_{\rm PM}^{\rm sym})^\#\otimes_\Z\Q\cong \Q\langle u_0,u_1,\dots\rangle$ is equal to the graded Hilbert series
of the symmetric algebra on $\Q\{u_0,u_1,\dots\}$.
On the other hand, the Hilbert series of $P_{\rm PM}^{\rm sym}\otimes \Z\Q$ is
equal to the Hilbert series on the symmetric algebra on $T_{\rm PM}^{\rm sym}\otimes_\Z\Q$.
So $(T_{\rm PM}^{\rm sym})^\#\otimes_Z\Q$ and $\Q\{u_0,u_1,\dots\}$
have the same graded Hilbert series, and must therefore be equal.
If we take the completion, we get $(T_{\rm PM}^{\rm sym})^\vee\otimes_\Z\Q=\Q\{\{u_0,u_1,\dots\}\}$.
The proof for matroids is similar and $T_{\rm M}^{\rm sym})^\vee\otimes_\Z\Q=\Q\{\{u_0,u_1\}\}$.
\end{prooffreecd}
One can choose a basis in the free Lie algebra. We will use the Lyndon basis.
A word (in some alphabet $A$ with a total ordering) is a {\em Lyndon word} if it 
is strictly smaller than any cyclic permutation of $w$ with respect to the lexicographic ordering. 
In particular, Lyndon words are aperiodic. If $\alpha\in \N$, we define $b(\alpha):=u_{\alpha}$.
If $\alpha=\alpha_1\alpha_2\cdots\alpha_d$ is a Lyndon word of length $d>1$, we define
$b(\alpha)=[b(u_\beta),b(u_{\gamma})]$ where $\gamma$ is a Lyndon word of maximal length
for which $\alpha=\beta\gamma$ and $\beta$ is a nontrivial word. 
The Lyndon basis of $\Q\{u_0,u_1\}$  (respectively
$\Q\{u_0,u_1,\dots\}$) is the set of all $b(\alpha)$ where $\alpha$ is a word in $\{0,1\}$ (respectively $\N$).
For details, see~\cite{Reu}.
Define ${\mathfrak t}^{\rm sym}_{\rm M}(d,r)$\label{not:mf tsym} (respectively ${\mathfrak t}^{\rm sym}_{\rm PM}(d,r)$) as the set of all 
Lyndon words $\alpha$ in the alphabet $\{0,1\}$ (respectively $\N$) of length $d$ with $|\alpha|=d$.
The following theorem follows.
\begin{theorem}\label{theo:Tdual}
The space $(T_{\rm (P)M}^{\rm sym})^\vee(d,r)\otimes_\Z\Q$ of $\Q$-valuative additive invariants
for (poly)matroids on $\underline{d}$ of rank $r$ has the basis given by all $b(\alpha)$
with $\alpha\in {\mathfrak p}_{\rm (P)M}^{\rm sym}(d,r)$.
\end{theorem}
\begin{example}
For $d=6$, $r=3$ we have 
$$
{\mathfrak t}_{\rm M}^{\rm sym}(6,3)=\{000111,001011,001101\}
$$f
and 
\begin{multline*}
{\mathfrak t}_{\rm PM}^{\rm sym}(6,3)=
\{000003,000012,
000021,000102,000111,\\
000201,001002,001011,001101\}. 
\end{multline*}
\end{example}
\begin{proposition}
The Hopf algebra $P_{\rm PM}^{\rm sym}\otimes_\Z\Q$ is isomorphic to the ring $\QSym$ of quasi-symmetric functions over $\Q$.
\end{proposition}
\begin{proof}
If we set $u_i=p_{i+1}$ then the associative algebra $P_{\rm PM}^{\rm sym}\otimes_\Z\Q$ 
is isomorphic to $\NSym=\Q\langle p_1,p_2,\dots\rangle$.
The ring $\NSym$ has a Hopf algebra structure with $\Delta(p_i)=p_i\otimes 1+1\otimes p_i$ (see~\cite[\S7.2]{Derksen}).
The reader may verify that 
$$
\Delta(u_i)=\sum u_i\otimes 1+1\otimes u_i.
$$
This shows that the isomorphism is a Hopf-algebra isomorphism. 
It follows that $P_{\rm PM}^{\rm sym}\otimes_\Z\Q$ is isomorphic to $\QSym$, the Hopf-dual of $\NSym$. 
\end{proof}
If we identify $P_{\rm PM}^{\rm sym}\otimes_\Z\Q$ with $\QSym$, then ${\mathcal G}$ is equal 
to $\psi^{\rm sym}_{\rm PM}$.

If a multiplicative invariant is also valuative, then there exists a group homomorphism $\widehat{f}:P^{\rm sym}_{\rm *M}\to A$
such that $f=\widehat{f}\circ \psi^{\rm sym}_{\rm *M}$. Since $\psi^{\rm sym}_{\rm *M}$ is onto, 
$\widehat{f}$ is a ring homomorphism as well. So there is a bijection
between valuative, multiplicative invariants with values in $A$, and ring homomorphisms $\widehat{f}:P^{\rm sym}_{\rm\star M}\to A$.
By Corollary~\ref{cor:Ppoly},
the ring $P^{\rm sym}_{\rm\star M}$ is a polynomial ring, so ring homomorphisms 
$P^{\rm sym}_{\rm\star M}\to A$ are in bijection with set maps to~$A$
from a set of generators $\widetilde{G}$ of~$P^{\rm sym}_{\rm\star M}$.
One such set is a lift of a basis of ${\mathfrak m}_{\rm \star M}/{\mathfrak m}_{\rm \star M}^2$.  
The next corollary follows.
\begin{corollary}
The set of valuative, multiplicative invariants 
on the set of $*$ma\-troids with values in $A$ is isomorphic to 
$\Hom_\Z({\mathfrak m}_{*M}/{\mathfrak m}_{*M}^2,A).$
\end{corollary}

\section*{Acknowledgement}

The second author is grateful to David Speyer for collaboration
which led to a proof of Theorem~\ref{theo:1.1} before he got in
touch with the first author.  

\newpage
\appendix
\section{Equivalence of the weak and strong valuative property}\label{apA}
In this section we will prove that the weak valuative property and the
strong valuative property are equivalent.

For a megamatroid polyhedron $\Pi$, let $\vertices(\Pi)$ be the vertex
set  of the polyhedron.\label{not:vertices}
\label{not:W}
Let $W_{\rm MM}(d,r)$ be the subgroup of $Z_{\rm MM}(d,r)$
generated by all $m_{\rm val}(\Pi;\Pi_1,\dots,\Pi_k)$ where $\Pi=\Pi_1\cup \cdots \cup \Pi_k$
is a megamatroid polyhedron decomposition.
\label{not:W(V)}
Define $W_{\rm MM}(d,r,V)$ as the subgroup generated by all the $m_{\rm val}(\Pi;\Pi_1,\dots,\Pi_k)$
where $\vertices(\Pi)\subseteq V$. 

A megamatroid $\rk:2^{\underline{d}}\to \Z\cup\infty$ is called {\em bounded from above} if $\rk(\underline{i})<\infty$ for $i=1,2,\dots,d$.
\label{not:W+}
The group $W_{\rm MM}^+(d,r)$
is the subgroup of $Z_{\rm MM}(d,r)$ generated by all $m_{\rm val}(\Pi;\Pi_1,\dots,\Pi_k)$
where $\Pi$ is bounded from above, and $W_{\rm MM}^+(d,r,V)$
is the subgroup of $Z_{\rm MM}(d,r)$ generated by all $m_{\rm val}(\Pi;\Pi_1,\dots,\Pi_k)$ where
$\Pi$ is bounded from above and $\vertices(\Pi)\subseteq V$.

\begin{lemma}\label{lemcones}
If $\rk$ is a megamatroid bounded from above,
then there exist megamatroids $\rk_1,\dots,\rk_k$ which are bounded from above 
and integers $a_1,\dots,a_k$
such that
$$
\langle \rk\rangle-\sum_{i=1}^k a_i\langle\rk_i\rangle\in W_{\rm MM}^+(d,r,\vertices(\Pi))
$$
and $\vertices(Q(\rk_i))$ consists of a single vertex of $\Pi:=Q(\rk)$ for all $i$.
\end{lemma}

\noindent This lemma follows from the Lawrence-Varchenko polar decomposition of~$Q(\rk)$
\cite{Lawrence2,Varchenko}.  For explicitness we give a proof.

\begin{proof}
Let $T$ be the group generated by $W_{\rm MM}^+(d,r,\vertices(\Pi))$
and all megamatroid polyhedra $\Gamma$ which are bounded from above,
and whose vertex set consists of a single element of $\vertices(\Pi)$.
We prove the lemma by induction on $\left|\vertices(Q(\rk))\right|$.
If $\left|\vertex(\Pi)\right|=1$ then the result is clear.
Otherwise, we can find vertices $v$ and $w$ of $\Pi$
such that $v-w$ is parallel to $e_i-e_j$
for some $i,j$ with $i> j$. Consider the half-line $L=\R_{\geq 0}(e_i-e_j)$
where $\R_{\geq 0}$ is the set of nonnegative real numbers. Let $\Pi+L$ be
the Minkowski sum. Let us call a facet $F$ of $\Pi$ a {\em shadow facet} if $(F+L)\cap \Pi=F$.
Suppose that $F_1,\dots,F_j$ are the shadow facets of $\Pi$.

We have a megamatroid polyhedron decomposition
$$
\Pi+L=\Pi\cup (F_1+L)\cup \cdots\cup (F_j+L).
$$
Note that $\Pi+L,F_1+L,\dots,F_j+L$ are bounded from above.
The set $\vertices(\Pi+L)$ is a proper subset of $\vertices(\Pi)$ because it cannot contain both $v$ and $w$.
Also $\vertices(F_i+L)$ is contained in $\vertices(F_i)$ for all $i$, and is therefore a proper subset of $\vertices(\Pi)$
for all shadow facets $F$.
The element
$$
\langle \rk\rangle+m_{\rm val}(\Pi+L;\Pi,F_1+L,\dots,F_j+L)
$$
is an integral combination of terms $\langle\rk'\rangle$
where $Q(\rk')$ is a face of $\Pi+L$ or a face of $F_i+L$ for some $i$.
In particular, for each such term $\langle \rk'\rangle$, the
polyhedron $Q(\rk')$ is bounded from above, 
and $\vertices(Q(\rk'))$ is a proper subset of $\vertices(Q(\rk))$.
Hence by induction 
 $$
\langle \rk\rangle+m_{\rm val}(\Pi+L;\Pi,F_1+L,\dots,F_j+L)\in T.
$$
Now it follows that  $\langle\rk\rangle\in T$.
\end{proof}

\begin{proposition}\label{propweakstrong}
Suppose that $\rk_1,\dots,\rk_k$ are megamatroids\linebreak which are bounded
from above and $a_1,\dots,a_k$ are integers such that
$$
\sum_{i=1}^k a_i[Q(\rk_i)]=0.
$$
Then we have
$$
\sum_{i=1}^k a_i\langle \rk_i \rangle\in W_{\rm MM}^+(d,r,V)
$$
where $V=\bigcup_{i=1}^k \vertices(Q(\rk_i))$.
\end{proposition}
\begin{proof}
First, assume that $Q(\rk_i)$ has only one vertex for all $i$. We prove the proposition by induction on $d$,
the case $d=1$ being clear. We will also use induction on $k$, the case $k=0$ being obvious.

For vectors $y=(y_1,\dots,y_d)$ and $z=(z_1,\dots,z_d)$, we say that $y>z$ in the lexicographic
ordering if there exists an $i$ such that $y_j=z_j$ for $j=1,2,\dots,i-1$ and $y_i>z_i$.
If $\rk$ is a megamatroid bounded from above, and $Q(\rk)$ has only one vertex $v$,
then $v$ is the largest element of $Q(\rk)$ with respect to the lexicographic ordering.

Assume $V=\{v_1,\dots,v_m\}$, where $v_1>v_2>\cdots>v_m$ in the lexicographical ordering.
Assume that $Q(\rk_1),\dots,Q(\rk_n)$ are the only megamatroids among $Q(\rk_1),\dots,Q(\rk_k)$ which
have $v_1$ as a vertex.
Because $v_1$ is largest in lexicographic ordering, $v_1$ does not lie
in any of the polyhedra $Q(\rk_{n+1}),\dots,Q(\rk_k)$. Because these polyhedra are closed,
there exists an open neighborhood $U$ of $v_1$ such that $U\cap Q(\rk_{j})=\emptyset$
for $j=n+1,\dots,k$.
If we restrict to $U$, we see that
$$
\sum_{i=1}^k a_i [Q(\rk_i)\cap U]=\sum_{i=1}^n a_i [Q(\rk_i)\cap U]=0
$$
Since $Q(\rk_1),\dots,Q(\rk_n)$ are cones with vertex $v_1$,
we have
$$
\sum_{i=1}^n a_i[Q(\rk_i)]=0.
$$
and
$$
\sum_{i=n+1}^k a_i[Q(\rk_i)]=0.
$$
If $n<k$, then by the induction on $k$, we know that
$$
\sum_{i=1}^n a_i\langle Q(\rk)_i\rangle\in W_{\rm MM}^+(d,r,V)
$$
and
$$
\sum_{i=n+1}^{k} a_i \langle Q(\rk_i)\rangle\in W_{\rm MM}^+(d,r,V),
$$
hence
$$
\sum_{i=1}^ka_i\langle Q(\rk_i)\rangle\in W_{\rm MM}^+(d,r,V).
$$

Assume that $n=k$, i.e., $Q(\rk_1),\dots,Q(\rk_k)$ all have vertex $v_1$.
After translation by $-v_1$, we may assume that $r=0$, and $v_1=0$.
Now $Q(\rk_1),\dots,Q(\rk_k)$ are all contained in the halfspace defined by $y_d\geq 0$
inside the hyperplane $y_1+\cdots+y_d=0$.

Define 
$$\rho:\{y\in \R^{d-1}\mid y_1+\cdots+y_{d-1}=-1\}\to \{y\in \R^d\mid y_1+\cdots+y_d=0\}$$
by $\rho(y_1,\dots,y_{d-1})=(y_1,\dots,y_{d-1},1)$.
Assume that $\rho^{-1}(Q(\rk_i))\neq \emptyset$ for $i=1,2,\dots,t$
and $\rho^{-1}(Q(\rk_i))=\emptyset$ for $i=t+1,\dots,k$.
For $i=1,2,\dots,t$, define  megamatroids $\rk_i':2^{\underline{d-1}}\to \Z\cup \{\infty\}$
such that $Q(\rk_i')=\rho^{-1}(Q(\rk_i))$.
We have 
$$\sum_{i=1}^t a_i[Q(\rk_i')]=
\sum_{i=1}^na_i[Q(\rk_i))]\circ \rho=0.
$$

Note that $Q(\rk_i')$ is bounded from above and
$\vertices(Q(\rk_i'))\subseteq \{-e_1,\dots,-e_{d-1}\}$ for $i=1,2,\dots,t$.
By induction on $d$ we have
\begin{equation}\label{eqxx}
\sum_{i=1}^t a_i\langle \rk_i'\rangle \in {W}_{\rm MM}^+(d-1,-1,\{-e_1,-e_2,\dots,-e_{d-1}\}).
\end{equation}
If $\Gamma$ is a megamatroid polyhedron inside $y_1+\cdots+y_{d-1}=-1$ which is bounded from above,
and $\vertex(\Gamma)\subseteq \{-e_1,\dots,-e_{d-1}\}$, then
define $\Cn(\Gamma)$
 as the closure of $\R_{\geq 0} \rho(\Gamma)$.
Note that $\Cn(\Gamma)$ is also a megamatroid polyhedron.
Define
$$\gamma:Z_{\rm MM}(d,-1,\{-e_2,\dots,-e_d\})\to Z_{\rm MM}(d,0,\{0\})
$$
by $\gamma(\langle\rk\rangle)=\langle\widehat{\rk}\rangle$, where $\widehat{\rk}$ is given by $Q(\widehat{\rk})=\Cn(Q(\rk))$.

If
$$
Q(\rk')=Q(\rk_1')\cup \cdots\cup Q(\rk_s')
$$
is a megamatroid decomposition inside $\{y\in \R^d\mid y_1+\cdots+y_{d-1}=-1\}$, then
$$
\Cn (Q(\rk'))=\Cn(Q(\rk_1'))\cup \cdots \cup \Cn(Q(\rk_s'))
$$
is also a megamatroid decomposition inside $y_1+\cdots+y_d=0$.

So  $\gamma$ maps $W^+_{\rm MM}(d,-1,\{-e_1,\dots,-e_{d-1}\})$ to
$W^+_{\rm MM}(d,0,\{0\})$.

Applying $\gamma$ to (\ref{eqxx}) we get
$$
\gamma\Big(\sum_{i=1}^t a_i\langle \rk_i'\rangle\Big)=
\sum_{i=1}^ta_i\langle\rk_i\rangle\in W^+_{\rm MM}(d,0,\{0\}).
$$
From this follows that
$\sum_{i=1}^ta_i[Q(\rk_i)]=0$. Since $\sum_{i=1}^k a_i[Q(\rk_i)]=0$, we have
that $\sum_{i=t+1}^ka_i[Q(\rk_i)]=0$.
Since $Q(\rk_i)$ is contained in the hyperplane defined by $y_d=0$ for 
$i=t+1,\ldots,k$, we can again use induction on $d$ to show that
$$
\sum_{i=t+1}^k a_i\langle \rk_i\rangle\in W^+_{\rm MM}(d,r,\{0\}).
$$
We conclude that
$$
\sum_{i=1}^k a_i\langle\rk_i\rangle=
\sum_{i=1}^{t} a_i\langle \rk_i\rangle+\sum_{i=t+1}^k a_i\langle \rk_i\rangle\in W^+_{\rm MM}(d,r,\{0\}).
$$

Assume now we are in the case where $\rk_1,\dots,\rk_k$ are arbitrary.
By Lemma \ref{lemcones}, we can find megamatroids $\rk_{i,j}$ bounded from above 
with only one vertex which is contained in the set $V$,
and integers $c_{i,j}$  such that
$$
\langle \rk_i\rangle-\sum_{j}c_{i,j}\langle\rk_{i,j}\rangle\in W_{\rm MM}^+(d,r,V)
$$
It follows that
$\sum_{i=1}^k a_ic_{i,j}[Q(\rk_{i,j})]=0$.
From the special case considered above, we obtain
$$
\sum_{i=1}^ka_i\langle\rk_i\rangle=\sum_{i=1}^k a_i\sum_jc_{i,j}\langle \rk_{i,j}\rangle\in W_{\rm MM}^+(d,r,V).
$$

\end{proof}

\begin{proofws}
It suffices to show that the kernel of $\Psi_{\rm MM}$ is contained
in $W_{\rm MM}(d,r)$.
Suppose that
$$
\Psi_{\rm MM}\Big(\sum_{i=1}^ka_i\langle \rk_i\rangle\Big)=\sum_{i=1}^k a_i[Q(\rk_i)]=0.
$$
Let $\sgn:\R\to\{-1,0,1\}$ be the signum function.
For a vector $\gamma=(\gamma_1,\dots,\gamma_d)\in \{-1,0,1\}^d$
and a megamatroid polyhedron $\Pi$, define
$$\Pi^\gamma=\{(y_1,\dots,y_d)\in\Pi\mid\forall i\,(\mbox{$\sgn y_i=\gamma_i$ or $ y_i=0$})\}.$$

For every $j$ we have a megamatroid polyhedron decomposition
\begin{equation}\label{eqdecompPij}
\Pi_j=\bigcup_{\gamma\in \{-1,1\}^d;\Pi_j^\gamma\neq \emptyset} \Pi_j^\gamma
\end{equation}
where $\gamma$ runs over $\{-1,1\}^d$. Intersections of the polyhedra $\Pi_i^\gamma$,
$\gamma\in \{-1,1\}$ are of the form $\Pi_i^\gamma$ where $\gamma\in \{-1,0,1\}^d$.
If $\Pi_i^\gamma\neq \emptyset$
define $\rk_i^\gamma$ such that $Q(\rk_i^\gamma)=\Pi_i^\gamma$.
From (\ref{eqdecompPij}) it follows that 
\begin{equation}\label{eqmval}
m_{\rm val}(\Pi_j;\{\Pi^{\gamma}_j\}_{\gamma\in \{-1,1\}^d})=
\langle\rk_j\rangle-
\sum_{\gamma\in \{-1,0,1\}^d;\Pi_i^\gamma\neq\emptyset}b^\gamma\langle \rk^\gamma_i\rangle
\in W_{\rm MM}(d,r)
\end{equation}
where the coefficients $b^\gamma\in \Z$ only depend on $\gamma$. (One
can show that $b_{\gamma}=(-1)^{z(\gamma)}$
where $z(\gamma)$ is the number of zeroes in $\gamma$, but we will not need this.)

For every $\gamma$ we have
$$
\sum_{i : \Pi^\gamma_i\neq \emptyset } a_i [\Pi_i^\gamma]=0
$$
For a given $\gamma$, we may assume after permuting the coordinates
that $\gamma_1\leq\gamma_2\leq \cdots \leq \gamma_d$.
It then follows that $\Pi_i^\gamma$ is bounded from above for all $i$.
By Proposition~\ref{propweakstrong}, we have
$$
\sum_{i} a_i\langle\rk^\gamma_i\rangle\in W_{\rm MM}(d,r)
$$
for all $\gamma$.
By (\ref{eqmval}) we get
$$
\sum_{i=1}^k a_i\langle\rk_i\rangle\in  W_{\rm MM}(d,r).
$$
\end{proofws}
\newpage

\section{Tables}\label{apB}
Below are the tables for the values of $p_{\rm PM}(d,r)$, $p_{\rm M}(d,r)$, $p_{\rm PM}^{\rm sym}(d,r)$, $p_{\rm M}^{\rm sym}(d,r)$,
$t_{\rm PM}(d,r)$, $t_{\rm M}(d,r)$, $t_{\rm PM}^{\rm sym}(d,r)$, $t_{\rm M}^{\rm sym}(d,r)$ for $d\leq 6$ and $r\leq 6$. Rows correspond to values of $d$ and columns correspond to values of $r$:
$$
\xymatrix{
\ar^r[r]\ar_d[d] & \\
 &}.
 $$
 \setlength\arraycolsep{2pt}
$$
\small
\begin{array}{ccccccccp{12pt}cccccccc}
 p_{\rm PM} & 0 & 1& 2 & 3 & 4 & 5 & 6 & &p_{\rm M} & 0 & 1 & 2 & 3 & 4 & 5 & 6  \\\cline{1-8}\cline{10-17}
 0& 1 &  & & & & & & & 0& 1 &  & & & & & \\\cline{2-8}\cline{11-17}
 1 & 1 & 1 &1 &1 &1 &1 & 1 & & 1 & 1 & 1 & & &  & & \\ \cline{2-8}\cline{11-17}
 2  & 1 & 3 & 5 &7 & 9 & 11 &13 & & 2  & 1 &3 &1 &  & & & \\ \cline{2-8}\cline{11-17}
 3 & 1 & 7 &19 & 37 & 61 & 91 & 127 & & 3 &1 &7 &7 &1 &  & & \\ \cline{2-8}\cline{11-17}
  4 & 1 & 15 & 65 & 175 &369 &671 & 1105 & &  4 & 1 & 15 &33 &15 & 1& & \\ \cline{2-8}\cline{11-17}
  5 & 1& 31& 211& 781& 2101 & 4651 & 9031 & & 5 & 1 & 31 & 131 & 131 & 31 &1  & \\\cline{2-8}\cline{11-17}
 6  &1  & 63 & 665 &3367 & 11529 & 31031  & 70993 & &6   & 1 & 63 &473 & 883  & 473 & 63 &1 \\ \\
 
 p_{\rm PM}^{\rm sym} & 0 & 1 & 2 & 3 & 4 & 5 & 6 & & p_{\rm M}^{\rm sym} & 0 & 1 & 2 & 3 & 4 & 5 & 6 \\ \cline{1-8}\cline{10-17}
 0& 1 &  & & & & & & & 0& 1  &  & & & & &\\ \cline{2-8}\cline{11-17}
 1 & 1& 1 &1 &1  &1 &1 & 1 & &  1 & 1 & 1  & &  & & &\\ \cline{2-8}\cline{11-17}
 2  &1 & 2 &3 & 4 &5 &6 &7 && 2  & 1&2  & 1 &  & & &\\ \cline{2-8}\cline{11-17}
 3 & 1 & 3& 6  & 10& 15& 21 & 28 &&3 & 1 & 3&3 & 1 & & &\\ \cline{2-8}\cline{11-17}
  4 &1 & 4 & 10  & 20 &35 & 56& 84 & & 4 &1 &4 & 6 &4 &1 & &\\ \cline{2-8}\cline{11-17}
  5 & 1 & 5 & 15 & 35 & 70 & 126 & 210 & & 5 & 1& 5 & 10 & 10 &5 &1 &\\ \cline{2-8}\cline{11-17}
 6   & 1 & 6 & 21 &56  &126 &252 & 462 & & 6  & 1& 6 &15 & 20 &15 & 6 &1 \\ \\
 t_{\rm PM} & 0 & 1 & 2 & 3 & 4 & 5 & 6  & &t_{\rm M} & 0 & 1 & 2 & 3 & 4 & 5 & 6 \\ \cline{1-8}\cline{10-17}
 0&  &  & & & & & & &0&  &  & & & & &\\ \cline{2-8}\cline{11-17}
 1 & 1 & 1 & 1 & 1 &1 &1 & 1 & &1 & 1 & 1 & &  & & &\\ \cline{2-8}\cline{11-17}
 2  &  & 1 & 2 & 3 &4 &5 & 6 & &2  & & 1 & &  & & &\\ \cline{2-8}\cline{11-17}
 3 & & 1& 4 & 9 & 16 &25 & 36 & &3 & &1 &1 &  & & &\\ \cline{2-8}\cline{11-17}
  4 &  & 1 & 8 & 27 & 64 & 125 & 216 & &4 & & 1& 4 &1 & & &\\ \cline{2-8}\cline{11-17}
  5 &  & 1 & 16 & 81 & 256 & 625 & 1296 & & 5 & &1 & 11 &11 &1 & &\\\cline{2-8}\cline{11-17}
 6   &  & 1 & 32 & 343 & 1024 & 3125 & 7776 & &6   & & 1 & 26 &66 &26 &1 & \\ \\
 t_{\rm PM}^{\rm sym} & 0 & 1 & 2 & 3 & 4 & 5 & 6 & &t_{\rm M}^{\rm sym} & 0 & 1 & 2 & 3 & 4 & 5 & 6 \\ \cline{1-8}\cline{10-17}
 0&  &  & & & & & & & 0&  &  & & & & &\\ \cline{2-8}\cline{11-17}
 1 & 1 & 1&1& 1 & 1&1 & 1 & &1 & 1 & 1 & &  & & &\\ \cline{2-8}\cline{11-17}
 2  & & 1 & 1& 2 &2 &3 & 3 & &2  & & 1 & &  & & &\\ \cline{2-8}\cline{11-17}
 3 & & 1& 2 & 3 &5 & 7& 9& &3 & & 1& 1 &  & & &\\ \cline{2-8}\cline{11-17}
  4 & & 1& 2 & 5 & 8 &14 & 20 & &4 & & 1& 1 & 1 & & &\\ \cline{2-8}\cline{11-17}
  5 & & 1 & 3 & 7 & 14 &25 &  42 & & 5 & & 1 & 2 & 2 & 1 & &\\ \cline{2-8}\cline{11-17}
 6  &  &1 & 3 & 9 & 20 & 42 & 75 & &6   & & 1  &2 &3 &2 &1 &
 \end{array}
 $$
 The tables
for $p_{\rm M},t_{\rm M},t_{\rm PM}^{\rm sym},t_{\rm M}^{\rm sym}$ can be computed recursively using the equations for
the generating functions in Theorems~\ref{theo1.4} and \ref{theo1.5}. The values for $p_{\rm PM},p_{\rm PM}^{\rm sym},p_{\rm M}^{\rm sym},t_{\rm PM}$
are trivial to compute, but are included here for comparision. The tables of $p_{\rm PM}^{\rm sym}$ and $p_{\rm M}^{\rm sym}$
are of course related to Pascal's triangle. 
The table for $p_{\rm M}^{\rm sym}$ appears in Sloane's 
{\em On-line Encyclopedia of Integer Sequences}~\cite{Sloane} 
as sequence {\bf A046802}. These numbers also appear in \cite{Singh}. We have $t_{\rm M}(d,r)=E(d-1,r-1)$ for $d,r\geq 1$, where the $E(d,r)$ are the Eulerian numbers. See the
{\em Handbook of Integer Sequences}~\cite{Sloane}, 
sequences {\bf A008292} and {\bf A123125}.
  The sequences $t_{\rm PM}^{\rm sym}$ and $t_{\rm M}^{\rm sym}$ are related to sequences {\bf A059966}, {\bf A001037}, and the sequence
  {\bf A051168} denoted by $T(h,k)$ in \cite{Sloane}.
  We have $t_{\rm PM}^{\rm sym}(d,r)=T(d-1,r)$ for $d\geq 1$ and $r\geq 0$,
  and $t_{\rm M}^{\rm sym}(d,r)=T(d-r-1,r)$ if $0\leq r<d$.

\section*{Index of selected notations}\label{notindex}
A subscript $_{\rm MM}$ or $_{\rm PM}$ or~$_{\rm M}$ on a notation refers
to the variant relating respectively to megamatroids or polymatroids or matroids.
The subscript $_{\rm *M}$ stands in for any of
$_{\rm MM}$ or $_{\rm PM}$ or~$_{\rm M}$, while
 $_{\rm (P)M}$ stands in for either of $_{\rm PM}$ or~$_{\rm M}$. 

Notations below with a dagger may have the parenthesis $(d,r)$ omitted,
in which case they refer to direct sums over all $d$ and~$r$.  
These are introduced on page~\pageref{not:no (d,r)}.

\begin{tabbing}
$W_{\rm MM}(d,r,V)$\quad \= \qquad \= \kill
$[\Pi]$ \> indicator function of a set $\Pi$, \pageref{not:[]}, \pageref{not:[Pi]}\\ 
$A_{\rm *M}(d,r)$ $\dagger$ \> the $\Z$-module generated by all
  $[R_{\rm *M}(\underline X,\underline r)]$ with
  $(\underline X,\underline r)\in\mathfrak a_{\rm *M}(d,r)$, \pageref{not:A}\\
$\mathfrak a_{\rm *M}(d,r)$ \> 
 index set, \pageref{not:mf a}\\
$B_{\rm *M}(d,r)$ \> the group generated by all $[Q(\rk)]-[Q(\rk\circ\sigma)]$, 
  \pageref{not:B}\\ 
$E$ \> the map $E(\langle \rk\rangle)=\sum_{F}\langle \rk_F\rangle$, $F$ ranging over
  faces of~$\rk$, \pageref{not:E}\\
  $\face(\Pi)$ \> the set of faces of a polyhedron $\Pi$, \pageref{not:face}\\
  $\mathcal F$ \> Billera-Jia-Reiner quasi-symmetric function, \pageref{not:F}\\
$\mathcal G$ \> polymatroid invariant, \pageref{not:G}\\
  $\ell(\underline{X})$ \> length of a chain $\underline{X}$, \pageref{not:ell}\\
  $\linhull(F)$ \>  linear hull of $F$, \pageref{not:linhull}\\
  ${\mathfrak m}_{\rm *M}$\>$\bigoplus_{d,r}P_{\rm *M}^{\rm sym}(d,r,1)$, \pageref{not:mstarM}\\
$P_{\rm *M}(d,r)$ $\dagger$ \> the $\Z$-module on indicator functions $[Q(\rk)]$, 
  \pageref{not:PPM}, \pageref{not:P}\\
  $P_{\rm *M}(d,r,e)$\>filtration of $P_{\rm *M}$, \pageref{not:Pdre}\\
  $\overline{P}_{\rm *M}(d,r,e)$ \> associated graded of $P_{\rm *M}$, \pageref{not:Pbar}\\
 $P_{\rm *M}^{\rm sym}(d,r)$ $\dagger$ \> $P/B$, the symmetrized version of~$P_{\rm *M}$, 
  \pageref{not:PPMsym}, \pageref{not:P^sym}\\
$P_{\rm *M}^{\rm sym}(d,r,e)$\>filtration of $P_{rm *M}^{\rm sym}$, \pageref{not:Pdresym}\\
$\overline{P}_{\rm *M}^{\rm sym}(d,r,e)$\>associated graded of $P_{\rm *M}^{\rm sym}$, \pageref{not:Pbarsym}\\
$p_{\rm (P)M}(d,r)$\> rank of $P_{\rm (P)M}(d,r)$, the number of independent valuative functions, \pageref{not:ppm}\\
$p_{\rm (P)M}^{\rm sym}(d,r)$\> rank of $P_{\rm (P)M}^{\rm sym}(d,r)$, the number of independent valuative invariants, \pageref{not:ppmsym}\\
$p_{\rm (P)M}(d,r,e)$\>rank of $P_{\rm (P)M}(d,r,e)$, \pageref{not:pdre}\\
$p_{\rm (P)M}^{\rm sym}(d,r,e)$\>rank of $P_{\rm (P)M}^{\rm sym}(d,r,e)$, \pageref{not:pdresym}\\  
$\mathfrak p_{\rm *M}(d,r)$ \> index set for a basis of $P_{\rm *M}(d,r)$,
\pageref{not:mf p}\\
$\mathfrak p_{\rm *M}^{\rm sym}(d,r)$\>index set for a basis of $P_{\rm *M}^{\rm sym}(d,r)$, \pageref{not:mf psym}\\
$Q(\rk)$ \> base polytope of a megamatroid, \pageref{not:Q}\\ 
$R_{\rm *M}(\underline{X},\underline{r})$ \> a (mega-, poly-)matroid whose polytope
  is a cone, \pageref{not:R}\\
  $\rk_{\Pi}$ \> rank function of a polytope $\Pi$, \pageref{not:rkPi}\\
$s_{\underline{X},\underline{r}}$ \> the indicator function for
  the chain $\underline X$ having ranks $\underline r$, 
  \pageref{not:s}\\
$s^{\rm sym}_{\underline{X},\underline{r}}$ \> the average
of~$s_{\underline{X},\underline{r}}$ under the symmetric group action, 
  \pageref{not:s^sym}\\
  $S$\> antipode ${\mathcal H}\to {\mathcal H}$ in a Hopf algebra, \pageref{not:antipode}\\
$S_{\rm *M}(d,r)$ \> set of (mega-, poly-)matroids, \pageref{not:SPM}, \pageref{not:S}\\
$S_{\rm (P)M}^{\rm sym}(d,r)$ \> isomorphism classes of (poly)matroids, \pageref{not:SPMsym}\\
${\mathcal T}$\>Tutte polynomial, \pageref{not:T}\\
$T_{\rm *M}(d,r)$ $\dagger$\> $\overline{P}_{\rm *M}(d,r,1)$, \pageref{not:TstarM}\\
$T_{\rm *M}^{\rm sym}(d,r)$ $\dagger$\> $\overline{P}_{\rm *M}^{\rm sym}(d,r,1)$, \pageref{not:TstarMsym}\\
$t_{\rm (P)M}(d,r)$\> rank of $T_{\rm (P)M}(d,r)$, number of independent additive functions,  \pageref{not:tstarM}\\
$t_{\rm (P)M}^{\rm sym}(d,r)$\> rank of $T_{\rm (P)M}^{\rm sym}(d,r)$, number of independent
additive invariants, \pageref{not:tstarMsym}\\
${\mathfrak t}_{\rm *M}(d,r)$\>index set for a basis of $T_{\rm *M}(d,r)$, \pageref{not:mf t}\\
${\mathfrak t}_{\rm (P)M}^{\rm sym}(d,r)$\>index set for a basis in $(T_{\rm (P)M}^{\rm sym}(d,r))^\vee\otimes_\Z\Q$, \pageref{not:mf tsym}\\
$\{U_{\alpha}\}$ \> basis of the ring of quasisymmetric functions, \pageref{not:Ualpha}\\
$\{u_{\alpha}\}$\>dual basis of $\{U_{\alpha}\}$, basis of $\Q$-valued invariants, \pageref{not:ualpha}\\
$V^\vee$\>dual space of $V$, \pageref{not:Vvee}\\
$V^\#$\>graded dual space of $V$, \pageref{not:Vgr}\\
$\vertices(\Pi)$ \> set of vertices of a polyhedron $\Pi$, \pageref{not:vertices}\\
$W_{\rm MM}(d,r)$ \> subgroup of $Z_{\rm MM}(d,r)$ generated by all 
  $m_{\rm val}(\Pi,\dots)$s, \pageref{not:W}\\ 
$W_{\rm MM}(d,r,V)$ \> ditto, $\Pi$ having all vertices in $V$, \pageref{not:W(V)}\\
$W_{\rm MM}(\dots)^+$ \> ditto, $\Pi$ bounded from above, \pageref{not:W+}\\
$Y_{\rm *M}(d,r)$ \>  the group generated by all 
  $\langle\rk\rangle-\langle\rk\circ\sigma\rangle$, \pageref{not:Y}\\ 
$Z_{\rm *M}(d,r)$ $\dagger$ \> the $\Z$-module on (mega-, poly-)matroids, \pageref{not:ZPM}, \pageref{not:Z}\\
$Z_{\rm *M}^{\rm sym}(d,r)$ $\dagger$ \> $Z/Y$, the symmetrized version of~$Z$, \pageref{not:ZPMsym},
  \pageref{not:Z^sym}\\
    $\Delta_{M}(d,r)$ \>hypersimplex defined by $y_1+\dots+y_d=r$, $0\leq y_i\leq 1$, \pageref{not:delta}\\
  $\Delta_{PM}(d,r)$ \> simplex defined by $y_1+\cdots+y_d=r$, $y_i\geq 0$, \pageref{not:delta}\\
  $\Delta$\> comultiplication ${\mathcal H}\to {\mathcal H}\to {\mathcal H}\otimes {\mathcal H}$ for
  a Hopf algebra ${\mathcal H}$, \pageref{not:mult}\\
  $\eta$ \> unit in a Hopf algebra, \pageref{not:unit}\\
  $\nabla$\>multiplication ${\mathcal H}\otimes {\mathcal H}\to {\mathcal H}$ in a Hopf algebra, \pageref{not:mult}\\
  $\epsilon$\> counit in a Hopf algebra, \pageref{not:counit}\\
  $\Pi^\circ$ \> relative interior of a polyhedron $\Pi$, \pageref{not:Picirc}\\
$\pi_{\rm MM}$ \> the quotient map $Z_{\rm MM}(d,r)\to Z_{\rm MM}^{\rm sym}(d,r)$,
  \pageref{not:pi}\\ 
$\rho_{\rm MM}$ \> the quotient map $P_{\rm MM}(d,r)\to P_{\rm MM}^{\rm sym}(d,r)$,
  \pageref{not:rho}\\ 
$\Psi_{\rm *M}$ \> the map $\Psi_{\rm MM}:Z_{\rm MM}(d,r)\to P_{\rm MM}(d,r)$,
  $\Psi_{\rm *M}(\langle \rk\rangle)=[Q(\rk)]$, \pageref{not:Psi_MM}\\
$\Psi_{\rm MM}^\circ$ \> the map $\Psi_{\rm MM}:Z_{\rm MM}(d,r)\to P_{\rm MM}(d,r)$,
  $\Psi_{\rm MM}(\langle \rk\rangle)=[Q(\rk)^\circ]$, \pageref{not:Psi_MM circ}\\
\end{tabbing}

\end{document}